\documentclass{article}

\newcounter{asmcount}

\usepackage{pdfsync}
\usepackage{amsmath, amsthm,amssymb,bbm,multirow,bm,mathrsfs,xfrac}
\usepackage[top=1.2in, bottom=1.2in, left=1in, right=1in]{geometry}
\usepackage{paralist}
\usepackage[hidelinks]{hyperref}

\newtheorem{theorem}{Theorem}[section]
\newtheorem{lemma}{Lemma}[section]
\newtheorem{proposition}{Proposition}[section]
\newtheorem{corollary}{Corollary}[section]

\newtheorem{assumption}{Assumption}[section]
\theoremstyle{definition}
\newtheorem{remark}{Remark}[section]
\newtheorem{definition}{Definition}[section]

\title{The Limit of Stationary Distributions of Many-Server Queues in the Halfin-Whitt Regime}

\author{Reza Aghajani\footnote{University of California San Diego, Department of Mathematics. {\tt  maghajani@ucsd.edu}} and Kavita Ramanan\footnote{Brown University, Division of Applied Mathematics. {\tt kavita\_ramanan@brown.edu} }}

\date{}

\begin{document}

\newcommand{\R}[0] {\mathbb{R}}
\newcommand{\N}[0] {\mathbb{N}}
\newcommand{\Z}[0] {\mathbb{Z}}
\newcommand{\hc}[0] {[0,\infty)}
\newcommand{\ho}[0] {(0,\infty)}

\def\Lone{\mathbb{L}^1 }
\def\Loneloc{\mathbb{L}^1_{\text{loc}}}
\def\Ltwo{\mathbb{L}^2 }
\def\Linf{\mathbb{L}^\infty}
\def\Hone{\mathbb{H}^1 }
\def\Htwo{\mathbb{H}^2 }

\newcommand{\W}[0]{\mathbb{W}}
\newcommand{\C}[0] {\mathbb{C}_{\mathbb{R}}[0,\infty) }
\newcommand{\D}[0] {\mathbb{D}[0,\infty) }
\newcommand{\Dx}[1] {\mathbb{D}_{#1}[0,\infty) }
\newcommand{\cx}[0] {\mathbb{X}}

\newcommand{\emF}[0]{\mathbb{M}_F[0,\infty)}
\newcommand{\emD}[0]{\mathbb{M}_D[0,\infty)}
\newcommand{\emX}[0]{\mathbb{M}_1(\mathcal{X})}

\newcommand{\Y}[0] {\mathbb{Y}}
\newcommand{\Vn}[0]{\mathbb{V}^{(N)}}
\newcommand{\V}[0] {\mathbb{V}}
\newcommand{\F}[0] {\mathcal{F}}
\newcommand{\K}[0] {\mathcal{K}}
\def\Tau{\mathcal{T}}

\newcommand{\Ept}[1]{\mathbb{E}\left[#1\right]}
\newcommand{\Eptil}[1]{\mathbb{E}[#1]}
\newcommand{\Prob}[1]{\mathbb{P}\left\{#1\right\}}
\newcommand{\Probil}[1]{\mathbb{P}\{#1\}}
\newcommand{\deq}[0]{\overset{(d)}{=}}
\newcommand{\dleq}[0]{\overset{d}{\leq}}
\def\law{\mathscr{L}aw}

\def\f1{\mathbf{1}}
\def\id{\textbf{Id}}
\def\half{\frac{1}{2}}
\newcommand{\indicone}[1]{\mathbbm{1}_{#1}}
\newcommand{\indic}[2]{\mathbbm{1}_{#1}\left(#2\right)}
\newcommand{\indicil}[2]{\mathbbm{1}_{#1}(#2)}

\def\filt{\mathcal{F}}
\def\filtn{\mathcal{F}^{(N)}}
\def\tfiltn{\tilde{\mathcal{F}}^{(N)}}
\def\filtnT{\mathcal{F}^{(N),\mathcal{T}}}

\def\Gen{G_E^{(N)}}
\def\gen{g_E^{(N)}}
\def\bGen{\overline G_E^{(N)}}
\def\ren{R^{(N)}}
\def\renbar{\overline R^{(N)}}
\def\tren{\tilde R^{(N)}}
\def\rens{R^{(N)}_*}

\def\xn{X^{(N)}}
\def\xnbar{\overline{X}^{(N)}}
\def\xnhat{\widehat{X}^{(N)}}
\def\xnp{X^{(N)+}}
\def\xnpbar{\overline{X}^{(N)+}}
\def\xnphat{\widehat{X}^{(N)+}}
\def\xnmhat{\widehat{X}^{(N)-}}

\def\Qnhat{\widehat L^{(N)}}

\def\zetan{\zeta^{(N)}}
\def\zetanp{\zeta^{\prime(N)}}
\def\zetanpp{\zeta^{\prime\prime(N)}}

\def\upsn{\Upsilon^{(N)}}

\def\tn{\tau^{(N)}}
\def\an{m^{(N)}}
\def\dtn{\theta^{(N)}}
\def\xin{\xi^{(N)}}

\def\lambdan{\lambda^{(N)}}
\def\lambdanbar{\overline\lambda^{(N)}}
\def\en{E^{(N)}}
\def\enbar{\overline{E}^{(N)}}
\def\enhat{\widehat{E}^{(N)}}
\def\kn{K^{(N)}}
\def\knbar{\overline{K}^{(N)}}
\def\knhat{\widehat{K}^{(N)}}
\def\dn{D^{(N)}}
\def\dnbar{\overline{D}^{(N)}}
\def\dnhat{\widehat{D}^{(N)}}
\def\delayn{u^{(N)}_0}
\def\delay{u_0}
\def\j0{j_0^{(N)}}

\def\inhat{\widehat{I}^{(N)}}
\def\I{\mathcal{I}}

\def\servicen{\mathcal{S}^{(N)}}
\def\agen{a^{(N)}}
\def\ageni{a^{(N),i}}
\def\agei{a^{i}}
\newcommand{\agenX}[1]{a^{(N),#1}}
\def\age{a}

\def\nun{\nu^{(N)}}
\def\nunbar{\overline{\nu}^{(N)}}
\def\nunhat{\widehat{\nu}^{(N)}}

\def\qn{L^{(N)}}
\def\tqn{\tilde{L}^{(N)}}

\def\zn{Z^{(N)}}
\def\znbar{\overline Z^{(N)}}
\def\znhat{\widehat Z^{(N)}}
\def\znphat{\widehat Z^{\prime(N)}}
\def\znpphat{\widehat Z^{\prime\prime(N)}}

\def\yn{Y^{(N)}}
\def\ynbar{\overline Y^{(N)}}
\def\ynhat{\widehat Y^{(N)}}
\def\vn{V^{(N)}}
\def\vnbar{\overline V^{(N)}}
\def\vnhat{\widehat V^{(N)}}
\def\tvn{\tilde V^{(N)}}
\def\Qn{Q^{(N)}}

\def\Mn{\mathcal{M}^{(N)}}
\def\Mnhat{\widehat{\mathcal{M}}^{(N)}}
\def\M{\mathcal{M}}
\def\Hn{\mathcal{H}^{(N)}}
\def\Hnbar{\overline{\mathcal{H}}^{(N)}}
\def\Hnhat{\widehat{\mathcal{H}}^{(N)}}

\def\pin{\pi^{(N)}}
\def\pinhat{\widehat \pi^{(N)}}

\def\arrivaln{\alpha^{(N)}}
\def\entryn{\beta^{(N)}}
\def\departn{\gamma^{(N)}}
\def\stationn{\kappa^{(N)}}

\def\endsup{L}
\def\endsupzero{\ell_0}
\def\supint{[0,\endsup)}

\newcommand{\twopartdef}[4]{	\left\{		\begin{array}{ll}	#1 & \mbox{if } #2 \\	#3 & \mbox{if } #4	 \end{array}	\right.}
\newcommand{\twopartdefoth}[3]{	\left\{		\begin{array}{ll}	#1 & \mbox{if } #2 \\	#3 & \mbox{otherwise. } 	 \end{array}	\right.}
\newcommand{\threepartdef}[6]{	\left\{		\begin{array}{ll}	#1 & \mbox{if } #2  \\	#3 & \mbox{if  }#4	\\#5 & \mbox{if }#6  \end{array}	\right.}

\newcommand{\bs}{\mathbf{S}}
\newcommand{\bp}{\mathbf{p}}
\newcommand{\bz}{\mathbf{0}}
\newcommand{\ba}{\bm{\alpha}}
\newcommand{\bmu}{\bm{\mu}}
\newcommand{\bnu}{\bm{\nu}}

\maketitle

\begin{abstract}
  We consider the so-called GI/GI/N queue, in which a  stream of jobs with independent and identically distributed  service times arrive as a renewal process to a common queue that is served by $N$ identical parallel servers in a first-come-first-serve manner.  We introduce a new  representation for the state of the system and, under suitable  conditions on the service and interarrival distributions,  establish convergence of the corresponding sequence of centered and scaled stationary distributions in the so-called Halfin-Whitt asymptotic regime. In particular, this resolves an open question posed by Halfin and Whitt in 1981. We also characterize the limit as the stationary distribution of an infinite-dimensional two-component Markov process that is  the unique solution to a certain stochastic partial differential equation. Previous results were essentially restricted to exponential service distributions or service distributions with finite support, for  which the corresponding limit process admits a reduced finite-dimensional Markovian representation. We develop a different approach to deal with  the general case when the Markovian representation of the limit is truly infinite-dimensional. This approach is  more broadly applicable to a larger class of networks.

\end{abstract}

\section{Introduction.}

\subsection{Background and motivation.}

A model  of a many-server system that arises in many applications is the so-called  GI/GI/N queue,  in  which jobs  with independent and identically distributed  (i.i.d.) service times arrive as a renewal process to a common queue that is processed by $N$ indistinguishable parallel servers.   When a job arrives, it is processed by a server chosen uniformly at random amongst the idle servers or, if all servers are busy, it joins the queue.   Servers process jobs from  the queue in a First-Come-First-Serve (FCFS)  manner and do not idle when there is a job waiting in the queue.  Motivated by applications in call centers, data centers and health care \cite{BroEtAl05,BerEtAl09,CayVer03}, a  particular focus in recent years has been on the characterization of steady state quantities such as the stationary  distribution of the total number of jobs in system (which includes those waiting in queue and those in service), and the stationary  probability that the  queue is non-empty or equivalently, that a job has a strictly positive wait time.  An exact computation of these quantities is in general not feasible for large systems and the goal, rather, has  been to obtain provably good approximations that are accurate in the limit as $N$, the number of servers, goes to infinity.

When the service distribution is exponential,  Halfin and Whitt \cite[Theorem 2]{HalWhi81} identified the correct asymptotic
regime that would lead to a meaningful approximation, that is, one in which the limit of the  stationary probability that the queue is non-empty lies strictly between zero and one.  Specifically, they showed that if the traffic intensity (i.e., ratio of the mean arrival rate to the mean service rate) of the system with $N$ servers has the form $1 - \beta N^{-1/2} + o(N^{-1/2})$ for some $\beta > 0$, then under natural assumptions on the initial conditions,  the sequence of centered and renormalized processes $\xnhat = (\xn - N)/\sqrt{N}$, where $\xn$ represents the total number of jobs in the system,  converges weakly (on every finite time interval) to a positive recurrent one-dimensional diffusion $X$ that  has a continuous piecewise linear drift. Moreover, they also showed in \cite[Proposition 1 and Corollary 2]{HalWhi81} that  as $N$ goes to infinity, the  sequence of stationary  distributions of $\xnhat$ converges to the unique stationary  distribution of $X$.  Since the exact form of this stationary distribution  can be easily calculated, this provides a useful  explicit approximation for the steady state probability of an $N$-server queue being strictly positive for large $N$.  In particular, this approximation can be used to determine the number of servers required to achieve a certain quality of service in the system.  The asymptotic scaling for the traffic intensity described above  is commonly  referred to as the Halfin-Whitt asymptotic regime.

However, statistical analyses of many-server queueing systems arising in real-world applications have shown that service distributions are typically not exponential \cite{BroEtAl05,CheEtAl14,GuaKoz14,Kol84}, thus motivating the need to extend the steady-state approximation result mentioned above to the case of general, non-exponential service distributions. The convergence, on finite time intervals, of the sequence of diffusion-scaled process $\{\xnhat\}$  has been established for various classes of service distributions (see, e.g., \cite{PuhRei00,ManMom08,GamMom08,Ree09,PuhRee10,Ree09,KasRam13}).  In contrast, results on the corresponding steady state distribution are limited to a few special classes of distributions (see  Section \ref{sec_intro_pw} for details).    A significant recent advance is the work of Gamarnik and Goldberg \cite{GamGol13}, which uses general bounds for the FCFS GI/GI/N queue to show that when the service distribution $G$ has a finite $(2+\epsilon)$ moment and satisfies other  technical conditions (see assumption $T_0$ therein),   the sequence of diffusion-scaled stationary  queue lengths is tight. One of the difficulties  in going beyond tightness and establishing convergence of the sequence of stationary distributions arises from the fact that prior to this work, in  the general non-exponential case, no candidate limit had been identified.   In particular, a natural candidate for the limit is the stationary distribution of the limit of the sequence of diffusion-scaled processes.  However, for non-exponential service distributions, this limit is typically no longer Markovian, and thus fewer tools are available to analyze or even establish existence of the stationary distribution of the limit process. More recently, an infinite-dimensional Markovian representation of the state was introduced in \cite{KasRam11,KanRam10} and a corresponding ``diffusion-scaled''  convergence result was established in \cite{KasRam13}. However, the limit process lies in the rather complicated state space $\R \times \mathbb{H}_{-2}$ (where $\mathbb{H}_{-2}$ is a certain distribution space) and is not a homogeneous Markov process on its own, although it can be made one when augmented to include a third component (see \cite[Remark 9.7]{KasRam13}). Thus, this limit process  appears  not to be easily amenable to analysis.

In this paper, we overcome these obstacles by introducing a novel representation of the state of the GI/GI/N queue,
which allows us to resolve the open problem stated in \cite{HalWhi81} for a large class of service distributions.  A detailed description of our approach is given in Section \ref{sec_intro_cont}.
We believe that our approach will be useful for the analysis of a broad class of many-server stochastic networks,
which are naturally modelled by infinite-dimensional stochastic processes.

\subsection{Main Results and Contributions.}\label{sec_intro_cont}

We introduce a different representation  $\yn_t=(\xn_t,\zn_t)$ for  the state of the $N$-server queue at time $t$, where we append to $\xn_t$ a function-valued component $\zn_t$. Roughly speaking, for $r > 0$, $\zn_t(r)$ represents the expected conditional number of jobs that entered service by time $t$ and are still in  service at time $t+r$, given the ages of jobs in service at time $t$. When the service distribution satisfies certain smoothness properties (see Assumption \ref{asm_service}), we show that $\zn$ takes values in the Hilbert space $\Hone\ho$, the space of square integrable functions on $(0,\infty)$ that have a square integrable weak derivative (see, e.g., \cite[Section 5.2.1]{evans} for the definition of a weak derivative). Although not necessarily Markovian on its own, the key feature of this representation is that the appended component $\zn$  contains just enough additional information to ensure a type of convergence  (see the discussion below)  of the scaled state descriptor  $\ynhat=(\xnhat,\znhat)$ to a tractable Markov process $Y$, which we call the \textit{diffusion model}.
Further, with  this choice of state space, it was shown in \cite[Proposition 4.18 and Theorem 3.8]{AghRam15spde} that
the diffusion model $Y$,  which takes values in a closed subspace $\Y$ of the Hilbert space $\R \times \Hone \ho$,  can  be  characterized as the unique solution to a certain stochastic partial differential equation (SPDE) in, and  has at most one stationary distribution.

Our first result,  Theorem \ref{thm_ergodic}, shows that for each $N\in\N$, the centered and renormalized state process $\ynhat$ is ergodic and has  stationary distribution $\pinhat$.   While the proof of this result is similar in flavor to the proof of ergodicity of the measure-valued representation obtained in \cite{KanRam12}, it entails some new  technicalities, including upgrading  classical stability results for the GI/GI/N queue \cite{Asm03}  to the state space $\Y$ of $\ynhat$ (see Proposition \ref{prop_Vergodic} and Appendix \ref{apx_vergodic}). Then, under an additional  finite $(3+\epsilon)$ moment assumption on the service distribution, we establish  our main result,  Theorem \ref{thm_interchange},  that the sequence  $\{\pinhat\}$ of stationary  distributions converges, as $N \rightarrow \infty$, to the unique  invariant distribution of the diffusion model $Y=(X,Z)$. In particular, this establishes the desired convergence of the stationary distributions of the sequence $\{\xnhat\}$, and identifies the limit  as the $X$-marginal of the unique invariant distribution of $Y$, thus resolving the open problem stated in \cite{HalWhi81} (and posed again in \cite{GamGol13}) for a large class of service distributions.  In addition, it also characterizes the convergence of stationary distributions of $\{ \znhat\}$, which could be useful for understanding the stationary distribution of other quantities of interest such as the workload process.

The proof of our convergence result involves three main steps.  The first step, summarized in Proposition \ref{thm_tightness},  is to establish  tightness of the sequence of stationary distributions $\{ \pinhat\}$ in $\Y$.  This constitutes one of the most technical parts of the paper, and involves an analysis of the dynamical equations governing $\znhat$,  the verification of uniform (in  $N$ and $t$) tightness criteria for several $\Hone(0,\infty)$-valued processes, and a uniform $\mathbb{L}^1$-bound on the stationary distributions of the sequence of (centered and scaled) queue length processes $\{(\xnhat)^+\}$ (see Corollary \ref{xnphat_bound}).  The latter constitutes a crucial strengthening of the  tightness result for   the stationary queue length processes proved in \cite{GamGol13}, and may be of independent interest. The second step, carried out in Section \ref{sec_conv_ic}, entails showing that any subsequential limit of $\{ \pinhat\}$ must be an  invariant distribution of  $Y$. For this we first show that for every $t \geq 0$, the finite-dimensional projections of $\ynhat_t$  converge in distribution, as $N \rightarrow \infty$, to those of the marginal $Y_t$ of the diffusion model, under certain convergence assumptions on a corresponding sequence of (augmented) initial conditions (where the augmentation is required because $\ynhat$ is not necessarily a Markov process on its own).
 The Hilbert structure of $\Hone \ho$  plays a crucial role in this proof  (see Proposition \ref{prop_subseq}).
 We then show  (in Proposition \ref{prop_yinfbar}) that  the assumptions on the initial conditions are satisfied when the sequence of laws of the initial conditions is equal to  (a converging subsequence of)  $\{\pinhat\}$.  The third and last step invokes the uniqueness of the invariant distribution of the diffusion model $Y$, which was established in \cite[Theorem 3.8]{AghRam15spde} and restated here as Proposition \ref{thm_spde}. In particular, we complement this  uniqueness result by establishing existence of the invariant distribution (see Corollary \ref{cor_exists}).

In summary,  techniques introduced in this work that are potentially useful for the study of a larger class of networks include:

\begin{enumerate}
\item
The new state representation  $\yn_t=(\xn_t,\zn_t)$ for the $N$-server queue, and a characterization of its dynamics.
It should be emphasized that the right choice of state space for the process is not obvious. Under general assumptions, the process $\znhat$ could also be viewed as lying in several other spaces (see Remark \ref{rem-statespace}). However, the space $\Hone \ho$ seems to be the most suitable choice that allows one to simultaneously carry out all the required steps under general, physically relevant assumptions on the service time distribution. Indeed, similar (though slightly more complex) representations have been shown to be useful in the analysis of the hydrodynamic limit of a seemingly unrelated load balancing model in \cite{AghRam16hydrodynamic,AghRam15pde}.

\item
Tightness characterization and estimates for $\Hone (0,\infty)$-valued random elements,
 and the proof of the uniform $\mathbb{L}^1$ bound on the queue lengths.
\item
The overall proof structure, which uses augmented initial conditions to deal with the fact that
$\ynhat$ is not Markovian on its own, and the exploitation of the Hilbert structure of $\Y$ to establish a
certain weak type of convergence of $\ynhat_t$ to $Y$. The latter result, when combined with the tightness of $\{\pinhat\}$ and the uniqueness of the invariant measure of the diffusion model $Y$, is used to establish the main result on the convergence of $\{\pinhat\}$.
\end{enumerate}

Moreover, our result provides a framework for potentially gaining further qualitative insight into and developing
numerically approximations to
the limiting stationary distribution of the queue lengths (and other quantities encoded in the state representation).
 Indeed,  since the diffusion model $Y$ is Markov (albeit infinite-dimensional), its invariant distribution can be characterized by studying the adjoint equation associated with this Markov process, much in the spirit of what has been done for finite-dimensional diffusion approximations of stochastic networks (see, e.g., \cite{HarWil87,DaiHar92,KanRam14}).  Such an investigation is relegated to future work.

\subsection{Relation to prior work.}\label{sec_intro_pw}

Diffusion-scale approximations of steady-state distributions of stochastic networks have mainly been established in settings where the diffusion limit is finite-dimensional, such as multi-class queueing networks in the conventional heavy traffic regime, where the diffusion limit is a finite-dimensional reflected Brownian motion.   In contrast, truly infinite-dimensional  limits arise in the setting of many-server queues with general service distributions in the Halfin-Whitt regime \cite{KriPuh97, DecMoy08,  PanWhi10, KasRam13, PuhRee10, ReeTal15}. In the simpler setting of infinite-server queues,  under the assumption that the hazard rate of the service distribution is an infinitely differentiable
function with all its derivatives bounded, the limit process was shown in \cite{DecMoy08, ReeTal15} to be an infinite-dimensional Ornstein-Uhlenbeck process on the space of tempered distributions (see also the sub-critical diffusion limit in \cite{KasRam13} and alternative representations in \cite{KriPuh97} under weaker conditions on the service distribution). As a consequence,  in this case, the limit process has an explicit Gaussian invariant distribution \cite{ReeTal15}.  However, the study of many-sever queues is considerably more complicated.  Limits of the scaled steady-state distributions of  many-server networks in the Halfin-Whitt regime  (i.e., over the infinite horizon) have hitherto been established for just a few classes of service distributions for which the limit admits a reduced finite-dimensional Markovian representation.  Specifically, Whitt \cite{Whi05} extended the result for exponential service distributions in \cite{HalWhi81} to the $H_2^*$ service time distribution (i.e., a mixture of an exponential distribution and a point mass at 0), the case of deterministic service times is considered in  \cite{JelManMom04}, and the latter result was generalized in \cite{GamMom08} to service distributions with finite support.

\subsection{Some common notation.}\label{sec_notation}

For $a, b \in \R$, let $a \wedge b$ and $a \vee b$ denote the minimum  and maximum of $a$ and $b$, respectively. Also,  $a^+\doteq a\vee 0$ and $a^-\doteq-(a\wedge 0)$. For a set $B$, $\indic{B}{\cdot}$ is the indicator function of the set $B$ (i.e., $\indic{B}{x} = 1$ if $x \in B$ and $\indic{B}{x} = 0$ otherwise). Moreover, with a slight abuse of notation, on every domain $V$, $\f1$ denotes the constant function equal to $1$ on $V$.

\subsubsection{Function spaces.}
\label{sec_notation1}

For  $n\in\N$ and  $V\subset\R^n$, $\mathbb{C}(V)$, $\mathbb{C}_b(V)$ and $\mathbb{C}_c(V)$ are respectively, the spaces of real-valued continuous functions on $V$, bounded continuous functions on $V$ and continuous functions with compact support on $V$.   When $V = [0,\infty)$, we will write $\mathbb{C}[0,\infty)$ for $\mathbb{C}([0,\infty))$ and analogously for the other spaces. We use $\|f\|_\infty$ to denote the supremum of $|f(s)|$, $s \in V$. For $f\in\mathbb{C}\hc$ and $T \in \hc$, $\|f\|_T$ denotes the supremum of $|f(s)|$, $s \in [0,T]$.  A function $f$ defined on $[0,\infty)$ is said to be locally bounded if  $\|f\|_T < \infty$ for every $T < \infty$.   Also, $\mathbb{C}^0\hc$ denotes the subspace of  functions $f\in\mathbb{C}\hc$ with $f(0)=0$, $\mathbb{C}^1\hc$ denotes the set of differentiable functions $f \in \mathbb{C}\hc$ whose derivative,  denoted by $f^\prime$, is continuous on $\hc$, and $\mathbb{C}^1_b\hc$ is the subset of functions in $\mathbb{C}^1\hc$ that are bounded and have bounded derivative.  Moreover, for every Polish space $\cx$, $\mathbb{C}(\hc;\cx)$ denotes  the set of continuous $\cx$-valued functions on $\hc$. Also,  $\mathbb{D}(\hc;\mathbb{X})$ is the set of $\mathbb{X}$-valued functions on $\hc$ which are right continuous and have finite left limits at every point in  $(0,\infty)$. Endowed with the Skorokhod topology, $\mathbb{D}(\hc;\mathbb{X})$  is a complete separable metric space. However, we sometimes also equip $\mathbb{D}(\hc;\mathbb{X})$ with the topology of uniform convergence on compact sets.  When $\mathbb{X} = \R$, we simply write $\mathbb{D}\hc$ for $\mathbb{D}(\hc;\R)$, and $\mathbb{D}^0 \hc$ for the subset of functions $f$ in $\mathbb{D}\hc$ with $f(0) = 0$.    Also, we use $\id$ to denote the identity function on $\hc$, $\id(t)=t$ for all $t \geq 0$.

Let $\mathbb{L}^1\ho,$ $\mathbb{L}^2\ho,$ and $\mathbb{L}^\infty\ho,$ denote, respectively, the spaces of integrable, square-integrable and essentially bounded functions on $\ho$ with their corresponding standard norms. The space $\Ltwo \ho$ is a Hilbert space with the inner product
\[
    \langle f,g\rangle_{\Ltwo}=\int_0^\infty f(x)g(x)dx.
\]
Also, $\mathbb{L}^1_{\text{loc}}\ho$  denotes the space of locally integrable functions on $\hc.$  The  space $\Hone\ho$  denotes the space of square integrable functions $f$ on $\ho$ whose weak derivative $f^\prime$ exists and is also square integrable. $\Hone\ho$ equipped with the inner product
\begin{equation}\label{H_ip}
    \langle f,g\rangle_{\Hone}=\langle f,g\rangle_{\Ltwo}+\langle f',g'\rangle_{\Ltwo},
\end{equation}
and the corresponding norm $\|f\|_{\Hone} = \left(\|f\|_{\Ltwo\ho}^2 + \|f^\prime\|_{\Ltwo\ho}^2\right)^{\frac{1}{2}},$ is a separable Banach space, and hence, a Polish space (see, e.g., \cite[Proposition 8.1 on p.\ 203]{Bre2011}). Throughout the paper, we may refer to the weak derivative of a function $f\in\Hone\ho$ as just the derivative of $f$. Also, recall that every function $f\in\Hone\ho$ is almost everywhere equal to an absolutely continuous function whose density coincides  almost everywhere with the weak derivative of $f$  \cite[Problem 5 on p.\ 290]{evans}.

\subsubsection{Measure spaces.}
\label{sec_notation2}
For every subset $V$ of $\R$ or $\R^2$, endowed with the Borel sigma-algebra, let $\mathbb{M}_{F}(V)$  be the space of finite positive  measures in $V$. For $\mu \in \mathbb {M}_F(V)$ and any bounded Borel-measurable function $f$ on $V$, we denote the integral of $f$ with respect to $\mu$  by
\[
    \mu(f)  \doteq \int_{V}  f(x) \mu (dx).
\]
Extending this notation to signed measures,  for every measure $\mu$ with representation $\mu=\mu^+-\mu^-;\mu^+,\mu^-\in \mathbb{M}_F(V)$, define $\mu(f)\doteq\mu^+(f)-\mu^-(f)$. We equip $\mathbb{M}_F (V)$ and $\mathbb{M}_{\leq 1}(V)$ with the weak topology: $\mu_n\Rightarrow \mu$ if and only if $\mu_n(f) \to \mu(f)$ for all $f\in\mathbb{C}_b(V).$  Recall that  Prohorov's metric $d_P$ on $\mathbb{M}_{F}(V)$ \cite[p.\ 72]{BillingsleyBook} induces the same topology  and $(\mathbb{M}_{F}(V),d_P)$ is a Polish space \cite[Theorem 6.8, Chap. 1]{BillingsleyBook}.  We also denote by $\mathbb{M}_{D}(V)$ the subspace of $\mathbb{M}_{F}(V),$ consisting of measures of the form $\sum_{i=1}^m\delta_{x_i}$ for some $m\in\N$ and $x_i\in V,i=1,...,m.$  Given any Polish space $\cx$ and  $\cx$-valued random elements $Z^1$ and $Z^2$,  $\law (Z^1)$ denotes the distribution of $Z^1$ on $\cx$ and $Z^1 \deq Z^2$ means $\law (Z^1) = \law (Z^2)$.

\section{Assumptions and main results.}\label{sec_main}

Throughout the paper, the superscript $N$ will be used to  refer to
quantities associated with the system with $N$ servers.

\subsection{Main assumptions.}\label{sec_asmumptions}

Let $\en$ denote the cumulative arrival process, that is, for all $t\geq0$, $\en_t$ represents the number of jobs that arrived to the system in the interval $[0,t]$.

\begin{assumption}\label{asm_arrival}
 There exists $\beta > 0$ such that for every $N\in\N$, the cumulative arrival process $\en$ satisfies $\en_t=\tilde E(\lambdan t),$ $t \geq 0$,  where
\begin{equation}\label{def_lambdan}
      \lambdan\doteq N-\beta\sqrt{N},
\end{equation}
and $\tilde E$ is a renewal process with delay given by a random variable $\tilde u_0$, and (i.i.d.) renewal times $\{\tilde u_n;n\in\N\}$,  with common cumulative distribution function (c.d.f.) $\tilde G_E$, which has mean $1$ and finite variance  $\sigma^2$. We further assume that $\tilde G_E$ is non-lattice and has full support (i.e., $\tilde G_E(x)<1$ for all $x>0$) and finite $(2+\epsilon)$ moment, for some $\epsilon>0.$
\end{assumption}

Assumption \ref{asm_arrival} implies that $\en$ is a renewal process with delay $\delayn=\tilde u_0/\lambdan$ and inter-arrival times $u^{(N)}_n=\tilde u_n/\lambdan,n\in\N,$ with c.d.f. $\Gen(x)\doteq\tilde G_E(\lambdan x),$ mean $1/\lambdan$, and finite $(2+\epsilon)$ moment. Since the service distribution has mean $1$ by Assumption \ref{asm_service}.a below, the condition \eqref{def_lambdan} captures the Halfin-Whitt asymptotic regime, where the traffic intensity scales as $N-\mathcal{O}(\sqrt{N})$.

\begin{remark}
The full support assumption imposed on the inter-arrival distribution is only used to show ergodicity of a Markovian state representation of the sub-critical GI/GI/N queue (in the proof of Proposition \ref{prop_Vergodic}). However, this is not a necessary condition for ergodicity, and our main results will continue to hold under weaker conditions on the inter-arrival distribution, as long as they guarantee  ergodicity of the state process $\{\yn_t;t\geq0\}$ introduced below.
\end{remark}

The sequence of service times $\{v_j;j\in\Z\}$ is assumed to be i.i.d. with  common c.d.f.\ $G$,  and independent of the arrival process. Define $\overline G \doteq 1-G$ to be the complementary c.d.f.  We start by imposing  smoothness conditions on  $G$ that ensure that  $\yn$ lies in a sufficiently nice state space.

\renewcommand{\theenumi}{\alph{enumi}}
\begin{assumption}\label{asm_service}
The c.d.f. $G$ of the service time distribution satisfies the following conditions:
\begin{enumerate}
    \item $G$ has a finite mean equal to one, and is continuously differentiable with derivative $g$. \label{asm_g}
    \item \label{asm_h} The  hazard rate function $h$ of the service distribution lies in $\mathbb{C}_b^1\hc$, where we recall
    \begin{equation}\label{def_h}
        h(x)\doteq \frac{g(x)}{\overline G(x)}.
    \end{equation}
    \setcounter{asmcount}{1}
    \end{enumerate}
\end{assumption}

\begin{remark} \label{remark_HH2}
Note that $g$ is the probability density function (p.d.f.)  of the service time distribution.
We list some important consequences of Assumption  \ref{asm_service}.
\begin{enumerate}
    \item  Note that $\overline G\in\Ltwo\ho$ since $\overline{G} \leq 1$ and Assumption \ref{asm_service}.\ref{asm_g} implies that $\overline G$ is integrable.
    \item   Assumption \ref{asm_service}.\ref{asm_h} implies that \[H\doteq \sup_{x\in\hc}h(x)<\infty.\] Since the hazard rate function $h$ is never integrable on its support, this implies that the support of $G$ must be $[0,\infty)$. Moreover,  since $g=\overline G h$ and $h'=g'/\overline G+g^2/\overline G^2$, Assumption \ref{asm_service}.\ref{asm_h} also implies that  $g$ lies in $\Ltwo \ho$, is differentiable with a continuous derivative $g^\prime$,  and
        \[
            H_2\doteq \sup_{x\in\hc} |h_2(x)| < \infty, \quad\text{ where }\quad\quad h_2(x)\doteq\frac{g^\prime(x)}{\overline G(x)}=h'(x)-h^2(x)\;;x\geq0.
        \]
        Furthermore, $g^\prime$  also lies in $\Ltwo\ho$, and hence, $\overline G$ and $g$ both lie in $\Hone\ho.$
\end{enumerate}
\end{remark}

For the main result, we will impose some additional smoothness and moment conditions.

\begin{assumption}\label{asm_serviceP}
The service time distribution satisfies the following conditions.
\begin{enumerate}
      \item $G$ has finite $(3+\epsilon)$ moment for some $\epsilon>0$, that is, $\overline G(x)=\mathcal{O}(x^{-(3+\epsilon)})$, as $x\to \infty.$ \label{asm_moment3}
      \item $g'$ has a bounded weak derivative $g''$ which satisfies $g''(x)=\mathcal{O}(x^{-(2+\epsilon)})$ as $x\to\infty$. \label{asm_gpp}
\end{enumerate}
\end{assumption}

\begin{remark}\label{remark_GL2}
For some intermediate results, we will only need to assume a finite $(2+\epsilon)$ moment for some $\epsilon > 0$.  Indeed, note that under this condition (and therefore, under Assumption \ref{asm_serviceP}.\ref{asm_moment3}),  $\int_\cdot^\infty \overline G(x)dx$ is integrable and, since it is bounded, also lies in $\mathbb{L}^2 \ho$.    Moreover, the stronger moment condition in Assumption \ref{asm_serviceP}.\ref{asm_moment3} ensures that the function $\int_\cdot^\infty \int_s^\infty \overline G(x)dx\;ds$ is also integrable.
\end{remark}

\begin{remark}
\label{remark-restrictive}
 It is easily verified (see Appendix \ref{apver}) that Assumptions  \ref{asm_service} and \ref{asm_serviceP} hold for a large class of distributions of interest including phase-type distributions, Gamma distributions with shape parameter $\alpha\geq 3$,  Lomax distributions (generalized Pareto distributions with location parameter $\mu=0$) with shape parameter $\alpha>3$, and the log-normal distribution, which has been empirically observed to be a good fit for service distributions arising in applications \cite[Section 4.3]{BroEtAl05}.
The tightness result in \cite{GamGol13} also required Assumption \ref{asm_arrival} and a technical condition that holds under Assumption \ref{asm_service} and a finite $(2+\epsilon)$ moment. We require a slightly stronger moment assumption and some additional smoothness of the c.d.f.
Thus, our assumptions appear to be not too restrictive, although it would be of interest to
see if the technical boundedness ad smoothness conditions can be further relaxed.
\end{remark}

\subsection{State representation and main results.}\label{sec_main_result}

For every $t\geq0$, let $\xn_t$ be the total number of jobs in the system. As mentioned in the introduction, we introduce a novel representation for the GI/GI/N queue, in which we append  to $\xn_t$ a function-valued state variable $\zn_t$.
To define this state variable,  for every $t\geq0$, let $\servicen(t)$ denote the set of indices of jobs that are  in service at time $t$, and for every job $j\in\servicen(t)$, let $\agen_j(t)$ be the age of that job at time $t$, which is defined to be the amount of time that the job has been in service up to time $t$. For  $t  \geq 0$, define
\begin{equation}\label{def_Zn_temp}
    \zn_t(r)\doteq \sum_{j\in\servicen(t)}\frac{\overline G(\agen_j(t)+r)}{\overline G(\agen_j(t))}\;;\quad\quad r\geq0.
\end{equation}
Note that for every job $j$ in service at time $t$, that is,  $j\in\servicen(t)$, the quantity ${\overline G(\agen_j(t)+r)}/{\overline G(\agen_j(t))}$ is the conditional probability that the job $j$ has not yet departed the network by  time $t+r$, given  $\agen_j(t)$, its age at time $t$.  Therefore, $\zn_t(r)$ is, roughly speaking, the conditional expected number of jobs that received service at time $t$ and did not complete service by time  $t+r$, given the ages of all jobs in service at time $t$. As shown in Theorem \ref{thm_ergodic},  under Assumption \ref{asm_service}, $\zn_t$ takes values in $\Hone \ho$ for each $t$.
We represent the state of the GI/GI/N system by the process $\{\yn_t;t\geq0\}$, defined by
\begin{equation}\label{def_yn}
  \yn_t\doteq(\xn_t,\zn_t), \quad t \geq 0.
\end{equation}
The connection between this  representation and that used in \cite{KasRam11,KasRam13}  is explained in Section \ref{sec_pre_Y}.

\begin{remark}
  For the special case of exponentially distributed service times, $\overline G(r)=e^{-r}$, in which case  $\zn_t(r)$ simplifies to
  \[
  \zn_t(r)=\sum_{j\in\servicen(t)}\frac{e^{-\agen_j(t)-r}}{e^{-\agen_j(t)}}=\sum_{j\in\servicen(t)}e^{-r} = e^{-r}\left| \servicen(t) \right| = e^{-r}(\xn(t)\wedge N),
  \]
  where the last equality uses the fact that the number of jobs in service at time $t$ is equal to $\xn(t)\wedge N$. In this case, the appended variable $\zn$ has no additional information compared to $\xn$.
\end{remark}

Next, define $\overline Y \doteq (\overline X,\overline Z)$ by setting
\begin{equation}\label{def_Ybar}
     \overline X\doteq 1, \quad\quad \overline Z(r)\doteq \int_0^\infty\overline G(x+r)dx;\;\; r\geq0.
\end{equation}
As shown in Corollary \ref{cor_ybar}, $\overline Y$ is the so-called fluid limit of the state process, which is the limit (on every finite interval) of the process $\yn/N$, as $N \rightarrow \infty$. Next,  for $F=X, Z, Y$, define  the centered and diffusion-scaled versions of $F$  as follows:
\begin{equation}\label{def_Ynhat}
    \widehat F^{(N)}\doteq\frac{F^{(N)}-N\overline F}{\sqrt{N}}.
\end{equation}

Our first main result, Theorem \ref{thm_ergodic}, identifies the state space of $\ynhat$ and the long-time behavior of $\ynhat$.  Recall (see, e.g., \cite[Lemma 3.3]{AghRam15spde}) that every function $f\in\Hone\ho$ has a (unique) representative $f^*$, that is,   $f= f^*$ a.e. on $\ho$,  such that $f^*$ is continuous on the closed interval $\hc$.  With a slight abuse of notation, for $r\in\hc$ (and in particular, for $r=0$), $f(r)$ denotes the evaluation of the continuous representative $f^*$ at $r$. Let
\begin{equation}\label{def_Yspace}
    \Y \doteq \left\{ (x,f) \in \R \times \Hone(0,\infty):  f(0) =  x \wedge 0 \right\},
\end{equation}
It is easy to see that $\Y$ is a closed subspace of $\R \times \Hone(0,\infty)$ and hence, is  a Polish space  (see \cite[Corollary 3.4]{AghRam15spde}). We equip $\Y$ with the corresponding Borel $\sigma$-algebra.

\begin{theorem}\label{thm_ergodic}
Suppose Assumption \ref{asm_service} holds and  the service distribution has a finite $(2+\epsilon)$ moment for some $\epsilon > 0$.   Then  for every $N\in\N$ and $t\geq0$,  $\ynhat_t$ is a  $\Y$-valued random element.  If, in addition,  Assumption \ref{asm_arrival} holds, then there exists a probability measure $\pinhat$ on $\Y$ such that for every initial condition $(\delayn,\xn_0,\agen_j(0);j\in\servicen(0))$,  $\widehat Y^{(N)}_t\Rightarrow\pinhat$ in $\Y$ as $t\to\infty$.
\end{theorem}

Theorem \ref{thm_ergodic} is proved in  Section \ref{sec_pre_ergodic}. The second assertion of  Theorem \ref{thm_ergodic} is  stated for the  augmented initial condition $(\delayn,\xn_0,\agen_j(0)$; $j\in\servicen(0))$, as opposed to just $\yn (0)$,  because $\yn$ is not necessarily Markov on its own, but its evolution is completely determined by the augmented initial condition. The question of determining whether (the $\xnhat$-marginal of) the sequence $\{\pinhat\}$ converges to a limit was first posed in \cite{HalWhi81}. It was shown in \cite{GamGol13} that under Assumption \ref{asm_arrival},  (a slightly weaker condition than) Assumption \ref{asm_service}.\ref{asm_g} and a finite $(2+\epsilon)$ assumption on the service distribution, that the $\xnhat$-marginal of $\{\pinhat\}$ is tight, but the question of convergence and the characterization of the limit remained open. We resolve this question under additional conditions on the service time distribution.

\begin{theorem}\label{thm_interchange}
    Suppose Assumptions \ref{asm_arrival}-\ref{asm_serviceP} hold. Then there exists a probability measure $\pi$ on $\Y$ such that $\pinhat\Rightarrow \pi$ in $\Y$ as $N\to\infty$. Furthermore, $\pi$ is the unique invariant distribution of a $\Y$-valued Feller Markov family $\{P^y;y\in\Y\}$.
\end{theorem}

Theorem \ref{thm_interchange} is proved in  Section \ref{sec_conv_proof}. Here, $P^y$ is the law of a process called the diffusion model $Y$ with initial condition $y$, which was introduced and studied extensively in \cite{AghRam15spde}. We recall its definition and properties in Section \ref{sec_spde}.

\begin{remark}
\label{rem-statespace}
  The choice of state representation is somewhat subtle.  Under our assumptions,
the process $\znhat$ could also be viewed as lying in the spaces
$\mathbb{W}^{1,1}(0,\infty)$, $\mathbb{C}[0,\infty)$ and $\mathbb{C}^1[0,\infty)$.
However, the choice $\Hone (0,\infty)$ seems to be the one that allows one to simultaneously
establish Theorems \ref{thm_ergodic} and \ref{thm_interchange}, and also  Theorems 3.7 and 3.8 of \cite{AghRam15spde}, which respectively, characterize the diffusion model and establish uniqueness of its invariant distribution.
\end{remark}

\section{Ergodicity of the state process.}\label{sec_pre}

In Section \ref{sec_pre_V}, we  describe the Markovian state representation  $\vn$
introduced in \cite{KasRam11} and establish its ergodicity. In Section \ref{sec_pre_Y} we show that $\yn$ is a  continuous functional of $\vn$, and then use this  in Section \ref{sec_pre_ergodic} to prove Theorem \ref{thm_ergodic}.

\subsection{The Markovian state descriptor $\vn$.} \label{sec_pre_V}

Define $\ren$ to be  the forward recurrence time  of the arrival process $\en$
\begin{equation} \label{ReDef}
    \ren_t\doteq \inf\{u>t: \en_u>\en_t\}-t, \quad t\geq0.
\end{equation}
Note that $\ren_t$ is the amount of time after $t$ that one must wait for  the next arrival, and
$\ren_0=\delayn$ (where $\delayn$ is as defined in Assumption \ref{asm_arrival}).

In \cite{KasRam11}, the GI/GI/N queue was represented by a three-component state descriptor $\vn=(\vn_t;t\geq0)$,
\begin{equation}\label{def_V}
    \vn_t\doteq(\ren_t,\xn_t,\nun_t),
\end{equation}
where  the second component, $\xn_t$, as defined earlier, is the number of jobs in the system at time $t\geq0$, and the third component, $\nun_t$, is a finite measure on $\hc$ that is the sum of delta masses, each at the age  (defined to be the amount of time  spent thus far in service)  of a job  in service at time $t$.  Given $N\in\N$, for every $t\geq0$, $\vn_t$ takes values in the space
\begin{equation}\label{def_Vspace}
    \Vn\doteq\{(r,x,\mu)\in\R_+\times \N\times \emD: x\wedge N =\mu(\f1) \},
\end{equation}
where $\emD$ is the measure space $\emD$  defined in Section \ref{sec_notation}.
When $\en$ is Poisson, the first component $\ren$ can be omitted.

We now express $\vn$ in terms of certain primitives of the queueing system; a more explicit construction is given in  \cite[Appendix A]{KanRam10}. Jobs are indexed by $j\in\Z$. A job $j$ arrives to the system at time $\arrivaln_j$, and if an idle server is available at that time, the job immediately enters service. Otherwise, the job joins the back of the queue and enters service at a later time $\entryn_j$. When its service is completed, the job departs from the system at time $\departn_j\doteq\entryn_j+v_j$, where $v_j$ is its service time requirement. Jobs that are initially in service are assigned non-positive indices $j\leq0$ and   other jobs, which are either initially in queue or arrived after time $0$, are assigned positive indices $j\geq1$, in the order of their arrival. Note that there are initially $\xn_0$ jobs in the system, and by the non-idling assumption, $\xn_0\wedge N$ of them are initially in service. Hence,  if the system is non-empty, $\j0\doteq-(\xn_0\wedge N)+1$ is the smallest job index at time zero.

For every job $j$, the age process $\{\agen_j(t);t\geq0\}$ can be expressed as follows: for $t\geq0$,
\begin{equation}
    \agen_j(t) = \twopartdefoth{[t-\entryn_j(t)]\vee 0}{t\leq \entryn_j+v_j, }{v_j}
\end{equation}
In other words,  the age of a job is zero before service entry, then grows linearly with rate $1$ until it is equal to  $v_j$  at the departure time, and remains constant afterwards. Also, define $\{\kn_t;t\geq0\}$ to be the cumulative service entry process, that is, $\kn_t$ is the total number of jobs that entered service during the interval $[0,t]$. Note that $\kn_t$ is also  the largest index of any  job that has entered service by  time $t$. The  measure $\nun_t$  can be expressed as
\begin{equation}\label{def_nun}
    \nun_t  = \sum_{j=\j0}^{\kn_t} \delta_{\agen_j(t)} 1_{\{\agen_j(t)<v_j\}}.
\end{equation}

Recall from Section \ref{sec_notation2} that for a function $f$ on $(0,\infty)$, $\nun_t(f)$ represents the integral of $f$ with respect to $\nun_t$, and also that $\f1$ is the function that is identically equal to $1$.  For every $t\geq0$, $\nun_t(\f1)$ is the number of jobs in service at time $t$, and therefore, the non-idling assumption implies the following relation:
\begin{equation}\label{number_in_service}
    \nun_t(\f1)=N\wedge \xn_t,
\end{equation}
or equivalently,
\begin{equation} \label{noIdlingcondition}
    N-\nun_t(\f1)=(N-\xn_t)^+.
\end{equation}

Although the results of this paper do not depend on the particular rule used to assign jobs to servers, for technical purposes, we define the station process $\{\stationn_j(t);t\geq0\}$   for every job $j$ as follows: $\stationn_j(t)$ is equal to the index $i\in\{1,...,N\}$ of the server at which the job $j$ receives/received service if it has already entered service, and is equal to $0$, otherwise. For $t\geq0$ define
\[
    \tfiltn_t\doteq \sigma\left( \ren(s),\agen_j(s),\stationn_j(s); s\in[0,t], j=-N+1,...-1,0,1,...\right),
\]
and let $\{\filtn_t;t\geq0\}$ be the associated right-continuous filtration that is complete with respect to $\mathbb{P}.$ It was shown in  \cite[Lemmas A.1 and B.1]{KanRam10} that  $\{\vn_t,\filtn_t;t\geq 0\}$ is a c\`adl\`ag Feller Markov process.

\begin{proposition}\label{prop_Vergodic}
Let Assumptions \ref{asm_arrival} and \ref{asm_service} hold. Then $\{ \vn_t,\filtn_t;t\geq 0\}$ is an ergodic Markov process.  In particular, there exists a $\Vn$-valued random element $\vn_\infty\doteq (\ren_\infty,\xn_\infty,\nun_\infty)$ such that for any initial condition $\vn_0$, $\vn_t \Rightarrow \vn_\infty$ in $\V$ as $t \rightarrow \infty$, and
when initialized at $\vn_\infty$, $\{\vn_t, t \geq 0\}$ is a stationary process.
\end{proposition}

This result is similar to existing results in the literature,  for example, ergodicity of an analogous state process in the presence of reneging follows from \cite[Theorem 7.1]{KanRam12}, and ergodicity of another Markovian representation was established in \cite[Section XII.2]{Asm03}.   However, since Proposition \ref{prop_Vergodic}  is not an immediate corollary of these results, a complete justification   is provided in   Appendix \ref{apx_vergodic}.

\subsection{Relation between $\vn$ and $\yn$.}\label{sec_pre_Y}

We now show that the representation $\yn$ defined in \eqref{def_yn} can be written as a functional of the Markovian state variable $\vn$. For brevity of notation, define the following family of operators  $\{\Phi_t,t\geq 0\}$ on the space of real-valued functions on $[0,\infty)$: for  $t\geq 0$,
\begin{equation}\label{PhiDef}
    (\Phi_tf)(x)=f(x+t)\;\frac{\overline{G}(x+t)}{\overline{G}(x)},\quad x\in[0,\infty).
\end{equation}
It is straightforward to see that $\Phi_0 f = f$, $\Phi_t$ maps the space of bounded continuous functions into itself, and the family of operators $\{\Phi_t, t \geq 0\}$ forms a semigroup:
\begin{equation}\label{semigroup}
  \Phi_t\Phi_s=\Phi_{t+s}, \quad\quad t,s\geq 0.
\end{equation}
Now, from definitions \eqref{def_Zn_temp} and \eqref{def_nun} of $\zn$ and $\nun$, respectively, we see that
 $\zn_t(\cdot)=\nun_t(\Phi_\cdot\f1)$. Thus, if we let $\mathcal{T}$ be the mapping on $\emF$ given by
\begin{equation}\label{def_Tmap}
    (\mathcal{T}\mu)(r) :=\mu(\Phi_r\f1),\quad\quad r\geq0,
\end{equation}
then  we can write
\begin{equation}\label{zn_nun_relation}
    \zn_t(\cdot)=\mathcal{T}\nun_t=\nun_t(\Phi_\cdot\f1),
\end{equation}
and hence,
\begin{equation}\label{yn_nun_relation}
    \yn_t=(\xn_t,\mathcal{T}\nun_t).
\end{equation}

Next, recall the definitions of $\overline X$ and $\overline Z$ in \eqref{def_Ybar}, and define
\begin{equation}\label{def_hPhi}
    \widehat{\Psi}(\alpha,x,\mu)\doteq\left(\frac{x-N\overline X}{\sqrt{N}}, \frac{{\mathcal{T}(\mu)-N\overline Z}}{\sqrt{N}} \right),
\qquad  (\alpha, x, \mu) \in \cup_{N \in \mathbb{N}} \Vn.
\end{equation}
In view of \eqref{yn_nun_relation} and \eqref{def_hPhi},  $\ynhat_t$ can be written as
\begin{equation}\label{ynhat_representation}
    \ynhat_t=\left(\frac{\xn_t-N\overline X}{\sqrt{N}}, \frac{{\mathcal{T}(\nun_t)-N\overline Z}}{\sqrt{N}} \right)=\widehat{\Psi}(\vn_t), \quad\quad       t\geq0.
\end{equation}

\begin{lemma}\label{lem_Tprops}
Suppose Assumption \ref{asm_service} holds.  Then, for every $\mu\in\emD$, $\mathcal{T}\mu$ lies in $\Hone\ho$ and has derivative
\begin{equation}\label{Tmap_prime}
    (\mathcal{T}\mu)'(r)=-\mu(\Phi_r h),\quad\quad r\geq0.
\end{equation}
If, in addition, the service distribution has a finite $(2+\epsilon)$ moment for some $\epsilon > 0$,      then the mapping $\mathcal{T}:\emD\to\Hone\ho$ is continuous and $\widehat{\Psi}$ is a continuous mapping from  $\cup_{N \in \mathbb{N}} \Vn \subset \R \times \N \times  \emD$ to $\R\times\Hone\ho$.
\end{lemma}

\begin{proof}
See Appendix   \ref{apx_props}.
\end{proof}

\subsection{Proof of Theorem \ref{thm_ergodic}.}\label{sec_pre_ergodic}

Fix $t \geq 0$ and $N \in \mathbb{N}$.  Since almost surely, $\nun_t$ takes values in $\emD$ by \eqref{def_nun} and $\zn_t = {\mathcal T}(\nun_t)$ by \eqref{zn_nun_relation},  the first assertion of Lemma \ref{lem_Tprops} shows that almost surely,  $\zn _t$ takes values in $\Hone \ho$.  The third assertion of Lemma \ref{lem_Tprops} and the fact that almost surely, $\vn_t$ is  $\Vn$-valued then imply that almost surely, $\ynhat_t$ takes values in $\mathbb{R} \times \Hone \ho$. Substituting $r=0$ in \eqref{zn_nun_relation}, we have $ \zn_t(0)=\nun_t(\f1)$, and hence, the  non-idling condition \eqref{number_in_service}  can be written as $\zn_t(0)=N\wedge \xn_t$. Also, since the distribution $G$ has mean equal to one, we have $\overline Z(0)=\int_0^\infty\overline G(x)dx =1.$ Therefore,
\begin{equation}\label{XZrelation}
        \znhat_t(0)=\frac{\xn_t\wedge N-N}{\sqrt{N}}=\frac{(\xn_t-N)\wedge 0}{\sqrt{N}} =\xnhat_t\wedge 0= -(\xnhat_t)^-,
\end{equation}
and therefore almost surely, $\ynhat_t$ takes values in $\Y$.

Next,  given $\vn_\infty = (R_\infty^{(N)}, \xn_\infty, \nun_\infty)$ as in Proposition \ref{prop_Vergodic}, define
\[
  \left(\xnhat_\infty,\znhat_\infty\right)\doteq \widehat{\Psi}(\vn_\infty) =\left(\frac{\xn_\infty-N\overline X}{\sqrt{N}}, \frac{\mathcal{T}(\nun_\infty)-N\overline Z}{\sqrt{N}}\right).
\]
We claim that the second assertion of Theorem \ref{thm_ergodic} holds with
\begin{equation}\label{def_pnhat}
  \pinhat\doteq\law(\xnhat_\infty,\znhat_\infty)= \law\left(\widehat{\Psi}(\vn_\infty)\right).
\end{equation}
To see this, note that by Proposition \ref{prop_Vergodic}, $\vn_t$ converges in distribution to $\vn_\infty$ as $t\to\infty$. Given the representation \eqref{ynhat_representation} of $\ynhat_t$ and the continuity of $\widehat{\Psi}$ established in Lemma \ref{lem_Tprops},  the  continuous mapping theorem then implies that $\ynhat_t \Rightarrow \pinhat$ as $t \rightarrow \infty$.

\section{Dynamics of the state process.}\label{sec_dynamics}

In Section \ref{sec_pre_fluid} we recall fluid limit results for $\vn$ that were established in \cite{KasRam11}
and state the corresponding fluid limit theorem for $\yn$.  In Section \ref{sec_pre_aux} we describe the equations governing the dynamics of $\vnhat$ and use this in Section \ref{sec_pre_ynhat} to formulate the equations governing the dynamics of $\ynhat$.  These results are used in the proof of Theorem \ref{thm_interchange}.

\subsection{A fluid limit theorem.}
\label{sec_pre_fluid}

A fluid limit theorem was established in \cite{KasRam11} for  $(\xn,\nun)$ under certain assumptions on the sequence of initial conditions.  In the Halfin-Whitt asymptotic regime, captured by \eqref{def_lambdan} due to the fact that the service distribution has unit mean, the load $\rho^{(N)}\doteq \lambda^{N}/N$ converges to $1$ as $N\to\infty$.  This is referred to as the \textit{critical case} in \cite{KasRam11}.  For $F=X,\nu,E,\lambda,D, K$, define the fluid scaling
\begin{equation}
\label{def-barnot}
  \overline F^{(N)}\doteq\frac{F^{(N)}}{N}.
\end{equation}
Also, recall from \eqref{def_Ybar} that $\overline X=1$,  define
\begin{equation}\label{def_ebar}
    \overline E \doteq\id,
\end{equation}
where $\id$ is the identity function on $\hc$, $\id(t)=t$ for all $t$, and define the measure $\overline\nu\in\emF$ as
\begin{equation}\label{def_nubar}
    \overline \nu(A)\doteq \int_A\overline G(x)dx,\quad\quad  A\in\mathcal{B}\hc.
\end{equation}

\begin{lemma}[Fluid Limit]\label{prop_fluid}
Suppose Assumptions \ref{asm_arrival} and \ref{asm_service}.\ref{asm_g} hold. Then,
\begin{enumerate}
    \item  as $N\to\infty$, $\enbar\to\overline E$ almost surely in $\mathbb{D}[0,\infty)$, and  $\Eptil{\enbar_t}\to t$ for every $t\geq0$;   \label{eLLN}
    \item if the sequence of initial conditions satisfy
        \begin{equation}\label{asm_ic}
            (\xnbar_0,\nunbar_0)\Rightarrow (\overline X,\overline \nu) \text{ in } \R\times\emF,\quad\text{ and }\quad \Eptil{\xnbar_0}\to\overline X,
        \end{equation}
        then  $(\xnbar_t, \nunbar_t)\Rightarrow(\overline X,\overline\nu)$ in $\R\times\emF$ for every $t\geq0$. \label{vLLN}
\end{enumerate}
\end{lemma}
\begin{proof}
For part \ref{eLLN}, by the definitions of $\en$ and $\tilde E$ in Assumption \ref{asm_arrival} we have
\[
    \enbar_t= \frac{\en_t}{N}=\frac{\lambdan}{N}\frac{\tilde E(\lambdan t)}{\lambdan},\quad\quad t\geq0.
\]
Since $\lambdan=N-\beta\sqrt{N}$, $\lambdan\uparrow\infty$ and $\lambdan/N\to 1$, by the functional Law of Large Numbers for renewal processes (see, e.g., \cite[Theorem 5.10]{CheYao2001}), the second term on the right-hand side of the display above converges to $\overline E=\id$ almost surely in $\D$. Also, for every $t\geq0$,
\[
    \lim_{N\to\infty}\Ept{\enbar_t}=\lim_{N\to\infty} \frac{\lambdan}{N}\Ept{\frac{\tilde E(\lambdan t)}{\lambdan}}=t,
\]
where the last equality is due to the elementary renewal theorem (see e.g. \cite[Proposition V.1.4]{Asm03}).

For the proof of part \ref{vLLN}, note that Assumption 1 in \cite{KasRam11} holds by part a. and the fact that the initial conditions satisfy \eqref{asm_ic} (note that  Assumption 1 in \cite{KasRam11} in fact requires the convergence in \eqref{asm_ic} to hold almost surely on some probability space, but by invoking the  Skorokhod representation theorem, this can  be replaced by distributional convergence since we only want to conclude distributional convergence.)
On the other hand, Assumption 2 in \cite{KasRam11} holds because $h$ is continuous by Assumption \ref{asm_service}.\ref{asm_h}. Thus, by Theorem 3.7 in \cite{KasRam11}, $(\xnbar,\nunbar)$ converges in distribution to the unique solution to the fluid equations (3.4) - (3.7) therein with initial conditions $(\overline E,\overline X,\overline \nu)$. The result then follows from the fact that $(\overline E,\overline X,\overline \nu)$ is a fixed point of the fluid equations (see Remark 3.8 in \cite{KasRam11}).
\end{proof}

\begin{corollary}\label{cor_ybar}
Suppose Assumptions \ref{asm_arrival} and \ref{asm_service}  hold, and the service distribution has a finite $(2+\epsilon)$ moment for some $\epsilon > 0$.  If the sequence of initial conditions satisfy \eqref{asm_ic}, then  $\ynbar=(\xnbar,\znbar)\Rightarrow \overline Y=(\overline X,\overline Z)$ in $\R\times\Hone\ho$, as $N\to\infty$.
\end{corollary}

\begin{proof}
Since $\xnbar \Rightarrow \overline X$ by Lemma \ref{prop_fluid} and the pair $(\overline X, \overline Z)$ is deterministic, it only remains to show that $\znbar \Rightarrow \overline Z$.  Given that $\znbar = {\mathcal T} (\nunbar)$ and $\overline Z = \mathcal{T}(\overline \nu)$, with ${\mathcal T}$ defined in \eqref{def_Tmap}, this follows from the convergence $\nunbar \Rightarrow \overline \nu$ established in Lemma \ref{prop_fluid}, the continuity of ${\mathcal T}$ established in Lemma \ref{lem_Tprops} and the continuous mapping theorem.
\end{proof}

\subsection{Dynamics of $\vn$ and auxiliary processes.} 	\label{sec_pre_aux}

We now summarize the relevant equations describing the dynamics of the Markovian state processes $\vn$ and $\vnhat$ that were established in  \cite[Section 5.1]{KasRam11} and \cite[Section 6]{KasRam13}.  Let $\{\dn_t;t\geq0\}$ be the cumulative departure process, that is, $\dn_t$ denotes the total number of departures from the system during $[0,t]$.  For any bounded continuous function $\varphi$ defined on $[0,\infty)\times\hc$, define the $\varphi$-weighted departure process as
\[
    \Qn_\varphi(t) \doteq \sum_{j=\j0}^{\kn_t} \indic{[\departn_j, \infty)}{t} \varphi(v_j,\departn_j),
\]
where recall from Section \ref{sec_pre_V}  that $\departn_j$ is the time of departure of job $j$. As shown in \cite[Corollary 5.5]{KasRam11},  the $\{\filtn_t\}$-compensator of $\Qn_{\varphi}$ is the process
\[
    A^{(N)}_\varphi(t) \doteq \int_0^t\nun_s\left(\varphi(\cdot,s)h\right)ds, \quad t\in[0,\infty),
\]
and the c\`adl\`ag  local $\{\filtn_t\}$-martingale $M_\varphi^{(N)}\doteq \Qn_\varphi- A_\varphi^{(N)}$ has  predictable quadratic variation $\langle M^{(N)}_\varphi\rangle_t = A^{(N)}_{\varphi^2}(t)$. (For definitions of quadratic variation process $\langle M \rangle$ and cross variation process $\langle M,N\rangle$ of martingales $M$ and $N$, see \cite[Definition 5.3 and Definition 5.5 of Chapter 1]{KarShr91}.) Now for every $t\in[0,\infty)$ and every Borel set $B\in\mathcal{B}[0,\infty)$, define
\begin{equation}\label{MartingaleMeasureNDef}
    \Mn_t(B)\doteq M^{(N)}_{\indicone{B}}(t).
\end{equation}
It is shown in Corollary 4.3 of \cite{KasRam13} that $\Mn=\{\Mn_t(B); t\in [0,\infty),B\in \mathcal{B}_0[0,\infty)\}$ is an orthogonal martingale measure with covariance functional
\begin{equation*}
    \langle \Mn(B),\Mn(\tilde B) \rangle_t = \int_0^t \nun_s(\indicone{B\cap \tilde B}h)ds,  \quad B, \tilde B \in {\mathcal B} (\R_+).
\end{equation*}
The stochastic integral of a continuous and bounded function $\varphi:\hc\times\hc\mapsto\R$ with respect to $\Mn$,  denoted  $\Mn(\varphi)$, is itself a martingale with predictable quadratic variation process
\begin{equation}\label{CoVar_mn}
  \langle \Mn(\varphi)\rangle_t=\int_0^t\nun_s(\varphi^2(\cdot,s)h)\;ds,\quad\quad t\geq0.
\end{equation}
For definitions and properties of martingale measures and space-time white noise, and the corresponding stochastic integrals, we refer the reader to  \cite[Chapters 1 and 2]{WalshBook}.

Next, define the family of operators $\Psi_t,t\geq 0$, taking functions on $[0,\infty)$ to functions on $[0,\infty)\times[0,\infty)$, as follows: for every $t\geq 0$,
\begin{equation}\label{PsiDef}
    (\Psi_tf)(x,s)=f(x+(t-s)^+)\;\frac{\overline{G}(x+(t-s)^+)}{\overline{G}(x)}, \quad \quad (x,s)\in[0,\infty)\times [0,\infty).
\end{equation}
In particular, $\Psi_t$ maps the space of continuous bounded functions into itself, and has the property that for every $t,s\geq0$ and function $f$ on $\hc$, $\|\Psi_tf\|_\infty \leq \|f\|_\infty$, and
\begin{equation}\label{groupproperty}
    (\Psi_s\Phi_tf)(x,u)=(\Psi_{s+t}f)(x,u),\quad (x,u)\in[0,\infty)\times[0,s].
\end{equation}
Then, for $f\in\mathbb{C}_b[0,\infty)$ and $t>0$, define the stochastic convolution integral $\Hn_t(f)$ by
\begin{align}\label{def_Hn}
    \Hn_t(f)\doteq \M^{(N)}_t(\Psi_t f).
\end{align}

We have now introduced the required notation to describe the dynamics of each $N$-server system. Using equation (6.14) and definitions (5.10) (with $\nunhat$ replaced by $\nun$) and (6.11) (with $\widehat{\mathcal{K}}^{(N)}$ and $\knhat$ replaced by $\mathcal{K}^{(N)}$ and $\kn$, respectively)  in \cite{KasRam13}, for every $t\geq 0$ and every bounded, absolutely continuous function $f$, we have
\begin{equation}\label{nuNfEQ}
    \nun_t(f)= \nun_0(\Phi_tf) - \Hn_t(f)+\int_0^t(f\overline G)(t-s)d \kn_s,
\end{equation}
where $\kn$ satisfies the mass balance equation (see (6.3) and (6.4) of \cite{KasRam13})
\begin{equation}\label{kNEQ}
    \kn_t=-\left(\xn_t-N\right)^+ +\left(\xn_0-N\right)^+   +\en_t.
\end{equation}

\subsection{Equations governing the evolution of $\ynhat$.}
\label{sec_pre_ynhat}

To express the equations governing the evolution of $\ynhat$, we introduce the  centered many-server (CMS) mapping (for the so-called critical case) from \cite[Definition 5.4]{KasRam13}.  Recall  $\mathbb{D}^0[0,\infty)$ is the space of functions $f \in \mathbb{D}[0,\infty)$ with $f(0) = 0$.

\begin{definition}(\textbf{Centered Many-Server Mapping})\label{def_cmsm}
Given $(\eta,x_0,\zeta)\in \mathbb{D}^0[0,\infty) \times \R \times \mathbb{D}[0,\infty)$ with $\zeta(0)=x_0\wedge 0$, a pair $(\kappa,x) \in \mathbb{D}^0[0,\infty)\times \mathbb{D}[0,\infty)$ is said to solve the centered many-server equations associated with $(\eta,x_0,\zeta)$ if it satisfies,  for all $t\geq0$,
\begin{align}
    \label{cms_abstract1}x(t)\wedge0 &= \zeta(t) + \kappa(t) -\int_0^t{g(t-s)\kappa(s)ds},\\
    \label{cms_abstract2}\kappa(t) &=  \eta(t)-x^+(t) +x_0^+.
\end{align}
When such a solution exists and is unique for an input $(\eta,x_0,\zeta)$, the mapping $\Lambda:\mathbb{D}^0[0,\infty) \times \R \times \mathbb{D}[0,\infty) \mapsto \mathbb{D}^0[0,\infty)\times \mathbb{D}[0,\infty)$ that takes $(\eta,x_0,\zeta)$ to the solution $(\kappa,x)$ is called the Centered Many-Server (CSM) Mapping. The set of $(\eta,x_0,\zeta)$ for which a solution exists is denoted by $\text{dom}(\Lambda)$.
\end{definition}

\begin{lemma}\label{lem_Lambda}
    The mapping $\Lambda$ is single-valued on $\text{dom}(\Lambda)$, is  continuous with respect to the topology of uniform convergence on compact sets on $\mathbb{D}[0,\infty)$ and is measurable with respect to the Skorokhod topology on $\mathbb{D}[0,\infty)$. Furthermore, for every $(\eta,x_0,\zeta)\in \mathbb{C}^0[0,\infty) \times \R \times \mathbb{C}^0[0,\infty)$ with $\zeta(0)=x_0\wedge 0$, $(\eta,x_0,\zeta)\in\text{dom}(\Lambda)$ and $\Lambda(\eta,x_0,\zeta) \in \mathbb{C}[0,\infty)\times \mathbb{C}[0,\infty)$.
\end{lemma}

\begin{proof}
 Proposition 7.3 of \cite{KasRam13} shows that $\Lambda$ is single-valued and continuous, and  the measurability claim follows from \cite[Lemma 7.4]{KasRam13}. The last property holds by \cite[Lemma 4.10]{AghRam15spde}.
\end{proof}

Define $\overline K\doteq \id$, $\overline\M=0$ and $\overline{\cal H}=0$. Recall the definition \eqref{def_Ynhat} of $\widehat{H}^{(N)}$ for $H=X,Z,Y$, and define $\widehat{H}^{(N)}$ likewise for $H=E,K,\M,\mathcal{H}$.

\begin{proposition}\label{prop_Zn}
  Suppose Assumptions \ref{asm_arrival} and \ref{asm_service} hold, the service distribution has a finite $(2+\epsilon)$ moment for some $\epsilon > 0$ and  the initial conditions satisfy \eqref{asm_ic}. Then, for every $N\in\N$,
\begin{equation}\label{xk}
    (\knhat,\xnhat)=\Lambda(\enhat, \xnhat_0,\znhat_0-\Hnhat(\f1)),
\end{equation}
almost surely, where $\Lambda$ is the CMS mapping defined in Definition \ref{def_cmsm}, and in particular,
\begin{equation}\label{khatNdef}
    \knhat_t=\enhat_t-\xnphat_t+\xnphat_0, \quad\quad t\geq0.
\end{equation}
Also, for every $N\in\N$ and $t\geq0$, $\znhat_t$ lies in $\Hone\ho$ and satisfies
\begin{equation}\label{Zhat_dK}
    \znhat_t(r) =  \znhat_0(t+r) -  \Mnhat_t(\Psi_{t+r}\f1)+\knhat_t\overline G(r)-\int_0^t\knhat_s g(t-s+r)ds,\quad r\geq0,
\end{equation}
and its derivative $(\znhat_t)'$ satisfies
\begin{equation}\label{Zphat_dK}
    (\znhat_t)'(r) =  (\znhat_0)'(t+r) +  \Mnhat_t(\Psi_{t+r}h)-\knhat_tg(r)-\int_0^t \knhat_sg'(t-s+r)ds,\quad r\geq0.
\end{equation}
\end{proposition}

\begin{proof}
Since, by \cite[Remark 5.2]{KasRam13},  Assumption 4 in \cite{KasRam13} holds under  Assumption \ref{asm_service} here,  \eqref{xk} follows from \cite[Lemma 7.2]{KasRam13}, on noting that  $\znhat_0$ is denoted by $\mathcal{J}^{\nunhat_0}(\f1)$ in \cite{KasRam13}.  Equation \eqref{khatNdef} holds by \eqref{xk} and the definition of the mapping $\Lambda.$ Next, note that  $\znhat_t$ is $\Hone \ho$-valued by Theorem \ref{thm_ergodic}.  Moreover,  given the representation \eqref{zn_nun_relation} of $\zn_t$, substituting $f=\Phi_r\f1$ in  \eqref{nuNfEQ}   and  using properties \eqref{semigroup} and \eqref{groupproperty} of $\{\Phi_s\}$ and $\{\Psi_s\}$ and \eqref{def_Hn}, we have
\[
    \zn_t(r)=\nun_t(\Phi_r\f1)= \zn_0(t+r) -\Mn_t(\Psi_{t+r}\f1)+\int_0^t\overline G(t-s+r)d\kn_s,\;\;r\geq0.
\]
Also, by definition \eqref{def_Ybar} of $\overline Z$,  recalling that $\overline K_s = s$, we have
\[
    \overline Z_t(r) =\int_r^\infty \overline G(s)ds     =\int_{t+r}^\infty\overline G(s) ds +\int_r^{t+r}\overline G(s)ds     =\overline Z(t+r)+\int_0^t\overline G(t-s+r)d\overline K_s.
\]
Substituting $\zn_t$ and $\overline Z$ from the last two displays in \eqref{def_Ynhat}, recalling the definitions of $\Mnhat$ and $\knhat$, and performing an integration by parts, \eqref{Zhat_dK} follows.

Moreover, by \eqref{zn_nun_relation} and Lemma \ref{lem_Tprops}, the function $r\mapsto\zn_t(r)$ has derivative $(\zn_t)'(r)=-\nun_t(\Phi_rh)$.  Also, since $h \in \mathbb{C}_b^1\hc$ by Assumption \ref{asm_service}.\ref{asm_h}, we
can   substitute $f=h$ in \eqref{nuNfEQ}, and again invoke \eqref{semigroup}, \eqref{groupproperty} and \eqref{def_Hn}, to obtain
\[
    (\zn_t)'(r)= (\zn_0)'(t+r) +\Mn_t(\Psi_{t+r}h)-\int_0^tg(t-s+r)d\kn_s,\;\;r\geq0.
\]
Similarly, since $\overline Z(r)=\int_r^\infty\overline G(x)dx$, $\overline Z$  has derivative $\overline Z'$, where, again using $\overline K_s = s$,
\[
    \overline Z'(r) = -\overline G(r)= -\int_r^\infty g(s)ds     =\overline Z'(t+r)-\int_0^t g(t-s+r)d\overline K_s.
\]
The equation \eqref{Zphat_dK} follows from the last two displays.
\end{proof}

Substituting $\knhat$ from \eqref{khatNdef} in \eqref{Zhat_dK} and using integration by parts, we obtain
\begin{align}\label{ZhatNEQ}
    \znhat_t(r) =&  \znhat_0(t+r)- \Mnhat_t(\Psi_{t+r}\f1) +\int_0^t \overline G (t-s+r)d\enhat_s\\
    & - \xnphat_t\;\overline G(r)- \xnphat_0 \overline{G} (t+r)  + \int_0^t{\xnphat_s g(t-s+r)ds} \notag.
\end{align}

\section{Tightness of stationary distributions.}\label{sec_tight}

Below is the main result of this section.

\begin{proposition}\label{thm_tightness}
Suppose Assumptions \ref{asm_arrival}-\ref{asm_serviceP} hold. Then, $\{\pinhat\}_{N\in\N}$ is tight in $\Y$.
\end{proposition}

Recall that by Theorem \ref{thm_ergodic}, for every initial configuration, $\ynhat_t$ converges in distribution to $\pinhat$ as $t\to\infty.$  By Remark \ref{remark_tight} below, to prove Proposition \ref{thm_tightness}, it suffices to show that for some initial condition,  $\{\ynhat_t=(\xnhat_t,\znhat_t) ;t\geq0,N\in\N\}$ is tight in $\Y$. In Section \ref{sec_ic} we describe a specific initial condition, also considered in \cite{GamGol13}, that is convenient for establishing tightness of $\{\ynhat_t;t\geq0,N\in\N\}$. In Section \ref{sec_Xbound}, we recall the bound on the length of the $N$-server queue obtained in \cite{GamGol13}, and establish a uniform $\Lone$ bound for the scaled queue length process $(\xnhat)^+$.  Our bound is stronger than the bound obtained in \cite[Theorem 1]{GamGol13}, and is required to establish the tightness of the family $\{\znhat_t;t\geq0,N\in\N\}$ in Section \ref{sec_Ztight}. The proof of Proposition \ref{thm_tightness} is concluded in Section \ref{sec_tight_proof}.

\begin{remark}\label{remark_tightness}
We recall  from \eqref{ynhat_representation} and \eqref{def_V} that our representation is related to that of
\cite{KasRam13} through the relation $\ynhat_t = \widehat{\Psi}(  R_t^{(N)}, X_t^{(N)}, \nu_t^{(N)})$, where $\widehat{\Psi}$ is a continuous map from  a subspace  of $\R \times \N \times \mathbb{M}_D \hc$ to $\R \times \Hone \hc$.
 It is worth  emphasizing that nevertheless, tightness of the sequences $\{\ynhat_t\}_{N\in\N}$
 and $\{\pinhat\}_{N\in\N}$
does not follow from results in \cite{KasRam13} on the tightness of the diffusion-scaled  version of the sequence
$\{(R_t^{(N)}, X_t^{(N)}, \nu_t^{(N)})\}$ because the latter tightness was proved with respect to a different topology.   Specifically,  in \cite{KasRam13} the tightness of $\nunhat_t$ was proved
 only in the distribution space $\mathbb{H}_{-2}\ho$,  whereas the continuity of the map $\widehat{\Psi}$ with respect to its third argument, $\nu_t^{(N)}$, holds with respect to the topology of the measure space $\mathbb{M}_D\hc$.
Furthermore, the work \cite{KasRam13} did not consider existence or tightness of corresponding stationary  distributions.
\end{remark}

\begin{remark}\label{remark_tight}
Suppose a family $\{\xi^{(N)}_t;t\geq0,N\in\N\}$ of random variables taking values in a Polish space $\cal X$ is tight, and for every $N\in\N$, $\xi^{(N)}_t\Rightarrow \xi^{(N)}_\infty$ in $\cal X$ as $t\to\infty$.   Then,  for every $\epsilon>0$, there exists a compact subset $K_\epsilon\subset\cal X$ such that $\sup_{N\in\N,\;t\geq0} \Prob{\xi^{(N)}_t\in K^c_\epsilon}<\epsilon.$
Since $K^c_\epsilon$ is open, by the portmanteau theorem for weak convergence, for every $N\in\N$ we have
\[\Prob{\xi^{(N)}_\infty\in K^c_\epsilon} \leq \limsup\limits_{t\to\infty} \Prob{\xi^{(N)}_t\in K^c_\epsilon} \leq \sup_{t\geq0}\Prob{\xi^{(N)}_t\in K^c_\epsilon},\]
and hence,
\[\sup_{N\in\N}\Prob{\xi^{(N)}_\infty\in K^c_\epsilon}\leq \sup_{N\in\N,\;t\geq0} \Prob{\xi^{(N)}_t\in K^c_\epsilon}<\epsilon.\]
Therefore, the sequence $\{\xi^{(N)}_\infty\}_{N\in\N}$ is also tight in $\cal X$.
\end{remark}

\subsection{A special initial condition.}\label{sec_ic}
Recall from Assumption \ref{asm_arrival} that the arrival process $\en$ is a renewal process with inter-arrival distribution $\Gen$ and delay $\delayn$. Define $\rens$ by
\[
    \Prob{\rens \in A} = \int_A \bGen(x)dx, \quad\quad   A\in\mathcal{B}\hc.
\]
When the delay $\delayn$, and hence the initial forward recurrence time $\ren_0$, has the same distribution as $\rens$, the arrival process $\en$ becomes a stationary renewal process (see  \cite[Theorem V.3.3]{Asm03}). Next, define $\xn_*\doteq N$. Also, let $\{a_j^*;j\in\N\}$ be a sequence of i.i.d. random variables with common distribution $\Probil{a_j^* \in A} = \int_A \overline G(x)dx$ for all $A\in\mathcal{B}\hc,$ and define
\begin{equation}\label{def_nuns}
    \nun_*\doteq \sum_{j=1}^N \delta_{a^*_j}.
\end{equation}
Finally, define  $V^{(N)}_* = (\xn_*, \nun_*)$, and
\begin{equation}\label{def_pis}
  \pin_*\doteq \law(\rens,\xn_*,\nun_*).
\end{equation}
In summary, when the initial condition $\vn_0$ is distributed as $\pin_*$,  the arrival process is stationary, the queue is empty and all jobs are in service with independent initial ages distributed according to the so-called residual distribution of $G$, which has p.d.f.\ $\overline G(x)$.

We now establish a fluid limit theorem;  recall the fluid scaling notation introduced in \eqref{def-barnot}.

\begin{lemma}
If $(\xn_0,\nun_0)=(\xn_*,\nun_*)$ for every $N$, then  $\{(\xnbar_0,\nunbar_0)\}$ satisfies \eqref{asm_ic}.
\end{lemma}

\begin{proof}
Note that $\xnbar_*=1$, and, by \eqref{def_Ybar}, $\overline{X} \equiv 1$.  Also, for every  $f \in \mathbb{C}_b[0,\infty)$, $\nunbar_*(f)=\sum_{j=1}^N f(a^*_j)/N$. Thus, almost surely,
\[
    \nunbar_*(f)\to \Ept{f(a_1^*)} = \int_0^\infty f(x)\overline G(x)dx =\overline \nu(f),
\]
where the last equality holds by \eqref{def_nubar}.   Therefore,
we have  $(\nunbar, \overline{X}^{(N)}) \to (\overline \nu, \overline{X})$ in $\emF$, almost surely, and hence, in distribution.
\end{proof}

\subsection{A uniform $\mathbb{L}^1$-bound on the stationary queue length.}\label{sec_Xbound}

We now establish a uniform $\mathbb{L}^1$-bound on  $\{(\xnhat_t)^+, N \in \mathbb{N}, t \geq 0\}$.  We  first obtain such a bound for a related family of random variables, whose terms were shown to stochastically dominate the corresponding terms of $\{(\xnhat_t)^+, N \in \mathbb{N}, t \geq 0\}$ in \cite{GamGol13}.   Define
\begin{equation}\label{def_Qnhat}
	\Qnhat\doteq\frac{1}{\sqrt{N}}\sup_{t\geq 0}\left(A^{(N)}(t)-\sum_{i=1}^{N} D_i(t)\right),
\end{equation}
where $A^{(N)}$ is a stationary renewal process with renewal time distribution $\Gen$, and $D_i; i=1,...,N,$ are i.i.d. stationary  renewal processes independent of $A^{(N)}$, with renewal distribution $G$.   Note that since $G$ has a p.d.f., $D_i(0) = 0$ for every $i$.

\begin{proposition}\label{prop-tempQ}
Suppose Assumptions  \ref{asm_arrival}-\ref{asm_service} hold, the service distribution has a finite $(2+\epsilon)$ moment for some $\epsilon  > 0$, and $\law(\vn_0) = \pin_*$. Then, there exists a constant $C_Q<\infty$ such that
\begin{equation}\label{temp_Q}
    \sup_{N\in\N}\Ept{\Qnhat}\leq C_Q.
\end{equation}
\end{proposition}

\begin{proof}
One can construct a pure (zero-delayed) renewal process $\tilde A^{(N)}$ with the same renewal distribution as $A^{(N)}$, such that $A^{(N)}(t)\leq \tilde A^{(N)}(t)+1$ for all $t\geq0$. Therefore, it is sufficient to prove the proposition for the case when $A^{(N)}$ is a pure renewal process, which we assume for the rest of the proof. Set $\tn_0=0$, and let $\tn_n$ and $\theta^{(N)}_n$, $n\geq1,$ be, respectively, the epoch times and renewal times of the process $A^{(N)}$.  Since the process $A^{(N)}-\sum_{i=1}^ND_i$ is decreasing on $[\tn_n,\tn_{n+1})$ for all $n\geq0$, $\Qnhat$ in \eqref{def_Qnhat} can be written as $\Qnhat=\qn/\sqrt{N}$, with
\begin{equation}
    \qn\doteq \sup_{k\geq0} \left( k-\sum_{i=1}^N D_i(\tn_k) \right) = \sup_{k\geq0} \left( S^{(N)}_k \right),
\end{equation}
where, since $D_i (\tn_0) = D_i(0) = 0$ for every $i$, we have
\[
    S^{(N)}_k\doteq\sum_{j=1}^k \upsn_j,\quad\quad \text{with } \quad\quad \upsn_j\doteq 1- \sum_{i=1}^N (D_i(\tn_{j})-D_i(\tn_{j-1})).
\]
By the stationarity of the process $D_i$ for each $i$, the independence of $A^{(N)}$ and $D_i$, $i=1,\dots,N$, the fact that the service distribution has unit mean, and the renewal times of $A^{(N)}$ have mean $1/\lambda^{(N)}$ by Assumption \ref{asm_arrival}, we see that $\upsn_j,j\geq1,$ are identically distributed, with common mean
\[
    \an\doteq \Ept{\upsn_1}=1-N\Ept{D_1(\theta^{(N)}_1)}=1-\frac{N}{\lambdan}= \frac{-\beta}{\sqrt{N}-\beta},
\]
where the last equality  uses \eqref{def_lambdan}.  It is shown in Appendix \ref{sec-gamgol} that the assumptions of this proposition imply those of \cite[Theorem 1]{GamGol13}. Hence, by \cite[Equation (12)]{GamGol13}\footnote{Note that in \cite[p.\ 17]{GamGol13}, the number of servers is denoted by $n$, $S^{(N)}_j$ is denoted by $W_{n,j}$, $\an$ is denoted by $a_n$, $p$ is replaced by $r$, and $C_1$ is replaced by $K_r(C_1+1)^{r/2}$.}, there exists $p>2$ and a constant $C_1<\infty$,  such that for every $k\in\N$ and $x\geq 0,$
\begin{equation}\label{gg_bound}
    \Prob{\max_{1\leq j \leq k}(S^{(N)}_j-j\an)>x}\leq C_1 k^{\frac{p}{2}}x^{-p}.
\end{equation}

Since $|\an|\sqrt{N}\to \beta$ as $N\to\infty$, \eqref{temp_Q} is equivalent to the inequality $\Eptil{|\an|\qn}\leq C$ for every $N \in \mathbb{N}$, for some  constant $C<\infty$. We then see that  for every $x\geq0$ and $N>\beta ^2$,
\begin{align*}
  \Prob{|\an|\qn>x} & = \Prob{S^{(N)}_j> \frac{x}{|\an|}\text{ for some }j\geq 0} \\
                    & = \Prob{(S^{(N)}_j-j\an)> \frac{x}{|\an|}+j|\an|\text{ for some }j\geq 0} \\
                    & \leq \sum_{k=0}^\infty\Prob{(S^{(N)}_j-j\an)> \frac{x}{|\an|}+j|\an|\text{ for some } 2^{k}\leq j < 2^{k+1} } \\
                    & \leq \sum_{k=0}^\infty\Prob{\max_{0\leq j \leq 2^{k+1}}(S^{(N)}_j-j\an)> \frac{x}{|\an|}+2^k|\an|} \\
                    & \leq C_1\sum_{k=0}^\infty  (2^{k+1})^{\frac{p}{2}}\left(\frac{x}{|\an|}+2^k|\an|\right)^{-p}\\
                    & = C_2\sum_{k=0}^\infty  \left(\frac{x}{|\an|\sqrt{2^k}}+|\an|\sqrt{2^k}\right)^{-p}, \\
\end{align*}
with $C_2\doteq \sqrt{2^p}C_1$, where the last inequality follows from \eqref{gg_bound}. Now, since for every non-negative random variable $W$, $\Ept{W}\leq 1+ \int_1^\infty \Prob{W>x}dx$, we have
\[
    \Ept{|\an|\qn}\leq 1+C_2\sum_{k=0}^\infty\int_1^\infty\left(\frac{x}{|\an|\sqrt{2^k}}+|\an|\sqrt{2^k}\right)^{-p}dx.
\]
Performing the change of variable $r=(x/|\an|\sqrt{2^k})+|\an|\sqrt{2^k},$
setting
\[
    r_0\doteq\frac{1}{|\an|\sqrt{2^k}}+|\an|\sqrt{2^k},\quad\text{ and }\quad C_3\doteq \frac{C_2}{p-1},
\]
and recalling that $p>1$, we have, for any $n\in\N$,
\begin{align*}
  \Ept{|\an|\qn}& \leq 1+C_2\sum_{k=0}^\infty|\an|\sqrt{2^k}\int_{r_0}^\infty r^{-p}dr \\
                & = 1+C_3\sum_{k=0}^\infty|\an|\sqrt{2^k}\left(\frac{1}{|\an|\sqrt{2^k}}+|\an|\sqrt{2^k}\right)^{-p+1}\\
                & \leq 1+C_3\sum_{k=0}^n|\an|\sqrt{2^k}\left(\frac{1}{|\an|\sqrt{2^k}}\right)^{-p+1} + C_3\sum_{k=n+1}^\infty|\an|\sqrt{2^k}\left(|\an|\sqrt{2^k}\right)^{-p+1} \\
                & \leq 1+C_3\sum_{k=0}^n\left(|\an|\sqrt{2^k}\right)^{p} +C_3\sum_{k=n}^\infty\left(|\an|\sqrt{2^k}\right)^{-p+2}.
\end{align*}
Now, we have $\sum_{k=0}^n\left(|\an|\sqrt{2^k}\right)^{p} = |\an|^p\sum_{k=0}^n (2^\frac{p}{2})^k \leq (|\an|\sqrt{2^n})^p$, and since $p>2$,
\begin{align*}
    \sum_{k=n}^\infty\left(|\an|\sqrt{2^k}\right)^{-p+2}& = |\an|^{-p+2}\sum_{k=n}^\infty (2^{-\frac{p-2}{2}})^k
    \leq \frac{1}{1-2^{-\frac{p-2}{2}}}(|\an|\sqrt{2^n})^{-p+2}.
\end{align*}
Combining the last three displays, for $C_4\doteq(1-2^{-(p-2)/2})^{-1}C_3$, we have
\begin{equation}\label{2}
  \Ept{|\an|\qn} \leq 1+C_4\left((|\an|\sqrt{2^n})^p+(|\an|\sqrt{2^n})^{-p+2}\right).
\end{equation}
For fixed $N$, the above inequality holds for every $n\in\N$. We now find an appropriate $n=n^*(N)$ for each $N$ to get the desirable bound. Define $z(N)\doteq \log_2(|\an|^{-2})>0$, so that $|\an|\sqrt{2^{z(N)}}=1$, and let $n^*(N)=\lceil z(N)\rceil$, so that $z(N)\leq n^*(N)\leq z(N)+1$, to obtain
\begin{eqnarray}
    \nonumber   (|\an|\sqrt{2^{n^*(N)}})^p+(|\an|\sqrt{2^{n^*(N)}})^{-p+2}
    &\leq & (|\an|\sqrt{2^{z(N)+1}})^p+(|\an|\sqrt{2^{z(N)}})^{-p+2}\\
    \label{3}
    & \leq & 2^{p/2} + 1.
\end{eqnarray}
Combining \eqref{2} and \eqref{3}, we have $\mathbb{E}[|\an|\qn] \leq 1+C_4(2^{p/2}+1)$, which completes the proof.
\end{proof}

\begin{corollary}\label{xnphat_bound}
Suppose Assumptions  \ref{asm_arrival}-\ref{asm_service} hold, the service distribution has a finite $(2+\epsilon)$ moment for some $\epsilon  > 0$, and $\vn_0$ is distributed as $\pin_*$.   Then
\begin{equation}\label{xnphat_bd_eq}
    \sup_{N\in\N,t\geq0} \; \Ept{\xnphat_t}\leq C_Q.
\end{equation}
\end{corollary}

\begin{proof}
Define
\[
    \tqn (t)\doteq \frac{1}{\sqrt{N}}\sup_{0\leq s\leq t}\left(A^{(N)}(s)-\sum_{i=1}^ND_i(s)\right),\quad\quad t\geq0.
\]
In light of Proposition \ref{prop-tempQ} and the fact that  $\tqn (t)\leq \Qnhat$ for all $t\geq0$, to prove the corollary  it suffices to show that for every $t\geq0$, $\xnphat_t\dleq\tqn (t),$  where recall that, for random variables $W_1$, $W_2$, $W_1 \dleq W_2$  denotes the property that $W_1$ is  stochastically dominated by the random variable $W_2$,  that is, $\Prob{W_1>x}\leq \Prob{W_2>x}$ for every $x\in\R$.   The  inequality  $\xnphat_t\dleq\tqn (t)$ follows from \cite[Theorem 3]{GamGol13} because, as shown in Appendix  \ref{sec-gamgol}, the assumptions of the corollary imply the conditions of that theorem, namely  assumptions H-W and $T_0$ in \cite[Section 2.1]{GamGol13}.
\end{proof}

\subsection{Tightness of $\{\znhat\}$.}\label{sec_Ztight}

In this section, we establish the following result.

\begin{proposition}\label{prop_Zuhb}
Suppose Assumptions \ref{asm_arrival}-\ref{asm_serviceP} hold, and  $\law(\vn_0) = \pin_*$. Then $\{\znhat_t;t\geq0,N\in \N\}$ is tight in $\Hone\ho$.
\end{proposition}

\begin{proof}
It clearly suffices to establish the tightness of the sequences associated with each of the terms on the right-hand side of \eqref{ZhatNEQ}. For the first three terms, this is established in Lemmas \ref{lem_Mterm}-\ref{lem_Eterm} below, and for the fourth term it follows from Lemma \ref{lem_Xterm}.  When the initial condition is distributed as $\pin_*$, $\xn_0=N$ for all $N$ and hence, the fifth term vanishes.  Finally, tightness of the sixth term follows from Lemma \ref{lem_X2term} below.
\end{proof}

The rest of this section is devoted to the details of the proof of this proposition. We first describe criteria for tightness of $\Hone\ho$-valued processes in Section \ref{sec_tight_crit}, and then verify these criteria for the family $\{\znhat_t;t\geq0,N\in\N\}$ in Sections \ref{sec_Zmgale} and \ref{sec_Zothers}.

\subsubsection{Tightness criteria for $\mathbb{H}^1(0,\infty)$-valued random elements.}\label{sec_tight_crit}

\begin{proposition}\label{prop_Htightness}
    Suppose the family of random variables $\{\zeta_\alpha(r);\alpha\in\mathcal{A},r\in\ho\}$ satisfies the following properties:
    \begin{enumerate}
        \item For every   $L \in (0,\infty)$,
              \begin{equation}\label{pre_ZboundEQ1}
                \lim_{\lambda\to \infty}\sup_{\alpha\in\mathcal{A}}\Prob{\|\zeta_\alpha\|_{\Ltwo(0,L)} >\lambda }=0,
              \end{equation}
              and for every $\delta>0$,
              \begin{equation}\label{pre_ZboundEQ4}
                \lim_{L\to\infty}\sup_{\alpha\in\mathcal{A}}\Prob{\|\zeta_\alpha\|_{\Ltwo(L,\infty)} >\delta }=0.
              \end{equation}
        \item For  $\alpha\in\mathcal{A}$, almost surely,  $\zeta_\alpha(\cdot)$ has a weak derivative $\zeta^\prime_\alpha(\cdot)$ on $\ho$ such that   for  $L \in (0,\infty)$,
              \begin{equation}\label{pre_ZboundEQ2}
                \lim_{\lambda\to \infty}\sup_{\alpha\in\mathcal{A}}\Prob{\|\zeta^\prime_\alpha\|_{\Ltwo(0,L)} >\lambda }=0,
              \end{equation}
              and for every $\delta>0$,
              \begin{equation}\label{pre_ZboundEQ5}
                \lim_{L\to\infty}\sup_{\alpha\in\mathcal{A}}\Prob{\|\zeta^\prime_\alpha\|_{\Ltwo(L,\infty)} >\delta }=0.
              \end{equation}
        \item For  $\alpha\in\mathcal{A}$, almost surely, $\zeta^\prime_\alpha(\cdot)$ has itself a weak derivative $\zeta^{\prime\prime}_\alpha(\cdot)$ on $\ho$ such that for  $L \in (0,\infty)$,
              \begin{equation}\label{pre_ZboundEQ3}
                \lim_{\lambda\to \infty}\sup_{\alpha\in\mathcal{A}}\Prob{\|\zeta^{\prime\prime}_\alpha\|_{\Ltwo(0,L)} >\lambda }=0.
              \end{equation}
    \end{enumerate}
    Then $\{\zeta_\alpha(\cdot);\alpha\in\mathcal{A} \}$  is tight in $\Hone\ho$.
\end{proposition}

We defer the proof of Proposition \ref{prop_Htightness} to Appendix \ref{apx_Wtight}.  Now, we provide a more easily verifiable sufficient condition for the criteria of Proposition \ref{prop_Htightness} to hold.

\begin{lemma}\label{lem_Htight}
Consider a family $\{r \in \ho \mapsto \zeta_\alpha(r);\alpha\in\mathcal{A}\}$ of random functions such that for every  $\alpha\in\mathcal{A}$, almost surely, the weak derivatives $\zeta_\alpha^\prime$ and $\zeta_\alpha^{\prime\prime}$ exist.
Suppose  there exist  functions $R_1,R_2\in\Lone\ho$, and a  function $R_3\in\Lone_{\text{loc}}\ho$ such that for every ${\alpha\in\mathcal{A}}$ and $r\geq0$,
\[
    \Ept{|\zeta_\alpha(r)|^2}\leq R_1(r),\quad\quad\Ept{|\zeta^\prime_\alpha(r)|^2}\leq R_2(r),\quad\quad\Ept{|\zeta_\alpha^{\prime\prime}(r)|^2}\leq R_3(r).
\]
Then, the family $\{\zeta_\alpha(\cdot);\alpha\in\mathcal{A}\}$ is tight in $\Hone\ho$.
\end{lemma}

\begin{proof}
Fix $L>0.$ Using Markov's inequality and Tonelli's theorem, for every ${\alpha\in\mathcal{A}}$ and $\lambda\geq0$, we have
\begin{align*}
    \Prob{\|\zeta_\alpha\|_{\Ltwo(0,L)}>\lambda} &\leq \frac{1}{\lambda^2}\Ept{\|\zeta_\alpha\|^2_{\Ltwo(0,L)}} = \frac{1}{\lambda^2}\int_0^L\Ept{|\zeta_\alpha(r)|^2}dr \leq  \frac{1}{\lambda^2}\int_0^L R_1(r)dr.
\end{align*}
The right-hand side does not depend on $\alpha$, and is finite because $R_1$ is locally integrable. Thus, taking the supremum over ${\alpha\in\mathcal{A}}$, and then the limit as $\lambda\to\infty$ on both sides of the above inequality, \eqref{pre_ZboundEQ1} follows.  The proofs of properties \eqref{pre_ZboundEQ2} and \eqref{pre_ZboundEQ3}  are exactly analogous.
Similarly, for $\delta>0$,
\begin{align}
    \Prob{\|\zeta_\alpha\|_{\Ltwo(L,\infty)} >\delta } &\leq \frac{1}{\delta^2}\int_L^\infty\Ept{|\zeta_\alpha(r)|^2}dr \leq \frac{1}{\delta^2}\int_L^\infty R_1(r)dr.
\end{align}
Since $R_1$ is integrable, the right-hand side converges to zero as $L\to\infty.$ Taking the supremum over ${\alpha\in\mathcal{A}}$, and then the limit $L\to\infty$ on both sides above, \eqref{pre_ZboundEQ4} follows. The above argument, with  $R_1$ and $\zeta_\alpha$ replaced by $R_2$ and $\zeta^\prime_\alpha$, respectively, shows that  the conditions of the lemma imply \eqref{pre_ZboundEQ5}.  Thus, the lemma follows from Proposition \ref{prop_Htightness}.
\end{proof}

\subsubsection{A bound on the martingale measure stochastic integrals}\label{sec_Zmgale}

In this section we deal with the component $\Mnhat_t(\Psi_{t+\cdot}\f1)$ that arises on the right-hand side of
\eqref{ZhatNEQ}.   We start by establishing an estimate on the covariance function \eqref{CoVar_mn} of $\Mn$.

\begin{lemma}\label{lem_nunbar}
    Suppose Assumptions \ref{asm_arrival}-\ref{asm_service} hold,  $G$ has a finite $(2+\varepsilon)$ moment, and
$\law(\vn_0) = \pin_*$.  Then, for every $t\geq0$ and nonnegative, absolutely continuous, bounded function $f$, we have
    \begin{equation}\label{nunbar_bound}
        \Ept{\nunbar_t(f)}\leq \int_0^\infty (f\overline G)(x)dx+ C_Q\int_0^t |(f\overline G)^\prime(s)|ds.
    \end{equation}
    In addition, if Assumption \ref{asm_serviceP}.\ref{asm_moment3} holds, there exists a bounded function $R_M$ on $\hc$, which is also integrable on $\ho$, such that for every $t\geq0$ and $N\in\N$,
    \begin{equation}\label{nunbar_int_bd}
      \Ept{\int_0^t\nunbar_{t-s}\left( \frac{\overline G^2(\cdot+s+r)}{\overline G^2(\cdot)}\frac{g(\cdot)}{\overline G(\cdot)}\right)ds} \leq R_M(r), \quad r\geq0.
    \end{equation}
\end{lemma}

\begin{proof}
    For the first claim, fix $f$ as in the statement of the lemma, and $t\geq 0$. Dividing both sides of \eqref{nuNfEQ} by $N$, using the fluid scaling notation from
\eqref{def-barnot} and  taking expectations, we have
    \begin{equation}\label{nu_temp0}
        \Ept{\nunbar_t(f)} = \Ept{\nunbar_0(\Phi_tf)}- \Ept{\Hnbar_t(f)}+ \Ept{\int_0^t(f\overline G)(t-s)d \knbar_s}.
    \end{equation}
    Hence, by the definition of $\pi_*^{(N)}$,  $\nun_0$ has the same distribution as $\nun_*=\sum_{i=1}^N\delta_{a_J^*}$, where $\{a_j^*,j\geq0\}$ are i.i.d with p.d.f. $\overline G$, and by the definition of $\Phi_t$ in \eqref{PhiDef}, we have
    \begin{equation}\label{nu_temp1}
        \Ept{\nunbar_0(\Phi_tf)}=\frac{1}{N}\Ept{\sum_{j=1}^{N}\frac{f(a^*_j+t)\overline G(a^*_j+t)}{\overline G(a^*_j)}}= \Ept{\frac{f(a^*_1+t)\overline G(a^*_1+t)}{\overline G(a^*_1)}} = \int_0^\infty (f\overline G)(x+t)dx.
    \end{equation}
    Next, by $\eqref{def_Hn}$, $\Hn$ is the stochastic integral of a deterministic function with respect to the martingale measure $\Mn$, and is thus centered. Hence,
    \begin{equation}\label{nu_temp2}
      \Eptil{\Hn_t(f)}=0.
    \end{equation}
    Finally, substituting $\kn$ from \eqref{kNEQ}, integrating by parts, and using the relation $\xn_0=N$ and the  nonnegativity of $f$, we have for each $t \geq 0,$
    \begin{align}\label{nu_temp3a}
        \int_0^t(f\overline G)(t-s)d \knbar_s& =  \int_0^t(f\overline G)(t-s)d \enbar_s -f(0)(\xnbar_t-1)^+ -\int_0^t  (\xnbar_s-1)^+(f\overline G)^\prime(t-s)ds\notag\\
        & \leq \int_0^t(f\overline G)(t-s)d \enbar_s +\int_0^t  (\xnbar_s-1)^+|(f\overline G)^\prime(t-s)|ds
    \end{align}
    When  $\vn_0$ is distributed as $\pin_*$, $E^{(N)}$ is a stationary renewal process and satisfies $\Eptil{\enbar(t)}=\lambdanbar t$. Therefore, using integration by parts twice,  Fubini's theorem, and the relations $\lambdanbar\leq 1$ and $f \geq 0$, we have
    \begin{align}\label{nu_temp3b}
        \Ept{\int_0^t(f\overline G)(t-s)d \enbar_s} & =
        f(0)\Ept{\enbar_t}+\int_0^t \Ept{\enbar_s}(f\overline G)'(t-s)ds \notag\\
        & = \lambdanbar f(0)t+\lambdanbar\int_0^t s(f\overline G)'(t-s)ds \notag\\
        &\leq \int_0^t(f\overline G)(s)ds.
        \end{align}
    Moreover, by Corollary \ref{xnphat_bound}, for every $s\geq0$, $\Eptil{(\xnbar_s-1)^+}=\Eptil{\xnphat_s}/\sqrt{N}\leq C_Q$, and hence,
    \begin{equation}\label{nu_temp4}
        \Ept{\int_0^t  (\xnbar_s-1)^+\big|(f\overline G)^\prime(t-s)\big|ds} \leq  \int_0^t  \Ept{(\xnbar_s-1)^+}\big|(f\overline G)^\prime(t-s)\big|ds \leq C_Q\int_0^t \big|(f\overline G)^\prime(s)\big|ds.
    \end{equation}
    Equation \eqref{nunbar_bound} then follows on substituting the relations \eqref{nu_temp1}-\eqref{nu_temp4} into \eqref{nu_temp0}.

    For the second claim, note that by Assumption \ref{asm_service},  recalling the constants $H$ and $H_2$ defined in Remark \ref{remark_HH2}, for every $s,r\geq0$,  the function $f_*\doteq\overline G^2(s+r+\cdot)g(\cdot)/\overline G^3(\cdot)$ is non-negative, absolutely continuous and bounded, and satisfies $|f_*\overline G(x)|\leq H \overline G(x+s+r).$ Also,
        \begin{align*}
            (f_*\overline G)'(x)= \frac{-2\overline G(x+s+r)g(x+s+r)g(x)+\overline G^2(x+s+r)g'(x)}{\overline G(x)^2}+\frac{2\overline G^2(x+s+r)g^2(x)}{\overline G^3(x)},
        \end{align*}
    which implies $|(f_*\overline G)'(x)|\leq (4H^2+H_2)\overline G(x+s+r)$. Therefore, replacing $f$ with $f_*$ and $t$ by $t-s$ in \eqref{nunbar_bound}, we have
    \begin{align*}
        \Ept{\nunbar_{t-s}\left( \frac{\overline G^2(\cdot+s+r)}{\overline G^2(\cdot)}\frac{g(\cdot)}{\overline G(\cdot)}\right)} \leq H\int_0^\infty \overline G(s+r+x)dx+C_Q(4H^2+H_2)\int_0^{t-s} \overline G(s+r+x)dx.
    \end{align*}
   Hence, by Tonelli's theorem, \eqref{nunbar_int_bd} holds with
   \[
        R_M(r)\doteq (H+C_Q(3H^2+H_2))\int_r^\infty\int_s^\infty \overline G(x)dx\;ds.
   \]
    The function $R_M$ is bounded and integrable by Assumption  \ref{asm_serviceP}.\ref{asm_moment3} (see Remark \ref{remark_GL2}).
\end{proof}

The next lemma shows absolute continuity of the martingale measure stochastic integral term.

\begin{lemma}\label{que_Mfubini}
Suppose  Assumption \ref{asm_service} holds. Then, for every $N\in\N$ and every $t,r\geq 0$,
\begin{equation}\label{MNderivative}
	   {\cal M}^{(N)}_t(\Psi_{t+r}\f1)= {\mathcal H}^{(N)}_t(\f1) -\int_0^r {\cal M}^{(N)}_t(\Psi_{t+u}h) du,
\end{equation}
and, with $h_2=g/\overline G$  as defined in Remark \ref{remark_HH2},
\begin{equation}\label{MNderivative2}
	   {\cal M}^{(N)}_t(\Psi_{t+r}h)= {\mathcal H}^{(N)}_t(h) +\int_0^r {\cal M}^{(N)}_t(\Psi_{t+u}h_2) du.
\end{equation}
\end{lemma}

\begin{proof}
Fix $t\geq0$. By Assumption \ref{asm_service}.\ref{asm_h}, the function    $(\Psi_{t+u}h)(x,s)=g(x+t+u-s)/\overline G(x)$
is bounded, continuous and satisfies
\[
    \int_0^r (\Psi_{t+u}h)(x,s)du =\frac{\overline G(x+t-s)}{\overline G(x)}-\frac{\overline G(x+t+r-s)}{\overline G(x)}= \Psi_t\f1(x,s) - \Psi_{t+r}\f1(x,s).
\]
Therefore, by the stochastic Fubini theorem for orthogonal martingale measures  \cite[Theorem 2.6]{WalshBook},
\[
    \int_0^r \Mn_t(\Psi_{t+u}h) du =\Mn_t\left(\int_0^r\Psi_{t+u}h\;du\right) =\Mn_t(\Psi_t\f1)- \Mn_t(\Psi_{t+r}\f1),
\]
and \eqref{MNderivative} follows from definition \eqref{def_Hn} of $\Hn$. Similarly, by Assumption \ref{asm_service}.\ref{asm_h},     the function $(\Psi_{t+u}h_2)(x,s)$ $=$ $g'(x+t+u-s)/\overline G(x)$ is bounded, continuous and satisfies
\[
    \int_0^r (\Psi_{t+u}h_2)(x,s)du =\frac{g(x+t+r-s)}{\overline G(x)}-\frac{g(x+t-s)}{\overline G(x)}= \Psi_{t+r}h(x,s)-\Psi_th(x,s).
\]
Applying (again) the  Fubini theorem for martingale measures and \eqref{def_Hn}, this yields \eqref{MNderivative2}.
\end{proof}

Now we can prove tightness of the martingale measure integral term.

\begin{lemma}\label{lem_Mterm}
    Suppose Assumptions \ref{asm_arrival}-\ref{asm_service} hold, $G$ has a finite $(2+\epsilon)$ moment and
$\law(\vn_0) = \pin_*$.  Then,  $\{\Mnhat_t(\Psi_{t+r}\f1);t\geq0,N\in\N\}$ is tight in $\Hone\ho$.
\end{lemma}

\begin{proof}
By definition \eqref{PsiDef} of $\Psi_s$,  the expression for $<\Mn (\phi)>$ in  \eqref{CoVar_mn} and  the bound \eqref{nunbar_int_bd}, for every $t,r\geq0$ and $N\in \N$ we have
\begin{align}\label{Mbound1}
        \Ept{\Mnhat_t(\Psi_{t+r}\f1)^2}=   \Ept{\int_0^t\nunbar_{t-s}\left(\frac{\overline G^2(s+r+\cdot)} {\overline G^2(\cdot)}\frac{g(\cdot)}{\overline G(\cdot)}\right)ds}\leq R_M(r).
\end{align}
Furthermore, by \eqref{MNderivative}, $\Mnhat_t(\Psi_{t+\cdot}\f1)$ is (locally) absolutely continuous with density (and hence, weak derivative) $\Mnhat_t(\Psi_{t+\cdot}h)$, which satisfies the  analogous relation
\begin{align*}
        \Ept{\Mnhat_t(\Psi_{t+r}h)^2}&=  \Ept{\int_0^t\nunbar_{t-s}\left(\frac{ g^2(s+r+\cdot)} {\overline G^2(\cdot)}\frac{g(\cdot)}{\overline G(\cdot)}\right)ds} \leq H^2  \Ept{\int_0^t\nunbar_{t-s}\left(\frac{\overline G^2(s+r+\cdot)} {\overline G^2(\cdot)}\frac{g(\cdot)}{\overline G(\cdot)}\right)ds},
\end{align*}
where the inequality holds by Assumption \ref{asm_service}.\ref{asm_h}, with $H$ being the constant in Remark \ref{remark_HH2}.  Therefore, the bound \eqref{nunbar_int_bd} implies that for $t, r \geq 0$, $N \in\mathbb{N}$,
\begin{equation}\label{Mbound2}
    \Ept{\Mnhat_t(\Psi_{t+r}\;h)^2}\leq H^2R_M(r).
\end{equation}
Likewise, by \eqref{MNderivative2}, $\Mnhat_t(\Psi_{t+\cdot}h)$ is (locally) absolutely continuous with density (and hence, weak derivative)  $\Mnhat_t(\Psi_{t+\cdot}h_2)$, which, by Assumption \ref{asm_service}.\ref{asm_h},
\begin{align}\label{Mbound3}
    \Ept{\Mnhat_t(\Psi_{t+r}h_2)^2}&= \Ept{\int_0^t\nunbar_{t-s}\left(\frac{|g^\prime(s+r+\cdot)|^2}{\overline G^2(\cdot)}\frac{g(\cdot)}{\overline G(\cdot)}\right)ds}\leq H_2^2R_M(r),
\end{align}
where $H_2$ is as in  Remark \ref{remark_HH2}. The lemma follows from \eqref{Mbound1}-\eqref{Mbound3}, the integrability of the function $R_M$ guaranteed by Lemma \ref{lem_nunbar}, and the tightness criteria of Lemma \ref{lem_Htight}, with $R_1=R_M$, $R_2=H^2R_M$, and $R_3=H_2^2R_M$.
\end{proof}

\subsubsection{Tightness of other terms.}\label{sec_Zothers}
In the next four lemmas, we prove tightness of other sequences associated with each of the other components of $\znhat$ on the right-hand side of \eqref{ZhatNEQ}.

\begin{lemma}\label{lem_ICterm}
    Suppose Assumptions \ref{asm_arrival}-\ref{asm_service} hold, $G$ has a finite $(2+\epsilon)$ moment and $\vn_0$ is distributed as $\pin_*$. Then the family $\{\znhat_0(t+\cdot);t\geq0,N\in\N \}$ is tight in $\Hone\ho$.
\end{lemma}

\begin{proof}
By \eqref{PhiDef},  the representation for $\zn$ given in \eqref{zn_nun_relation} and the scaling in \eqref{def_Ynhat},  we have
\begin{align*}
    \znhat_0(t+r) = \frac{1}{\sqrt{N}}\left( Z_0^{(N)}(t+r) -N \overline Z(t+r) \right) = \frac{1}{\sqrt{N}}\left( \nun_0(\Phi_{t+r}\f1) -N \overline Z(t+r) \right).
\end{align*}
Since $\nun_0$ has the same distribution as $\nun_*$ defined in \eqref{def_nuns}  and $\overline{Z}$ is defined by \eqref{def_Ybar}, we have
\begin{equation*}
    \znhat_0(t+\cdot)\deq \frac{1}{\sqrt{N}}\sum_{j=1}^{N}\left( \frac{\overline G(a^*_j+t+\cdot)}{\overline G(a^*_j)} -\int_0^\infty \overline G(x+t+\cdot)dx  \right),
\end{equation*}
where  $\{a^*_j;j\geq1\}$ is i.i.d with p.d.f. $\overline G$. Hence,  $\{\overline G(a^*_j+t+r)/\overline G(a^*_j)\}_{j\geq1}$ is also i.i.d. with mean $\int_0^\infty \overline G(x+t+r)dx$, and
\begin{equation*}
    \Ept{\znhat_0(t+r)^2} 	= \textit{Var} \left( \frac{\overline G(a^*_1+t+r)}{\overline G(a^*_1)} \right) 	\leq \Ept{\left|\frac{\overline G(a^*_1+t+r)}{\overline G(a^*_1)}\right|^2} =\int_0^\infty \frac{\overline G^2(x+t+r)}{\overline G(x)}\;dx \leq \int_{t+r}^\infty\overline G(x)dx    .
\end{equation*}
Hence, for the non-increasing,  integrable function $R_1(r)\doteq\int_r^\infty \overline G(x) dx$  (see Remark \ref{remark_GL2}), we have
\begin{equation}\label{ic_temp1}
    \Ept{\znhat_0(t+r)^2}\leq R_1(t+r)\leq R_1(r).
\end{equation}

Furthermore, since $\overline G$ has derivative $g$, $\znhat_0(t+\cdot)$ has derivative $(\znhat_0)'(t+\cdot)$, where
\[
    (\znhat_0)'(t+\cdot)\deq-\frac{1}{\sqrt{N}}\sum_{j=1}^{N}\left( \frac{ g(a^*_j+t+\cdot)}{\overline G(a^*_j)} -\int_0^\infty  g(x+t+\cdot)dx  \right).
\]
Similar to \eqref{ic_temp1}, using Assumption \ref{asm_service}.\ref{asm_h} and Remark \ref{remark_HH2}, it can be shown that
\begin{equation}\label{ic_temp2}
    \Ept{(\znhat_0)'(t+r)^2} \leq \Ept{\left|\frac{ g(a^*_1+t+r)}{\overline G(a^*_1)}\right|^2} \leq H^2 \Ept{\left|\frac{\overline G(a^*_1+t+r)}{\overline G(a^*_1)}\right|^2} \leq H^2 R_1(r).
\end{equation}
Finally, by Assumption \ref{asm_service}.\ref{asm_h},
 $(\znhat_0)'(t+\cdot)$ has derivative $(\znhat_0)''(t+\cdot)$, where
\[
    (\znhat_0)''(t+\cdot)\deq-\frac{1}{\sqrt{N}}\sum_{j=1}^{N}\left( \frac{ g^\prime(a^*_j+t+\cdot)}{\overline G(a^*_j)} -\int_0^\infty  g^\prime(x+t+\cdot)dx  \right),
\]
which satisfies the analogous inequality
\begin{equation}\label{ic_temp3}
    \Ept{(\znhat_0)''(t+r)^2} \leq \Ept{\left|\frac{ g^\prime(a^*_1+t+r)}{\overline G(a^*_1)}\right|^2}
    \leq H^2_2 \Ept{\left|\frac{\overline G(a^*_1+t+r)}{\overline G(a^*_1)}\right|^2} \leq H_2^2 R_1(r),
\end{equation}
where $H_2$ is the constant from Remark \ref{remark_HH2}. The lemma then  follows from \eqref{ic_temp1}-\eqref{ic_temp3} and the tightness criteria of Lemma \ref{lem_Htight} with $R_2=H^2R_1$, and $R_3=H_2^2R_1$.
\end{proof}

\begin{lemma}\label{lem_Eterm}
    Suppose Assumptions \ref{asm_arrival}, \ref{asm_service} and  \ref{asm_serviceP}.\ref{asm_gpp} hold, $G$ has a finite $(2+\epsilon)$ moment, and  $\vn_0$ is distributed as $\pin_*$. Then, there exists $C_E<\infty$ such that for every (deterministic) absolutely continuous function $f$ on $\ho$,
    \begin{align}\label{pre_part2}
        \Ept{\left|\int_0^t f(t-s)d\widehat E^{(N)}_s\right|^2} \leq  C_E\left\{t|f(t)|^2 +\left(\int_0^t s|f^\prime(s)|ds\right)\left(\int_0^t |f^\prime(s)|ds\right)+ \left(\int_0^t f(s)ds\right)^2\right\}.
    \end{align}
    Moreover, the family
    \[\left\{\int_0^t\overline G(\cdot+t-s)d\enhat_s;t\geq0,N\in\N\right\}\]
    is tight in $\Hone\ho$.
\end{lemma}

\begin{proof}
To prove \eqref{pre_part2}, fix $t\geq0$ and $f$ as in the statement of the lemma. Since $\en$ is a right-continuous stationary renewal process when $\vn_0$ is distributed as $\pin_*$, the process $\{\en_s,0\leq s\leq t\}$ has the same law as the process $\{\en_t-\en_{t-s},0\leq s \leq t\}$. Hence, since $\overline E=\id$ by \eqref{def_ebar}, the process $\check{E}^{(N)}$ defined as
\[
    \check{E}^{(N)}_s\doteq \enhat_t-\enhat_{t-s} = \frac{1}{\sqrt{N}}\left(\en_t-\en_{t-s}-N s\right)
\]
has the same distribution as $\enhat$ on $[0,t]$, and therefore, for every fixed $t\geq0$,
\[
    \int_0^t f(t-s)d\enhat_s =\int_0^tf(s)d\check{E}^{(N)}_s \deq\int_0^tf(s)d\enhat_s.
\]
Now, using integration by parts and the fact that $N=\lambdan+\beta\sqrt{N}$ by \eqref{def_lambdan}, we have
\begin{align*}
    \int_0^tf(s)d\enhat_s &=\frac{1}{\sqrt{N}}\int_0^tf(s)d\en_s- \frac{\lambdan}{\sqrt{N}}\int_0^tf(s)ds-\beta\int_0^tf(s)ds\\
	&=   \frac{f(t)}{\sqrt{N}}\left(\en_t-\lambdan t\right)
        - \frac{1}{\sqrt{N}}\int_0^tf^\prime(s) \left(\en_s-\lambdan s\right)ds-\beta\int_0^tf(s)ds.
\end{align*}
Combined with  $(a+b+c)^2\leq 8(a^2+b^2+c^2)$ and the Cauchy-Schwartz inequality, this implies
\begin{align}\label{pre_nupart2temp}
    \Ept{\left|\int_0^t f(t-s)d\widehat E^{(N)}_s\right|^2} & \leq \frac{8|f(t)|^2 }{N}\;\Ept{\left|E^{(N)}_t-\lambda^{(N)} t\right|^2} \notag\\
    &\hspace{5mm}+\frac{8}{N}\left(\int_0^t |f'(s)| \Ept{\left|E^{(N)}_s-\lambda^{(N)} s\right|^2}ds \right) \left(\int_0^t|f'(s)|ds\right)\notag\\
    &\hspace{5mm}+ 8\beta^2 \left(\int_0^t f(s)ds\right)^2.
\end{align}
Now, recall by Assumption \ref{asm_arrival} that $\en(\cdot)=\tilde E(\lambdan \cdot)$ is a stationary renewal process, satisfying $\Eptil{\en_t}=\lambdan t$, and let $\tilde U$ be the renewal function associated to $\tilde E$. By the equation (2.12) in \cite[Theorem 7.2.4]{Whi02} for the variance of a stationary renewal process and Lorden's inequality $\tilde U(t)\leq t+\Eptil{\tilde u_1^2}$  (e.g. see \cite[Proposition V.6.2]{Asm03}), we have  for all $s\in[0,t]$,
\begin{equation}\label{temp_dE0}
  \Ept{\left|E^{(N)}_s-\lambda^{(N)}s\right|^2}= \text{Var}(E^{(N)}_s)=\text{Var}(\tilde E(\lambdan s))=2\int_0^{\lambdan s} (\tilde U(v)-v-\frac{1}{2})dv \leq 2\Eptil{\tilde u_1^2} \lambdan s,
\end{equation}
where $\Eptil{\tilde u_1^2}=1+\sigma^2<\infty$ by Assumption \ref{asm_arrival}. Substituting the above inequality and the inequality $\lambdanbar\leq 1$ in \eqref{pre_nupart2temp}, we obtain \eqref{pre_part2} with $C_E=8\max\{\Eptil{\tilde u_1^2},\beta^2,1\}<\infty$.

Now define $\zeta_{N,t}(r)\doteq \int_0^t\overline G(t-s+r)d\enhat_s$ for $r\geq0$. Substitute $f(\cdot)=\overline G(\cdot+r)$ in \eqref{pre_part2} and note that $\overline G$ is decreasing, $g\leq H\overline G$ by Assumption \ref{asm_service}.\ref{asm_h}, and $\int_r^\infty \overline G(s)ds\leq\int_0^\infty \overline G(s)ds= 1$ by Assumption \ref{asm_service}.\ref{asm_g}, to obtain
\begin{align*}
    \Ept{|\zeta_{N,t}(r)|^2} &\leq C_E\left\{t\overline G^2(t+r)+\int_0^t sg(s+r)ds\;\int_0^t g(s+r)ds+ \left(\int_0^t \overline G(s+r)ds\right)^2\right\}\\
    &\leq C_E\left\{(t\overline G(t))\overline G(r)+H^2\int_0^\infty s\overline G(s)ds\int_r^\infty \overline G(s)ds+  \int_r^\infty \overline G(s)ds\right\}.
\end{align*}
The moment condition on $G$ implies $\overline G(t)=\mathcal{O}(t^{-(2+\epsilon)})$ for some $\epsilon>0$ as $t\to\infty$ and hence, $t\overline G(t)$ is uniformly bounded by a finite constant $c_1$, the constant $c_2\doteq\int_0^\infty s\overline G(s)ds$ is finite, and also the functions $\overline G$ and $\int_\cdot^\infty\overline G(x)dx$ are integrable (see Remark \ref{remark_HH2} and Remark \ref{remark_GL2}). Therefore, the function
\[
    R(r)\doteq c_1\overline G(r)+(1+c_2)\int_r^\infty \overline G(s)ds
\]
is integrable and we have
\begin{equation}\label{temp_dE1}
    \Ept{|\zeta_{N,t}(r)|^2} \leq C_E (1+H^2) R(r).
\end{equation}

Moreover, since $\overline G$ has derivative $-g$, almost surely, $\zeta_{N,t}$ has derivative $\zeta_{N,t}'$ on $\ho$ with $\zeta_{N,t}'(r)=-\int_0^tg(t-s+r)d\enhat_s$. Substituting $f=g(\cdot+r)$ in \eqref{pre_part2} and noting that $g\leq H\overline G$ and $|g'|\leq H_2\overline G$ by Remark \ref{remark_HH2}, we obtain the following analogous bound:
\begin{align}\label{temp_dE2}
\Ept{|\zeta_{N,t}'(r)|^2} &\leq C_E\left\{t g^2 (t+r) +\int_0^t s|g'(s+r)|ds \; \int_0^t|g'(s+r)|ds +\left(\int_0^tg(s+r)ds\right)^2\notag\right\}\\
&\leq C_E\left\{H^2(t\overline G(t))\overline G(r)+H_2^2 \int_0^\infty s\overline G(s)ds\;\int_r^\infty \overline G(s)ds + H^2 \int_r^\infty \overline G(s)ds\right\}\notag\\
&\leq C_E(H^2+H_2^2)R(r).
\end{align}

Next, since $g$ has derivative $g'$ by Assumption \ref{asm_service}.\ref{asm_h}, almost surely, $\zeta_{N,t}'$ has derivative $\zeta_{N,t}''$ on $\ho$ with $\zeta_{N,t}''(r)=-\int_0^tg'(t-s+r)d\enhat_s$.  Substituting $f=g'(\cdot+r)$ in \eqref{pre_part2} and using the fact that $|g'|\leq H_2\overline G$ by  Remark \ref{remark_HH2}, we have
\begin{align*}
    \Ept{|\zeta_{N,t}''(r)|^2} &\leq C_E\left\{t (g' (t+r))^2 +\left(\int_0^t(1+s)|g''(s+r)|ds\right)^2 +\left(\int_0^t|g'(s+r)|ds\right)^2\right\}\\
    &\leq C_E\left\{H_2^2(t\overline G(t))+\left(\int_0^\infty(1+s)|g''(s)|ds\right)^2+H_2^2 \left(\int_0^\infty \overline G(s)ds\right)^2\right\}.
\end{align*}
The first term on the right-hand side above is bounded, as discussed above, by $C_EH_2^2c_1<\infty$ due to the moment condition on $G$.  Also, since $g''$ is bounded and satisfies $g''(x)=\mathcal{O}(x^{-(2+\epsilon)})$ as $x\to\infty$ by Assumption \ref{asm_serviceP}.\ref{asm_gpp}, the integral in the second term is bounded by a constant $c_3<\infty$. The last term is  bounded by $C_EH_2^2$, due to Assumption \ref{asm_service}.\ref{asm_g}. Hence,
\begin{equation}\label{temp_dE3}
    \Ept{|\zeta_{N,t}''(r)|^2} \leq C_E(H_2^2c_1+c_3+H_2^2)<\infty.
\end{equation}
The result follows from \eqref{temp_dE1}-\eqref{temp_dE3} and the tightness criteria in Lemma \ref{lem_Htight} with $R_1=C_E(1+H^2)R$, $R_2=C_E(H^2+H_2^2)R$ and $R_3=C_EH_2^2(1+c_1)+C_Ec_2.$
\end{proof}

\begin{lemma}\label{lem_Xterm}
Suppose Assumptions \ref{asm_arrival}-\ref{asm_service} hold, $G$ has a finite $(2+\epsilon)$ moment and $\vn_0$ is distributed as $\pin_*$. Then, the family
\[
    \left\{-\xnphat_t\;\overline G(\cdot);t\geq0,N\in\N\right\}
\]
is tight in $\Hone\ho$.
\end{lemma}

\begin{proof}
Define $\zeta_{N,t}(r)\doteq-\xnphat_t \overline G(r)$. Then, by Corollary \ref{xnphat_bound}, for every interval $I\subset\ho$, $N\in\N$ and $t\geq0$,
\begin{equation}\label{temp_X1}
    \Prob{\|\zeta_{N,t}\|_{\Ltwo(I)} >\lambda }= \Prob{ \xnphat_t \|\overline G\|_{\Ltwo(I)} >\lambda}  \leq \frac{\Ept{\xnphat_t}\|\overline G\|_{\Ltwo(I)}}{\lambda}\leq  \frac{C_Q\|\overline G\|_{\Ltwo(I)}}{\lambda}.
\end{equation}
Since $\overline G\in\Ltwo\ho$ by Assumption \ref{asm_service}.\ref{asm_g}, for every $L < \infty$, substituting $I=(0,L)$ in \eqref{temp_X1}, we have
\begin{equation}\label{temp_X2}
         \lim_{\lambda\to \infty}\sup_{t\geq0, N\in\N}\;\Prob{\|\zeta_{N,t}\|_{\Ltwo(0,L)} >\lambda }\leq \lim_{\lambda\to \infty}  \frac{C_Q\|\overline G\|_{\Ltwo(0,\infty)}}{\lambda}=0,
\end{equation}
and for every $\delta >0,$ replacing $I$ and $\lambda$ with $(L,\infty)$ and $\delta$  in \eqref{temp_X1}, we  have
\begin{equation}\label{temp_X3}
     \lim_{L\to\infty}\sup_{t\geq0, N\in\N}\Prob{\|\zeta_{N,t}\|_{\Ltwo(L,\infty)} >\delta  }\leq\lim_{L\to \infty}\frac{C_Q\|\overline G\|_{\Ltwo(L,\infty)}}{\delta}=0.
\end{equation}

Moreover, since $\overline G$ has derivative $-g$, $\zeta_{N,t}$ has derivative $\zeta'_{N,t}(r)=\xnphat_t g(r)$. By Assumption \ref{asm_service}.\ref{asm_h} and definition of the constant $H$ in Remark \ref{remark_HH2}, $|\zeta'_{N,t}|\leq H|\zeta_{N,t}|$, and therefore, by \eqref{temp_X2} and \eqref{temp_X3}, for every $L < \infty$ and $\delta>0$, we have
\begin{equation}\label{temp_X4}
    \lim_{\lambda\to \infty}\sup_{t\geq0, N\in\N}\;\Prob{\|\zeta'_{N,t}\|_{\Ltwo(0,L)} >\lambda }=0,\quad\quad     \lim_{L\to\infty}\sup_{t\geq0, N\in\N}\Prob{\|\zeta'_{N,t}\|_{\Ltwo(L,\infty)} >\delta }=0
\end{equation}
Finally, by Assumption \ref{asm_service}.\ref{asm_h}, $g$ has derivative $g'$  with $|g'|\leq H_2 \overline G$,
where $H_2$ is the constant in Remark \ref{remark_HH2}, and hence, $\zeta'_{N,t}$ has derivative $\zeta''_{N,t}(r)=\xnphat_t g'(r)$ which satisfies $|\zeta''_{N,t}|\leq H_2|\zeta_{N,t}|$. Therefore, by \eqref{temp_X2}, for every $\lambda  > 0, L < \infty$  we have
\begin{equation}\label{temp_X5}
    \lim_{\lambda\to \infty}\sup_{t\geq0, N\in\N}\;\Prob{\|\zeta''_{N,t}\|_{\Ltwo(0,L)} >\lambda }=0.
\end{equation}
The lemma then follows from \eqref{temp_X2}-\eqref{temp_X5}, and the tightness criteria of Proposition \ref{prop_Htightness}.
\end{proof}

\begin{lemma}\label{lem_X2term}
Suppose Assumptions \ref{asm_arrival}-\ref{asm_service} and \ref{asm_serviceP}.\ref{asm_gpp} hold, $G$ has a finite $(2+\epsilon)$ moment,   and $\vn_0$ is distributed as $\pin_*$. Then, the family
\[
    \left\{\int_0^t{\xnphat_sg(\cdot+t-s)ds};t\geq0,N\in\N\right\}
\]
is tight in $\Hone\ho$.
\end{lemma}

\begin{proof}
Define $\zeta_{N,t}(r)\doteq \int_0^t{\xnphat_sg(t+r-s)ds}$. Since $g\leq H\overline G$ by Assumption \ref{asm_service}.\ref{asm_h} and $\overline G$ is decreasing,
\begin{align*}
    |\zeta_{N,t}(r)|\leq H\int_0^t\xnphat_s\overline G(t+r-s)ds \leq H (\overline G(r))^{\sfrac{1}{2}} \int_0^t \xnphat_{t-s}\;(\overline G(s))^{\sfrac{1}{2}}\;ds,
\end{align*}
with $H$ being the constant in Remark \ref{remark_HH2}. The bound \eqref{xnphat_bd_eq} of Corollary \ref{xnphat_bound} then implies that for every interval $I\subset\ho$,
\begin{align}\label{temp_IX0}
    \Prob{\left\|\int_0^t\xnphat_s\overline G(\cdot+t-s)ds \right\|_{\Ltwo(I)}>\lambda}
    &\leq \frac{1}{\lambda}\|\overline G\;^{\sfrac{1}{2}}\|_{\Ltwo(I)}\int_0^t \Ept{\xnphat_{t-s}}\overline G(s)^{\sfrac{1}{2}}\;ds\notag\\
    &\leq \frac{C_Q}{\lambda}\|\overline G\;^{\sfrac{1}{2}}\|_{\Ltwo(I)}\| \overline G\;^{\sfrac{1}{2}}\|_{\Lone\ho}.
\end{align}
By the moment assumption on $G$,  $\overline G(r)=\mathcal{O}(r^{-(2+\epsilon)})$ as $r\to\infty,$ and $\overline G$ is also bounded by one, and thus, $\overline G\;^{\sfrac{1}{2}}\in\Lone\ho\cap\Ltwo\ho.$ Therefore, for every $L  < \infty$, substituting $I=(0,L)$ in \eqref{temp_IX0}, we have
\begin{equation}\label{temp_IX1}
         \lim_{\lambda\to \infty}\sup_{t\geq0, N\in\N}\;\Prob{\|\zeta_{N,t}\|_{\Ltwo(0,L)} >\lambda }\leq\lim_{\lambda\to \infty}  \frac{HC_Q}{\lambda}\|\overline G\;^{\sfrac{1}{2}}\|_{\Ltwo\ho}\| \overline G\;^{\sfrac{1}{2}}\|_{\Lone\ho}=0,
\end{equation}
and for every $\delta>0$, replacing $I$ and $\lambda$ with $(L,\infty)$ and $\delta$ in \eqref{temp_IX0}, we  have
\begin{equation}\label{temp_IX2}
    \lim_{L\to\infty}\sup_{t\geq0, N\in\N}\Prob{\|\zeta_{N,t}\|_{\Ltwo(L,\infty)} >\delta }\leq\lim_{L\to \infty}  \frac{HC_Q}{\delta}\|\overline G\;^{\sfrac{1}{2}}\|_{\Ltwo(L,\infty)}\| \overline G\;^{\sfrac{1}{2}}\|_{\Lone\ho}=0.
\end{equation}
Next,  by Assumption \ref{asm_service}.\ref{asm_h}, the derivative $g'$ of $g$ satisfies  $|g'|\leq H_2\overline G$,
with $H_2$ the constant in  Remark \ref{remark_HH2}, and hence,  $\zeta_{N,t}$ has derivative $\zeta^\prime_{N,t},$ where
\[
    |\zeta^\prime_{N,t}(r)|=\left|\int_0^t \xnphat_s\; g^\prime(t+r-s)ds\right|\leq H_2\int_0^t \xnphat_s\; \overline G(t+r-s)ds.
\]
Using the above inequality and \eqref{temp_IX0}, exactly analogous to \eqref{temp_IX1} and \eqref{temp_IX2}, for every $L<\infty$ and $\delta>0$ we have
\begin{equation}\label{temp_IX3}
    \lim_{\lambda\to \infty}\sup_{t\geq0, N\in\N}\;\Prob{\|\zeta'_{N,t}\|_{\Ltwo(0,L)} >\lambda }=0,\quad\quad     \lim_{L\to\infty}\sup_{t\geq0, N\in\N}\Prob{\|\zeta'_{N,t}\|_{\Ltwo(L,\infty)} >\delta }=0.
\end{equation}

Furthermore,  by Assumption \ref{asm_serviceP}.\ref{asm_gpp}, $g'$ has bounded derivative $g''$, and hence, $\zeta^\prime_{N,t}$ has derivative
\[
    \zeta^{\prime\prime}_{N,t}(r)=\int_0^t\xnphat_sg''(t-s+r)ds.
\]
By the same assumption, $g''$ is uniformly bounded by a constant (say) $C_{g''}$ and satisfies $g''(x)=\mathcal{O}(x^{-(2+\epsilon)})$ as $x\to\infty$, that is, there exist $T\in\ho$ and $C_T<\infty$ such that $|g''(x)|\leq C_T x^{-(2+\epsilon)}$ for all $x\geq T$. Therefore,  for $r\geq0$,
\[
    |\zeta^{\prime\prime}_{N,t}(r)|\leq \int_0^t \xnphat_s| g^{\prime\prime}(t-s+r)|ds \leq C_{g''}\int_0^{T\wedge t} \xnphat_s ds + C_T\int_{T\wedge t}^t \xnphat_{t-s} \;s^{-(2+\epsilon)}ds.
\]
Another use of the bound \eqref{xnphat_bd_eq} shows that for every $\lambda, L<\infty$,
\[
    \Prob{\|\zeta''_{N,t}\|_{\Ltwo(0,L)} >\lambda } \leq \frac{1}{\lambda}\Ept{\|\zeta''_{N,t}\|_{\Ltwo(0,L)}} \leq
    \frac{\sqrt{L}}{\lambda}\Ept{\|\zeta''_{N,t}\|_L} \leq \frac{C_Q\sqrt{L}}{\lambda}\left(C_{g''}T+C_T\int_0^\infty s^{-(2+\epsilon)}ds\right),
\]
and hence,
\begin{equation}\label{temp_IX4}
    \lim_{\lambda\to \infty}\sup_{t\geq0, N\in\N}\;\Prob{\|\zeta''_{N,t}\|_{\Ltwo(0,L)} >\lambda }=0.
\end{equation}
The lemma follows from \eqref{temp_IX1}-\eqref{temp_IX4}, and the criteria of Proposition \ref{prop_Htightness} for tightness in $\Hone\ho.$
\end{proof}

\subsection{Proof of Proposition \ref{thm_tightness}.} \label{sec_tight_proof}

First, we claim that the family $\{\xnhat_t;t\geq0,N\in\N\}$ is tight in $\R$. It clearly suffices to  establish tightness of $\{\xnphat_t;t\geq0,N\in\N\}$ and $\{\xnmhat_t;t\geq0,N\in\N\}$. Tightness of the first family follows from Corollary \ref{xnphat_bound}, which shows that $\{\xnphat\}$ is (uniformly) bounded in $\Lone$. To prove the tightness of $\{\xnmhat\}$, note that  since $\ynhat_t$ takes values in $\Y$ (see Theorem \ref{thm_ergodic}),  $\xnmhat_t =- \widehat Z_t^{(N)}(0)$ for every $t\geq0$  and $N\in\N$.  From \cite[Lemma 3.3.b]{AghRam15spde}, the embedding $I:\Hone\ho\mapsto\C$ that takes $f$ to its (unique) continuous representation is continuous, and hence, the mapping $f\mapsto I[f](0)$ is continuous from $\Hone\ho$ to $\R$. Therefore, the tightness of $\{\xnmhat_t\}$ follows from the tightness of $\{\znhat_t\}$ in $\Hone\ho$ proved in Proposition \ref{prop_Zuhb}, and the continuous mapping theorem. This proves the claim.

By another application of Proposition \ref{prop_Zuhb} and the claim above, $\{\ynhat_t=(\xnhat_t,\znhat_t);t\geq0,N\in\N\}$ is tight in $\Y$, when $\law(\vn_0)  = \pin_*$.  By Theorem \ref{thm_ergodic}, $\law(\ynhat_t)$ converges in distribution to $\pinhat$, as $t\to\infty.$ The result then follows from Remark \ref{remark_tight}. This concludes the proof.

\section{A diffusion model.}\label{sec_spde}

We now describe the infinite-dimensional Markov process $\{Y_t;t\geq0\}$, which was proposed in \cite{AghRam15spde} as an alternative diffusion model for the GI/GI/N queue that is more tractable than the limit process obtained in \cite{KasRam13}.  In the next section, we show that for each $t\geq0$, the finite-dimensional projections of the scaled sequence $\{\ynhat (t)\}$ converge to those of the corresponding marginals $Y(t)$ of the diffusion model $Y$, and the corresponding sequence of stationary distributions of $\ynhat$ converge to the unique invariant distribution of $Y$. We now introduce the diffusion model and review relevant results from \cite{AghRam15spde}.

The diffusion model is driven by a Brownian motion  $B= \{B(t), t \geq 0\}$, and an independent continuous martingale measure  ${\mathcal M} = \left\{{\mathcal M}_t (A), A \in {\mathcal B}[0,\infty), t \in   [0,\infty) \right\}$ with covariance
\begin{equation}\label{Mcov_original}
    \langle \M(A),\M(\tilde A) \rangle_t = t\;\int_0^\infty \indic{A\cap\tilde A}{x}g(x)dx.
\end{equation}
Note that ${\mathcal M}$ is a space-time white noise on $[0,\infty)^2$ based on the measure $g(x) dx \otimes dt$ (again, see \cite[Chapters 1 and 2]{WalshBook} for relevant definitions).
We assume that both $B$ and $\M$ are defined on a probability space $(\widehat\Omega,  \widehat{\mathcal F}, \widehat{\mathbb{P}})$.   Roughly speaking, the Brownian motion $B$  captures limit  fluctuations of the arrival process  and  $\M$ is  the limit of the sequence $\{\Mnhat\}$ of   scaled compensated departure processes introduced in \eqref{MartingaleMeasureNDef} and is  independent of  $B$ \cite[Proposition 8.4]{KasRam13} (see Lemma \ref{thm_HKresult} for a precise statement of this result).  Also, for every (deterministic) bounded function $\varphi$ on $\hc\times\hc$, the process
\begin{equation}\label{MintegralDef}
    \M_t(\varphi)\doteq\iint\limits_{[0,\infty)\times[0,t]}\varphi(x,s)\M(dx,ds),
\end{equation}
which is the stochastic integral of $\varphi$ on $\hc\times[0,t]$ with respect to the martingale measure,  is itself a martingale with covariance functional
\begin{equation} \label{Mcov}
    \langle \M(\varphi) \rangle_t = \int_0^t\int_0^\infty \varphi^2(x,s)g(x)dx\;ds.
\end{equation}
Analogous to \eqref{def_Hn}, for $f\in\mathbb{C}_b[0,\infty)$ and $t>0$, define the stochastic convolution integral $\mathcal{H}_t(f)$ as
\begin{align}\label{def_H}
    \mathcal{H}_t(f)\doteq \M_t(\Psi_t f).
\end{align}
Also, let  $\tilde{\mathcal{F}}_t\doteq \sigma( B(s), \mathcal{M}_s(A); 0\leq s\leq t,A\in\mathcal{B}\ho),$ and let the filtration $\{{\mathcal F}_t\}$ denote the augmentation (see \cite[Definition 7.2 of Chapter 2]{KarShr91}) of $\{\tilde{{\mathcal F}}_t\}$ with respect to $\widehat{\mathbb{P}}$. Finally, define
\begin{equation}\label{def_E}
  E_t=\sigma B_t-\beta t,\quad\quad t\geq0,
\end{equation}
where $\sigma$ is the variance of $\tilde{G}_E$, as defined prior to  Assumption \ref{asm_arrival}.

\begin{definition}(\textbf{Diffusion Model})\label{def_solutionProcess}
For every $\Y$-valued random element $Y_0=(X_0,Z_0)$, the diffusion model $Y^{Y_0}=\{Y_t^{Y_0}=(X_t,Z_t(\cdot));t\geq0\}$ with initial condition $Y_0$ is defined as
\begin{equation} \label{XKdefinition}
    (K,X)\doteq\Lambda(E,X_0,Z_0(\cdot)-\mathcal{H}(\f1)),
\end{equation}
where $\Lambda$ is the CMS mapping from Definition \ref{def_cmsm}, and for every $t,r\geq0$,
\begin{equation}\label{def_Z}
    Z_t(r) \doteq  Z_0(t+r) -\M_t(\Psi_{t+r}\f1)+\overline G(r)K_t-\int_0^t g(t-u+r)K_udu.
\end{equation}
\end{definition}

When the initial condition $Y_0$ is clear from the context, we often omit the superscript $Y_0$.  Also, deterministic initial conditions are denoted by lowercase $y_0=(x_0,z_0)\in\Y.$
Note that \eqref{XKdefinition} and \eqref{def_Z} are natural limit analogs of the equations \eqref{xk} and \eqref{Zhat_dK}. The following results  were established in \cite{AghRam15spde}.  Recall that $\pi$ is said to be an invariant distribution
of a Markov semigroup $\{{\mathcal P}_t\}$ on $\Y$ if for every $t \geq 0$, $\pi {\mathcal P}_t = \pi$.

\begin{proposition}\label{thm_spde}
Suppose Assumptions \ref{asm_arrival}-\ref{asm_service} hold, and $G$ has a finite $(2+\epsilon)$ moment for some $\epsilon > 0$.  Then the following statements are true:
\begin{enumerate}
    \item for every initial condition $Y_0$, the process $\{Y^{Y_0}_t;t\geq0\}$ has a continuous $\Y$-valued version;
    \item If $P^y$ denotes the law of $Y^y$ then $\{P^y, y \in \Y\}$ is a
time-homogeneous Feller Markov family;
    \item  The transition semigroup $\{{\mathcal P}_t\}$ associated with
$\{P^y, y \in \Y\}$ has at most one invariant distribution.
\end{enumerate}
\end{proposition}

\begin{proof}
Part a, b, and c are proved in Proposition 4.14, Theorem 3.7, and Theorem 3.8 of \cite{AghRam15spde}, respectively.
\end{proof}

In what follows, by an invariant distribution of the diffusion model, we will mean an invariant distribution of its  associated semigroup $\{{\mathcal P}_t\}$.

\section{Convergence of subsequential limits.}\label{sec_conv}

In this section we establish convergence of  the sequence of stationary distributions $\{\pinhat\}$, as stated in Theorem \ref{thm_interchange}. Since $\{\pinhat\}$ is tight by Proposition \ref{thm_tightness}, it suffices to uniquely characterize subsequential limits of $\{\pinhat\}$.   We will show that every subsequential limit of $\{\pinhat\}$  is an  invariant distribution (and therefore equal to the unique invariant distribution) of the diffusion model introduced in Section \ref{sec_spde}. As outlined in Fig.\ \ref{Figure:clt}, the proof follows three main steps. First, we show in Section \ref{sec_conv_CLT} (see Proposition \ref{prop_subseq}) that under suitable conditions on (the laws of) the initial conditions $\{V^{(N)}_0\}$, for every $t \geq 0$, finite-dimensional projections of $\ynhat_t$ converge weakly to those of $Y_t$ as $N \rightarrow \infty$.   Then, in Section \ref{sec_conv_proof} we verify that these conditions are satisfied when $V^{(N)}_0 = V^{(N)}_\infty$, which is the invariant distribution of $V^{(N)}$ from  Proposition \ref{prop_Vergodic}.  Also, by Proposition \ref{prop_Vergodic}, when $\vn_0 =  \vn_\infty$, for any $t \geq 0$, we see that $\ynhat_t = \widehat{\Psi} (V^{(N)}_t)$, with $\widehat{\Psi}$ given as in \eqref{ynhat_representation},  has the same distribution as $\ynhat_0 = \widehat{\Psi} (\vn_\infty)$, and hence, by \eqref{def_pnhat}, must have  law  $\pinhat$. Thus, the limit  of any convergent subsequence of $\{\pinhat\}$ must be an invariant distribution of the diffusion model.  Theorem \ref{thm_interchange} follows from Corollary \ref{cor_exists}.

\begin{figure}
\centering
\includegraphics[height=3.75cm]{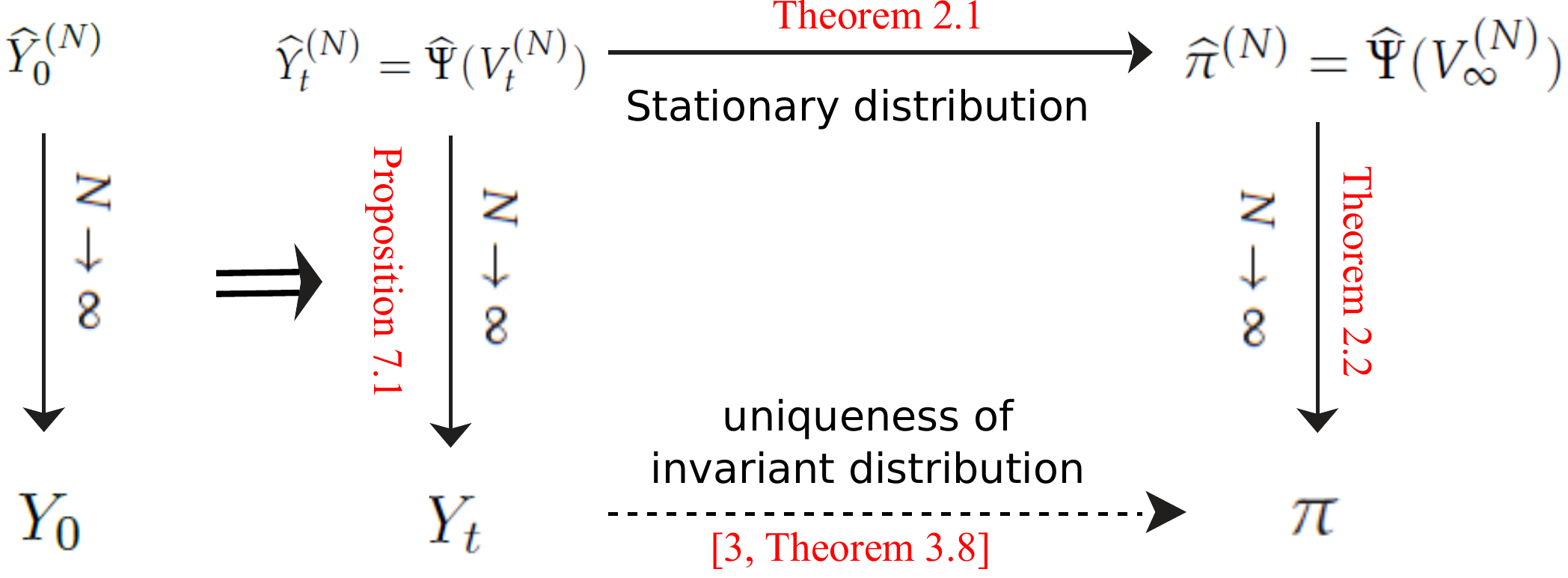}
\caption{Schematic for Proof of Theorem \ref{thm_interchange}}
\label{Figure:clt}
\end{figure}

\subsection{A  central limit theorem.}\label{sec_conv_CLT}

Recall from Section \ref{sec_spde} that $\{B_t;t\geq0\}$ is a Brownian motion, $E_t= \sigma B_t-\beta t$ for $t\geq0$, and $\M$ is an independent  martingale measure with covariance functional given by  \eqref{Mcov_original}. We now state  a result from
\cite{KasRam13} that, in particular, shows that $\enhat$ and $\Mnhat$ are asymptotically  independent and converge in distribution to $E$ and $\M$, respectively.

\begin{lemma} \label{thm_HKresult}
Suppose Assumptions \ref{asm_arrival}-\ref{asm_service} hold, $G$ has a finite $(2+\epsilon)$ moment for some $\epsilon > 0$,  and the sequence of initial conditions $\{\vn_0= (\ren_0,\xn_0,\nun_0)\}$ satisfies condition \eqref{asm_ic}.  Moreover, for $N\in\N$, define  $\ynhat_0\doteq\widehat\Psi(\vn_0)$,  with $\widehat\Psi$ given by \eqref{def_hPhi}, and assume that there exists a $\Y$-valued random element $Y_0=(X_0,Z_0)$ independent of $\M$ and $B$ such that, as $N\to\infty$,
\begin{equation}\label{HKresult_conditionY}
        \ynhat_0\Rightarrow Y_0,\quad\quad \text{ in }\Y.
\end{equation}
Then, for every $t \geq 0, k\in\N$, and  $\varphi_1,...,\varphi_k \in  \mathbb{C}_b(\hc\times\hc; \R)$, we have
\begin{equation}\label{HKconv}
    \left(\enhat,\xnhat_0,\znhat_0,\Hnhat(\f1), \Mnhat_t(\varphi_1),...,\Mnhat_t(\varphi_k)\right)\Rightarrow \left(E,X_0,Z_0,\mathcal{H}(\f1), \M_t(\varphi_1),...,\M_t(\varphi_k)\right),
\end{equation}
weakly in $\mathbb{D}[0,\infty)\times \Y   \times \mathbb{D}[0,\infty)\times \R^{k},$ where $\mathbb{D}[0,\infty)$ is equipped with the topology of uniform convergence on compact sets.
\end{lemma}

\begin{proof}
We first show that the conditions of the lemma imply   Assumptions 1-4 in \cite{KasRam13}. As discussed earlier, Assumption 1 of \cite{KasRam13} holds by condition \eqref{asm_ic} and Lemma \ref{prop_fluid}.\ref{eLLN}.  Assumptions 2 and 4 of \cite{KasRam13} on the service distribution hold because Assumption \ref{asm_service} implies that $h$ is bounded.  Moreover, by Assumption \ref{asm_arrival}, $\en$ is a renewal process with inter-arrival times $\{u^{(N)}_n;n\geq1\}$ with $u^{(N)}_n = \tilde u_n/\lambdan$, where $\{\tilde u_n;n\in\N\}$ is an i.i.d. sequence with mean $1$ and variance $\sigma^2$. Hence, $u^{(N)}_1$ satisfies
\[
    \Ept{u^{(N)}_1}=\frac{\Ept{\tilde u_1}}{\lambdan}=\frac{1}{\lambdan},\quad \text{Var}(u^{(N)}_1)=\frac{\text{Var}(\tilde u_1)}{(\lambdan)^2}=\frac{\sigma^2}{(\lambdan)^2},
\]
and since $\Eptil{\tilde u_1^2}<\infty$ and $\lambdan/\sqrt{N}\to\infty$ as $N\to\infty$, the Lindeberg condition
\[
    \lim_{N\to\infty} N^2\Ept{(u^{(N)}_1)^2\indic{(-\infty, u^{(N)}_1\sqrt{N})}{\epsilon}} =\lim_{N\to\infty} \frac{N^2}{(\lambdan)^2}\Ept{(\tilde u_1)^2\indic{(\infty, \tilde u_1)}{\frac{\lambdan\epsilon}{\sqrt{N}}}} =0,
\]
holds. Therefore, (part a. of) Assumption 3 of \cite{KasRam13} is also satisfied with $\overline \lambda=1$.

Then  \eqref{HKconv}  follows from  \eqref{HKresult_conditionY}, Corollary 8.7 in \cite{KasRam13} and the second display in the proof of that corollary with $\tilde f_j$ replaced by $\varphi_j$ and $f_1\doteq\f1$. Note that Assumption 5 in \cite{KasRam13} is not required for our assertion to hold,  as it is only used to establish convergence of  $\nunhat,$ which
we do not claim.
\end{proof}

Next, we use the result of Lemma \ref{thm_HKresult} to establish convergence of the sequence of finite-dimensional projections of $\ynhat_t$ for any $t \geq 0$. Let $\{e_k;k\in\N\}$ be a sequence of  functions with compact support in $\Hone\ho$  that form a countable orthonormal basis for the Hilbert space $\Hone\ho$ with inner product denoted by $\langle \cdot ,\cdot\rangle_{\Hone}$  (e.g., Daubechies Wavelets, see \cite[Section 9.2]{Wal13}), and let $e_j'$ denote the weak derivative of $e_j, j\in\N.$

\begin{proposition}\label{prop_subseq}
Suppose the conditions of Lemma \ref{thm_HKresult} hold, let $Y_0$ be the limit in \eqref{HKresult_conditionY}
and let $Y=(X,Z)$ be the diffusion model associated with the initial condition $Y_0$, as specified in Definition \ref{def_solutionProcess}.  Then, for every $t\geq0$ and $k\in \N$, as $N\to\infty$,
\begin{equation}\label{subseq}
    \left(\xnhat_t,\langle \znhat_t,e_1\rangle_{\Hone},...,\langle \znhat_t,e_k\rangle_{\Hone}\right)\Rightarrow \left(X_t,\langle Z_t,e_1\rangle_{\Hone} ,...,\langle Z_t,e_k\rangle_{\Hone}\right).
\end{equation}
\end{proposition}

\begin{proof}
For this proof, we use the convention that all continuity results on $\mathbb{D}[0,\infty)$ are with respect to the topology of uniform convergence on compact sets.    Fix $t\geq0$ and $k\in\N$. First, recall that by equations \eqref{xk} and \eqref{XKdefinition} for $(\kn,\xn)$ and $(K,X)$, respectively, we have
    \begin{equation}\label{temp_CLT0}
      (\knhat,\xnhat)=\Lambda(\enhat, \xnhat_0,\znhat_0-\Hnhat(\f1)),\quad\quad         (K,X)=\Lambda(E,X_0,Z_0-{\cal H}(\f1)),
    \end{equation}
    where,  by Lemma \ref{lem_Lambda}, the centered many-server mapping $\Lambda$ is continuous.

    Next, we compute the projections of $\znhat_t$ and $Z_t$ onto the directions $e_1,...,e_k$. Recall the representations \eqref{Zhat_dK} of $\znhat_t$ and \eqref{Zphat_dK} of $(\znhat)'$. By definition \eqref{PsiDef} of $\Psi_s,s\geq0$, $Z_t$ satisfies  \eqref{def_Z} and by equation (3.19) of \cite[Proposition 4.14]{AghRam15spde}, $Z_t$ has weak derivative $Z'_t$ where
    \begin{equation}\label{Zp}
    	Z'_t(r) \doteq  Z'_0(t+r) +\M_t(\Psi_{t+r}h)-K_tg(r)-\int_0^t K_sg'(t-s+r)ds.
    \end{equation}
    First, since the translation mapping $f\mapsto f(t+\cdot)$ is continuous from $\Ltwo\ho$ to itself, and for any $u\in\Ltwo$, the projection mapping $f\mapsto\langle f,u\rangle_{\Ltwo}$  is continuous from $\Ltwo\ho$ to $\R$,  there exist continuous mappings $F^1_j,F^2_j:\Ltwo\ho\mapsto\R, j=1,...,k,$ such that
    \begin{equation}\label{Temp_CLT1}
        \langle\znhat_0(t+\cdot),e_j\rangle_{\Ltwo} =F^1_j(\znhat_0), \quad  \langle Z_0(t+\cdot),e_j\rangle_{\Ltwo}=F^1_j(Z_0),
    \end{equation}
    and
    \begin{equation}
         \langle(\znhat_0)'(t+\cdot),e'_j\rangle_{\Ltwo} =F^2_j((\znhat_0)'),\quad\langle Z'_0(t+\cdot),e'_j\rangle_{\Ltwo}=F^2_j(Z'_0).
    \end{equation}

    Next, for every $j=1,\dots,k$ define the function $\varphi_j:\hc\times\hc\mapsto\R$ as
    \begin{equation}\label{tdef_phi}
      \varphi_j(x,s)\doteq \langle \Psi_{t+\cdot}\f1(x,s),e_j\rangle_{\Ltwo} = \int_0^\infty \Psi_{t+r}\f1(x,s)e_j(r)\;dr= \int_0^\infty \frac{\overline G(x+t-s+r)}{\overline G(x)}e_j(r)dr,
    \end{equation}
    which is bounded by $\|e_j\|_{\Lone}< \infty$ (recall that $e_j$ has compact support by choice). Also, by the continuity of $\overline G$ and boundedness of $1/\overline G$ on finite intervals,  the dominated convergence theorem shows that $\varphi_j$ is continuous. Since $\varphi_j\in\mathbb{C}_b(\hc\times\hc)$, by the stochastic Fubini theorem for orthogonal martingale measures $\Mnhat$ (see e.g. \cite[Theorem 2.6]{WalshBook}), we have
    \begin{equation}
        \langle\Mnhat_t(\Psi_{t+\cdot}\f1),e_j \rangle_{\Ltwo} = \int_0^\infty  \Mnhat_t(\Psi_{t+r}\f1)e_j(r) dr =\Mnhat_t\left(\int_0^\infty\Psi_{t+r}\f1 e_j(r) dr\right)= \Mnhat_t(\varphi_j).
    \end{equation}
    Similarly, the function $\tilde \varphi_j:\hc\times\hc\mapsto\R$ defined as
    \begin{equation}\label{tdef_tphi}
        \tilde \varphi_j(x,s)\doteq \langle \Psi_{t+\cdot}h(x,s),e^\prime_j\rangle_{\Ltwo}= \int_0^\infty \Psi_{t+r}h(x,s)e'_j(r)\;dr= \int_0^\infty \frac{g(x+t-s+r)}{\overline G(x)}e^\prime_j(r)dr,
    \end{equation}
    is also continuous and bounded (by $H\|e'_j\|_{\Lone}<\infty$) due to  Assumption \ref{asm_service}.\ref{asm_h}.  Another application of the stochastic Fubini theorem for orthogonal martingales shows that
    \begin{equation}
        \langle\Mnhat_t(\Psi_{t+\cdot}h),e'_j \rangle_{\Ltwo} = \int_0^\infty  \Mnhat_t(\Psi_{t+r}h)e'_j(r) dr =\Mnhat_t\left(\int_0^\infty\Psi_{t+r}h e'_j(r) dr\right)= \Mnhat_t(\tilde\varphi_j).
    \end{equation}
    Similarly, replacing $\Mnhat$ with $\M_t$, we have
    \begin{equation}
        \langle\M_t(\Psi_{t+\cdot}\f1),e_j \rangle_{\Ltwo} = \M_t(\varphi_j),\quad\quad\langle\M_t(\Psi_{t+\cdot}h),e'_j \rangle_{\Ltwo} = \M_t(\tilde\varphi_j).
    \end{equation}

    Finally, since $\knhat$ is almost surely bounded on finite intervals, by Fubini's theorem,
    \[
        \langle \knhat_t\overline G(\cdot)-\int_0^t\knhat_sg(t-s+\cdot) ds,e_j\rangle_{\Ltwo} = \knhat_t \langle \overline G,e_j\rangle_{\Ltwo} -\int_0^t \knhat_s \langle g(t-s+\cdot),e_j\rangle_{\Ltwo} ds
    \]
for  $j=1,\dots,k$.
    By Assumptions \ref{asm_service}.\ref{asm_g} and \ref{asm_service}.\ref{asm_h}, both $\overline G$ and $g$ lie in $\Ltwo\ho$ (see Remarks \ref{remark_HH2} and \ref{remark_GL2}), and hence $\langle \overline G,e_j\rangle_{\Ltwo}$ is finite, and the function $s\mapsto \langle g(t-s+\cdot),e_j\rangle$ is bounded and continuous on $[0,t]$. Moreover, for every $t\geq0$ and $u \in \mathbb{C}_b[0,\infty)$, the mapping $f\mapsto\int_0^tf(s)u(s)ds$ from $\D$ to $\R$ is continuous. Therefore, by equation \eqref{temp_CLT0}, there exist continuous functions $\tilde F_j^3:\D\mapsto\R$ and $F_j^3:\D \times \R \times \D\mapsto\R$, such that
    \begin{equation}
      \langle \knhat_t\overline G(\cdot)-\int_0^t\knhat_sg(t-s+\cdot) ds,e_j\rangle_{\Ltwo} =\tilde F_j^3(\knhat)=F_j^3(\enhat, \xnhat_0,\znhat_0-\Hnhat(\f1)).
    \end{equation}
    By the same argument,
    \begin{equation}
      \langle K_t\overline G(\cdot)-\int_0^tK_sg(t-s+\cdot) ds,e_j\rangle_{\Ltwo} =F_j^3(E, X_0,Z_0-\mathcal{H}(\f1)).
    \end{equation}
    Similarly, by Assumption \ref{asm_service}.\ref{asm_h}, $g'$ also lies in $\Ltwo,$ and hence  $\langle g,e'_j\rangle$ is finite and the mapping $s\mapsto\langle g^\prime(t-s+\cdot)e^\prime_j\rangle$ is bounded and continuous on $[0,t]$. Therefore, there exists a continuous function $F^4_j:\D \times \R \times \D\mapsto\R$ such that
    \begin{equation}
      \langle \knhat_t g(\cdot)+\int_0^t\knhat_sg'(t-s+\cdot) ds,e'_j\rangle_{\Ltwo} =F_j^4(\enhat, \xnhat_0,\znhat_0-\Hnhat(\f1)),
    \end{equation}
    and
    \begin{equation}\label{Temp_CLT2}
      \langle K_t g(\cdot)+\int_0^tK_sg'(t-s+\cdot) ds,e'_j\rangle_{\Ltwo} =F_j^4(E, X_0,Z_0-\mathcal{H}(\f1)).
    \end{equation}

    Now, combining equation \eqref{temp_CLT0}, representations \eqref{Zhat_dK} of $\znhat_t$, \eqref{Zphat_dK} of $(\znhat)'$,  \eqref{def_Z}  of $Z_t$ and \eqref{Zp} of $Z'_t$, equations \eqref{Temp_CLT1}-\eqref{Temp_CLT2}, and the fact that addition is continuous on $\D$, we conclude that there exists a continuous  function \[F:\I\doteq\R\times\Hone\ho\times \D\times\D\times \R^{2k}\mapsto\R^{k+1},\]
    such that
    \begin{equation}\label{inner_temp1}
        \left(\xn_t,\langle \znhat_t,e_1\rangle_{\Hone},...,\langle \znhat_t,e_k\rangle_{\Hone}\right)=F\left(\inhat\right),\quad\text{ and }\quad
        \left(X_t,\langle Z_t,e_1\rangle_{\Hone},...,\langle Z_t,e_k\rangle_{\Hone}\right)=F(I),
    \end{equation}
    where
    \[
        \inhat = \left(\xnhat_0,\znhat_0,\enhat,\Hnhat(\f1),\Mnhat_t(\varphi_j),\Mnhat(\tilde\varphi_k);j=1,\dots,k\right),
    \]
    and
    \[
        I= \left(X_0,Z_0,E,\mathcal{H}(\f1),\M_t(\varphi_j),\M_t(\tilde\varphi_j);j=1,\dots k\right).
    \]
    Thus,  by Lemma \ref{thm_HKresult} with  $\varphi_j$ and $\tilde \varphi_j ;j=1,...,k,$ defined in \eqref{tdef_phi} and \eqref{tdef_tphi},  we have
    \begin{equation}\label{inner_temp2}
        \inhat\Rightarrow I,
    \end{equation}
    in  $\I$. The result in \eqref{subseq} follows from \eqref{inner_temp1}, \eqref{inner_temp2}, and the continuous mapping theorem.
\end{proof}

\subsection{A fluid limit theorem for $\vn_\infty$.}\label{sec_conv_ic}

In order to use Proposition  \ref{prop_subseq} with the sequence of initial conditions $\vn_0=\vn_\infty,N\in\N$, we verify the conditions of that proposition. Recall that
\[
    \vn_\infty=(\ren_\infty,\xn_\infty,\nun_\infty),\quad\text{ and } \quad\vnbar_\infty = (\renbar_\infty,\xnbar_\infty,\nunbar_\infty) =\frac{1}{N}(\ren_\infty,\xn_\infty,\nun_\infty).
\]

\begin{lemma}\label{lem_nunbarTight}
Suppose Assumptions \ref{asm_arrival} and \ref{asm_service} hold.  Then, $\{\nunbar_\infty\}_{N\in\N}$ is tight in $\emF$.
\end{lemma}

\begin{proof}
For  $N\in\N$, consider the process $\{V_t^{(N)}=(\ren_t,\xn_t,\nun_t),t\geq0\}$ initialized at $\vn_*$, as defined
in Section \ref{sec_ic}.  Then, for every $t\geq0$, $\nunbar_t$ is a sub-probability measure and hence
\begin{equation}\label{temp_nuntight1}
        \sup_{N\in\N,\;t\geq0}\Ept{\nunbar_t(\f1)}\leq 1.
\end{equation}
    Also,  for every $c\geq0$, consider the piecewise linear function $f_c$ defined as
    \[
        f_c(x)\doteq\threepartdef{0}{0\leq x < c,}{\frac{1}{c}(x-c)}{c\leq x < 2c,}{1}{x\geq 2c,}
    \]
    which is non-negative and bounded, has weak derivative $f'_c=\frac{1}{c}\indicone{(c,2c)}$, and satisfies $f_c(0)=0$ and $\indicone{[2c,\infty)}\leq f_c\leq \indicone{[c,\infty)}$. By \eqref{nunbar_bound}, with $f$ replaced by $f_c$, and the definition of $\pi^{(N)}_*$ in \eqref{def_pis}, we have
    \begin{align*}
    \Ept{\nunbar_t(f_c)} & \leq   \int_0^\infty f_c(x)\overline G(x)dx+ C_Q\int_0^t (f'_c(x)\overline G(x) + f_c(x)g(x))dx\\
     & \leq (1+\frac{C_Q}{c})\int_c^\infty \overline G(x)dx + C_Q\int_c^\infty g(x)dx.
    \end{align*}
    The right-hand side of the equation above does not depend on $N$ or $t$, and converges to zero as $c\to\infty$ because $g$ and $\overline G$ are integrable by Assumption \ref{asm_service}.\ref{asm_g}. This implies that
    \begin{equation}\label{temp_nuntight2}
      \lim_{c\to\infty}\sup_{N\in\N,\;t\ge0}\Ept{\nunbar_t(\indicone{[2c,\infty)})}\leq \lim_{c\to\infty}\sup_{N\in\N,\;t\ge0}\Ept{\nunbar_t(f_c)} =0.
    \end{equation}
    Therefore, by \eqref{temp_nuntight1}-\eqref{temp_nuntight2}, the family $\{\nunbar_t;N\in\N,t\geq0\}$ satisfies the criteria in Proposition \ref{apx_Mtight}, and hence, is tight in $\emF$. Finally,  by Proposition \ref{prop_Vergodic}, $\nunbar_t$ converges in distribution to $\nunbar_\infty$ in $\emF$ as $t\to\infty$, and so tightness of the sequence $\{\nunbar_\infty\}_{N\in\N}$ follows from Remark \ref{remark_tight}.
\end{proof}

\begin{proposition}\label{prop_yinfbar}
Suppose Assumptions  \ref{asm_arrival}-\ref{asm_serviceP} hold. Then the sequence $(\xnbar_0,\nunbar_0) \doteq (\xnbar_\infty,\nunbar_\infty),N\in\N$, satisfies condition \eqref{asm_ic}.
\end{proposition}

\begin{proof}
    First, we claim that when $\vn_0 \deq \pin_*$, for every $f\in\mathbb{C}_b^1\hc$, there exists a constant $C_f<\infty$ such that
    \begin{equation}\label{temp_claim}
        \sup_{t\geq0,N\in\N} \frac{1}{\sqrt{N}} \Ept{|\nun_t(f)-N\overline\nu(f)|}\leq C_f.
    \end{equation}
  To prove the claim,   fix $f\in\mathbb{C}_b^1\hc$, $t\geq0$ and $N\in\N$. By the definition of $\overline \nu$ given in \eqref{def_nubar},
    \[\overline\nu(f)=\int_0^\infty f(s)\overline G(s)ds=\int_0^\infty (f\overline G)(s+t)ds+\int_0^t(f\overline G)(s)ds=\overline \nu(\Phi_tf)+\int_0^t(f\overline G)(t-s)ds.\]
   Substituting $\kn$ from \eqref{kNEQ} into the equation \eqref{nuNfEQ} for $\nun$, using integration by parts, the fact that $\xn_0=N$, $\overline{X} = 1$ and recalling the diffusion scaling form of \eqref{def_Ynhat},  we have
    \begin{align}\label{nuconf}
         \frac{1}{\sqrt{N}}\left(\nun_t(f)-N\overline \nu(f)\right)  &= \frac{1}{\sqrt{N}} \left(\nun_0(\Phi_tf)-N\overline\nu(\Phi_tf)\right) -\Hnhat_t(f) +\int_0^t (f\overline G)(t-s) d\enhat_s \notag\\
        &\hspace{5mm} -f(0)\xnphat_t -\int_0^t (f\overline G)'(t-s)\xnphat_sds,
    \end{align}
where    we used the fact that $\xnphat_0=0$ because $\xn_0 \stackrel{(d)}{=} X_*$. Since $\nun_0 \stackrel{(d)}{=} \nu_*$, with $\nu_*$ defined in \eqref{def_nuns}, $\nun_0(\Phi_tf)-N\overline\nu(\Phi_tf)\deq \sum_{j=1}^N\zeta^{N,t}_j$, where
    \[
       \zeta^{N,t}_j \doteq \delta_{a_j}(\Phi_t f)-\overline \nu(\Phi_tf)= f(a_j+t)\frac{\overline G(a_j+t)}{\overline G(a_j)}-\int_0^\infty f(x+t)\overline G(x+t)dx,
    \]
    and $\{a_j;j\geq1\}$ are i.i.d. with p.d.f $\overline G$. Hence, $\zeta^{N,t}_j;j=1,...,N,$ are also i.i.d. with $\Eptil{\zeta^{N,t}_1}=0$ and
    \[\text{Var}(\zeta^{N,t}_1)\leq \Ept{\delta_{a_1}(\Phi_tf)^2}=\int_0^\infty \frac{f^2(x+t)\overline G^2(x+t)}{\overline G^2(x)}\overline G(x) dx\leq \|f\|_\infty^2.\]
    Moreover, $\sum_{j=1}^N\zeta^{N,t}_j$ has zero mean and finite variance equal to $N\text{Var}(\zeta^{N,t}_1)$, and hence,
    \begin{align}\label{nuconf_temp1}
        \frac{1}{\sqrt{N}} \Ept{\left|\nun_0(\Phi_tf)-N\overline\nu(\Phi_tf)\right|}
        &\leq \frac{1}{\sqrt{N}} \left(\Ept{\big(\sum_{j=1}^N\zeta^{N,t}_j\big)^2}\right)^\half = \text{Var}(\zeta^{N,t}_1)^\half \leq \|f\|_\infty.
    \end{align}

    Next, by definition \eqref{def_Hn} of $\Hn_t$,  \eqref{CoVar_mn} and the bound \eqref{nunbar_int_bd}, we have
    \begin{align*}
        \Ept{\Hnhat_t(f)^2}&=\Ept{\int_0^t\nunbar_{t-s}\left( \frac{f^2(\cdot+s)\overline G^2(\cdot+s)}{\overline G^2(\cdot)}\frac{g(\cdot)}{\overline G(\cdot)}\right)\;ds}\leq \|f\|_\infty^2 R_M(0),
    \end{align*}
    with $R_M(0)<\infty$ by Lemma \ref{lem_nunbar}. Therefore,
    \begin{equation}\label{nuconf_temp2}
      \Ept{\left|\Hnhat_t(f)\right|}\leq\left(\Ept{\Hnhat_t(f)^2}\right)^\half \leq \|f\|_\infty \sqrt{R_M(0)}.
    \end{equation}
    Moreover, note that under Assumption \ref{asm_service}, for $f\in\mathbb{C}_b^1\hc$, $f\overline G$ is absolutely continuous with
    \begin{equation}\label{temp_vinf}
        f\overline G)'(s)
\leq (1+H)(\|f\|_\infty+\|f'\|_\infty) \overline{G}(s) \leq (1+H)(\|f\|_\infty+\|f'\|_\infty).
    \end{equation}
    Therefore, replacing $f$ with $f\overline G$ in \eqref{pre_part2} of Lemma \ref{lem_Eterm}, we have
    \begin{align*}
      \Ept{\left|\int_0^t(f\overline G)(t-s)d\enhat_s\right|^2}&\leq C_E \|f\|_\infty^2t\overline G(t)+C_E(1+H)^2(\|f\|_\infty+\|f'\|_\infty)^2 \left(\int_0^ts\overline G(s)ds\right)\left(\int_0^t\overline G(s)ds\right)\\
      &\hspace{5mm}+C_E\|f\|_\infty^2\left(\int_0^t\overline G(s)ds\right)^2.
    \end{align*}
    Since $\overline G(t)=\mathcal{O}(t^{-(3+\epsilon)})$ by Assumption \ref{asm_serviceP}.\ref{asm_moment3}, $t\overline G(t)$, $\int_0^t\overline G(s)ds$ and $\int_0^ts\overline G(s)ds$ are all uniformly bounded in $t$. Therefore, there exists a constant $C<\infty$ such that
    \begin{equation}\label{nuconf_temp3}
      \Ept{\left|\int_0^t(f\overline G)(t-s)d\enhat_s\right|} \leq \left(\Ept{\left|\int_0^t(f\overline G)(t-s)d\enhat_s\right|^2} \right)^\half
      \leq C (\|f\|_\infty+\|f'\|_\infty).
    \end{equation}

    Finally, using the bounds \eqref{xnphat_bd_eq} and \eqref{temp_vinf}, we have
    \begin{align}\label{nuconf_temp4}
        \Ept{\left|-f(0)\xnphat_t -\int_0^t (f\overline G)'(t-s)\xnphat_sds\right|}&\leq |f(0)|\Ept{\xnphat_t}+\int_0^t|(f\overline G)'|(t-s)\Ept{\xnphat_s}ds\notag\\
        &\leq C_Q|f(0)|+C_Q(1+H)(\|f\|_\infty+\|f'\|_\infty) \int_0^t\overline G(s)ds\notag\\
        &\leq 2C_Q(1+H)(\|f\|_\infty+\|f'\|_\infty).
    \end{align}
    The claim \eqref{temp_claim} follows from the equation \eqref{nuconf} and the bounds \eqref{nuconf_temp1}-\eqref{nuconf_temp2} and   \eqref{nuconf_temp3}-\eqref{nuconf_temp4}. In addition, since by Proposition \ref{prop_Vergodic} for every $N\to\infty$ $\nun_t\Rightarrow \nun_\infty$ in $\emF$ as $t\to\infty$, and for every $f\in\mathbb{C}_b^1\hc$, $\nun_t(f)\Rightarrow \nun_\infty(f)$, by the portmanteau theorem,
    \begin{equation}\label{temp_claim2}
      \sup_{N\in\N}\frac{1}{\sqrt{N}}\Ept{|\nun_\infty(f)-N\overline \nu(f)|}\leq \sup_{N\in\N}\liminf_{t\to\infty} \frac{1}{\sqrt{N}}\Ept{|\nun_t(f)-N\overline \nu(f)|}\leq C_f.
    \end{equation}

    Now, we return to the proof of the proposition. Since the limit $(\overline X,\overline \nu)$ is deterministic, to show that $(\xnbar_0,\nunbar_0)\Rightarrow (\overline X,\overline \nu)$, it is enough to show the convergence of each component. For  $N\in\N$,
    \begin{equation}\label{temp_xinf}
    \Ept{\left|\xnbar_\infty-1\right|}= \frac{1}{\sqrt{N}}\Ept{\left|\xnhat_\infty\right|}\leq \frac{1}{\sqrt{N}}\left(\Ept{\xnphat_\infty} +\Ept{\xnmhat_\infty}\right).
    \end{equation}
    By Proposition \ref{prop_Vergodic},
    $\vn_t\Rightarrow\vn_\infty$ in $\Vn$ as $t\to\infty$, and in particular, $\xn_t\Rightarrow\xn_\infty$. Hence, by Corollary \ref{xnphat_bound} and the portmanteau theorem,
    \begin{equation}\label{temp_xpinf}
        \Ept{\xnphat_\infty}\leq \liminf_{t\to\infty}\Ept{\xnphat_t}\leq C_Q.
    \end{equation}
    Moreover, recall that by the non-idling condition \eqref{number_in_service}, $\nun_t(\f1)=N\wedge\xn_t$, and  by another use of Proposition \ref{prop_Vergodic}, as $t\to\infty$,
    \[
        \nun_t(\f1)-N\wedge\xn_t\Rightarrow \nun_\infty(\f1)-N\wedge\xn_\infty.
    \]
    Therefore, since $\overline\nu(\f1)=\overline{X} = 1$,
    \[
        \xnmhat_\infty=-\frac{(\xn_\infty-N)}{\sqrt{N}} \wedge0 =-\frac{\xn_\infty\wedge N-N}{\sqrt{N}} =-\frac{\nun_\infty(\f1)-N\overline\nu(\f1)}{\sqrt{N}}.
    \]
    Hence, substituting $f=\f1$ in \eqref{temp_claim2} we obtain
    \begin{equation}\label{temp_xminf}
      \Ept{\xnmhat_\infty} \leq \frac{1}{\sqrt{N}}\Ept{|\nun_\infty(\f1)-N\overline \nu(\f1)|}\leq C_\f1.
    \end{equation}
    We conclude from \eqref{temp_xinf}-\eqref{temp_xminf} that
    \[
        \lim_{N\to\infty}\Ept{\left|\xnbar_\infty-1\right|}= \lim_{N\to\infty}\frac{1}{\sqrt{N}}\Ept{\left|\xnhat_\infty\right|} \leq \lim_{N\to\infty}\frac{1}{\sqrt{N}}(C_Q+C_\f1)=0.
    \]
    In other words,  as $N\to\infty$, $\xnbar_\infty$ converges to $\overline X\equiv1$ in $\mathbb{L}^1$, which in turn implies both $\xnbar_\infty \Rightarrow\overline X$ and $\Eptil{\xnbar_\infty}\to\overline X$.
    Finally, by another use of \eqref{temp_claim2}, for every $f\in\mathbb{C}_b^1\hc$,
    \[\lim_{N\to\infty}\Ept{\nunbar_\infty(f)-\overline \nu(f)} \leq \lim_{N\to\infty}\frac{1}{\sqrt{N}}C_f=0,\]
    which implies $\nunbar_\infty(f)\Rightarrow\overline \nu(f)$. The result then follows from the tightness of the sequence $\{\nunbar_\infty\}_{N\in\N}$ from Lemma \ref{lem_nunbarTight} and  Lemma \ref{random_measure_limit}.
\end{proof}

\subsection{Convergence of stationary distributions.} \label{sec_conv_proof}

We start with a preliminary result.

\begin{lemma}\label{lem_subseq}
Suppose Assumptions \ref{asm_arrival}-\ref{asm_serviceP} hold. If along some subsequence $\{N_j\}$, as $j\to\infty$, $\widehat\pi^{(N_j)}\Rightarrow \pi_0$ in $\Y$, then  $\pi_0$ is an invariant distribution of the semigroup $\{{\mathcal P}_t\}$ associated with the  diffusion model $Y$ defined in Section \ref{sec_spde}.
\end{lemma}

\begin{proof}
Set $\vn_0 = \vn_\infty$ for  $N\in\N$, let $\law (Y_0) = \pi_0$ and  let $Y = Y^{Y_0}=(X,Z)$ be the diffusion model with initial condition $Y_0$,  as specified in Definition \ref{def_solutionProcess}.  Since  $\ynhat_0=\widehat{\Psi}(\vn_0)$ has distribution $\pinhat=\law(\widehat{\Psi}(\vn_\infty))$ and $\pi^{(N_j)} \Rightarrow \pi_0$,  by continuity of the projection map on the Hilbert space
$\Hone \ho$ and an application of the continuous mapping theorem we have, as $j\to\infty$,
\begin{equation*}
    \left(\widehat X_0^{(N_j)},\langle \widehat Z^{(N_j)}_0,e_1\rangle_{\Hone},...,\langle \widehat Z^{(N_j)}_0,e_k\rangle_{\Hone}\right)\Rightarrow\left( X_0,\langle  Z_0,e_1\rangle_{\Hone} ,...,\langle  Z_0,e_k\rangle_{\Hone}\right),
\end{equation*}
    Next, since $\vn_0 = \vn_\infty$, by stationarity (see Proposition \ref{prop_Vergodic}) for every $t\geq0,$ $\vn_t\deq\vn_\infty$, and by measurability of the mapping $\widehat\Psi$ and the projection mappings on Hilbert spaces, we have
    \[\left(\xnhat_0,\langle \znhat_0,e_1\rangle_{\Hone},...,\langle \znhat_0,e_k\rangle_{\Hone}\right) \deq \left(\xnhat_t,\langle \znhat_t,e_1\rangle_{\Hone},...,\langle \znhat_t,e_k\rangle_{\Hone}\right).\]
    Moreover, since $(\xn_0,\nun_0) \deq(\xn_\infty,\nun_\infty)$, by Proposition \ref{prop_yinfbar} the sequence $\{(\xnbar_0,\nunbar_0)\}$ satisfies the condition \eqref{asm_ic}, and $\widehat{Y}^{(N_j)}_0 \Rightarrow Y_0$ as $j \rightarrow \infty$  by assumption. Hence,  the conditions of \eqref{thm_HKresult} hold and so,  by Proposition \ref{prop_subseq}, for any $t>0$ we have
    \begin{equation*}
        \left(\widehat X_t^{(N_j)},\langle \widehat Z^{(N_j)}_t,e_1\rangle_{\Hone},...,\langle \widehat Z^{(N_j)}_t,e_k\rangle_{\Hone}\right)\Rightarrow\left( X_t,\langle  Z_t,e_1\rangle_{\Hone} ,...,\langle  Z_t,e_k\rangle_{\Hone}\right)
    \end{equation*}
as $j\to\infty$.
    The last three displays imply that for every $t\geq0$,
    \[\left( X_0,\langle  Z_0,e_1\rangle_{\Hone} ,...,\langle  Z_0,e_k\rangle_{\Hone}\right)\deq\left( X_t,\langle  Z_t,e_1\rangle_{\Hone} ,...,\langle  Z_t,e_k\rangle_{\Hone}\right),\]
    and hence, the distributions of the projections of $Y_0$ and $Y_t$ onto a chain of increasing finite-dimensional subspaces $\mathbb{R} \otimes \mathbb{L}_k$, where $\mathbb{L}_k=\text{span}\{e_1, \ldots, e_k\}$ coincide.  Therefore, $Y_0 \deq Y_t$ (e.g., see the discussion in \cite[Chapter I.1]{Sko74}), and  $\pi_0=\law(Y_0)$ is an invariant distribution of the diffusion model.
\end{proof}

As a corollary of the previous lemma and Proposition \ref{thm_tightness}, we obtain the existence of an invariant distribution for the diffusion model $Y$.

\begin{corollary}\label{cor_exists}
    Under Assumptions \ref{asm_arrival}-\ref{asm_serviceP}, the Markov transition semigroup $\{{\mathcal P}_t\}$ associated with the  diffusion model $Y$ has a unique invariant distribution $\pi$.
\end{corollary}
\begin{proof}
    By Proposition \ref{thm_tightness}, the sequence $\{\pinhat\}$ is tight in $\Y$. Therefore, there exists a subsequence $\{N_j\}$ such that $\{\widehat\pi^{(N_j)}\}$ converges in distribution, as $j \rightarrow \infty$,  to a probability distribution $\pi$ on $\Y$.
 By Lemma \ref{lem_subseq}, this subsequential limit is an invariant distribution of the transition semigroup $\{{\mathcal P}_t\}$.  This  proves existence, and  uniqueness follows from Proposition \ref{thm_spde}.
\end{proof}

To conclude the proof of Theorem \ref{thm_interchange}, we use a standard argument by contradiction.

\begin{proof}[Proof of Theorem \ref{thm_interchange}]
Let $\pi$ be the unique invariant distribution of the transition semigroup $\{{\mathcal P}_t\}$ associated with the diffusion model.  Suppose the sequence $\{\pinhat\}$ does not converge to $\pi$.  Then, there exists a subsequence $\{N_j\}$ such that $\lim_{j\to\infty}\:d_p(\widehat\pi^{(N_j)},\pi)>0,$ where $d_p$ represents the Prohorov metric on the space of probability measures on $\Y$.By the  tightness of $\{\pinhat\}$ established in Proposition \ref{thm_tightness}, there exists a further subsequence, which we denote again by $\{N_j\}$, such that  $\widehat\pi^{(N_j)}\Rightarrow \pi_0$ in $\Y$, for some probability distribution $\pi_0$ on $\Y$. By Lemma \ref{lem_subseq}, $\pi_0$ is an invariant distribution of $\{{\mathcal P}_t\}$. But, since $\{{\mathcal P}_t\}$  has a unique invariant distribution $\pi$ by Proposition \ref{thm_spde}, it follows that $\pi_0=\pi$, and hence, that $\lim_{j\to\infty}\:d_p(\widehat\pi^{(N_j)},\pi) = 0,$ which is a contradiction.  Thus, we have shown that $\pinhat \Rightarrow \pi$.
\end{proof}

\appendix

\section{Proof of ergodicity of the process $\vn$.} \label{apx_vergodic}

\begin{proof}[Proof of Proposition \ref{prop_Vergodic}]
First note that by \eqref{def_lambdan}, the traffic intensity $\rho^{(N)}=\lambdan/N$ satisfies $\rho^{(N)}<1$.  As discussed in Section XII of \cite{Asm03}, the queueing system described in this paper can be viewed as an $N$ server queueing system where each server has its own queue, and each job, upon arrival, is routed to the queue with the least remaining workload (i.e., the residual service time of the job in service, plus the sum of service times of all jobs in that queue). The stability of this queueing system under the condition $\rho^{(N)}<1$ is studied in \cite{Asm03}, using a different Markovian representation and a different filtration. However,  some of their results are still relevant for our representation.

For the rest of the proof, we fix $N\in\N$, and suppress the superscript $(N)$ for brevity. Let $\tilde R$ be the backward recurrence time of the arrival process, that is,
\[
    \tilde R_t\doteq t-\sup\{s<t:E_s<E_t\};\quad\quad t\geq0.
\]
The process $\{\tilde V_t=(\tilde R_t,X_t,\nu_t);t\geq0\}$ is also a c\`adl\`ag Markov process. Let $\{\sigma(k);k\geq1\}$ be the sequence of jobs such that the system is empty right before their arrival time $T_{k}\doteq u_0+\dots+ u_{\sigma(k)-1}$, that is, $\tilde V_{T_k-}=(0,0,\mathbf{0})$, where $\mathbf{0}$ is the zero measure. Note that $\beta_-\doteq\sup\{x:G(x)=0\}$ is finite because $\lim_{x\to\infty} G(x)=1$, and $\alpha_+\doteq\sup\{x:\overline G_E(x)>0\}=\infty$ because $\tilde G_E$ (or equivalently, $G_E$) has full support on $\hc$  by Assumption \ref{asm_arrival}. Therefore, $\beta_-<\alpha_+$, and since $\rho<1$, we can apply Corollary XII.2.5 of \cite{Asm03} and conclude that $\{\sigma(k);k\geq1\}$ is an aperiodic, nonterminating (discrete-time) renewal process with finite mean renewal times. Moreover, $\tilde G_E$, and thus, $G_E$, are non-lattice by Assumption \ref{asm_arrival}, and another application of the fact that $\beta_-<\alpha_+$ established above implies $\Probil{v_1<u_1}>0$. Therefore, as shown in the proof of Corollary XII.2.8.(a) in \cite[p. 348]{Asm03}, we have $m\doteq\mathbb{E}_0\{T_1\}<\infty$, where $\mathbb{E}_0$ is the expectation under the initial condition $\tilde V_0=(0,0,\mathbf{0}).$ Therefore, $\tilde V$ is a c\`adl\`ag regenerative Markov process, with regeneration set $\{(0,0,\mathbf{0})\}$, regenerative points $T_k$, and cycle lengths $T_k-T_{k-1}$ with non-lattice distribution and finite mean $m$. Therefore, by \cite[Theorem VI.1.2]{Asm03}, there exists a random variable $V_\infty=(R_\infty,X_\infty,\nu_\infty)$ such that as $t\to\infty,$ $\tilde V_t\Rightarrow V_\infty$ in $\V$, and the law $\pi_\infty$ of $V_\infty$ is given by
\[
    \pi_\infty(f)=\frac{1}{m}\mathbb{E}_0\left\{\int_0^{T_1}f(\tilde V_s)ds\right\},\quad\quad \forall f\in\mathbb{C}_b(\mathcal V).
\]
Moreover, by \cite[Theorem VII.3.2]{Asm03}, $\pi_\infty$ is a stationary distribution for the Markov process  $\tilde V$. In particular, the law $\alpha_\infty$ of $R_\infty$ is the stationary distribution for the backward recurrence time of the arrival process, and thus has the representation $\alpha_\infty(dx)=1/\lambda\overline G_E(x)dx$ by \cite[Proposition V.3.5]{Asm03}.

Now recall from (3.2) in \cite[P. 151]{Asm03} that the distributions of $R_t$ and $\tilde R_t$ satisfy the relation
\[
    \Prob{R_t>y|\tilde R_t=x}=\frac{\overline G_E(x+y)}{\overline G_E(x)}.
\]
Hence, for every bounded and continuous function $f$,
\[
    \Ept{f(R_t)}=\int_0^\infty f(y)\Prob{R_t=dy}=\int_0^\infty\int_0^\infty f(y)\frac{g_E(x+y)}{\overline G_E(x)}dy\;\Prob{\tilde R_t=dx}=\Ept{\tilde f(\tilde R_t)},
\]
where $\tilde f(x)=\int_0^\infty g_E(x+y)f(y)dy/\overline G_E(x)$ is continuous and bounded by $\|f\|_\infty$.  Therefore, for every $f\in\mathbb{C}_b\hc$, recalling the expression for $\alpha_\infty$, we see that as $t\to\infty$,
\[
    \Ept{f(R_t)}\to\Ept{\tilde f(R_\infty)}=\frac{1}{\lambda}\int_0^\infty\int_0^\infty f(y)\frac{g_E(x+y)}{\overline G_E(x)}dy\overline G_E(x)dx=\frac{1}{\lambda}\int_0^\infty f(y)\overline G_E(y)dy=\Ept{f(R_\infty)}.
\]
Therefore, $R_t\Rightarrow R_\infty$, and the law $\alpha_\infty$ of $R_\infty$ is also a stationary distribution for the forward recurrence time  $R.$ We conclude that $\pi_\infty$ is a stationary distribution for $V$, and for every initial condition $V_0$, $V_t$ converges in distribution to $V_\infty$ as $t\to\infty.$
\end{proof}

\section{Properties of the mapping $\mathcal{T}$.}\label{apx_props}

\begin{proof}[Proof of Lemma \ref{lem_Tprops}]
For the first assertion, let $\mu=\sum_{j=1}^{m}\delta_{a_j}$, for some $m\in\N$ and $a_j\in\R_+$, $j=1,...,m$, and let $z=\mathcal{T}\mu$. Then we have
\[
    z(r)= \mu(\Phi_r\f1)=\sum_{j=1}^{m}\delta_{a_j}(\Phi_r\f1)=\sum_{j=1}^{m}\frac{\overline G(a_j+r)}{\overline G(a_j)},\quad\quad r\geq 0.
\]
Since $\overline G$ has derivative $g$, $z$ has derivative $z'$ given by
\[
    z^\prime(r) = -\sum_{j=1}^{m}\frac{g(a_j+r)}{\overline G(a_j)}=-\sum_{j=1}^m\delta_{a_j}(\Phi_rh)=-\mu(\Phi_rh),\quad\quad r\geq0.
\]
Both $\overline G$ and $g$ are in $\Ltwo\ho$ by Assumption \ref{asm_service}  (see Remark \ref{remark_HH2}), and hence,  both $z$ and $z^\prime$ also lie in $\Ltwo\ho$; thus $z\in\Hone\ho$.

For the second claim, let $\{\mu^k\}$ be a sequence in $\emD$ converging weakly to $\mu=\sum_{j=1}^{m}\delta_{a_j}$.
Since $\emD$ consists of sums of unit Dirac masses, for all sufficiently large $k$, $\mu^k$ must have a representation $\mu^k=\sum_{j=1}^{m}\delta_{a^k_j}$ for some $a^k_j>0$, with $a^k_j\to a_j$ as $k\to\infty$. Therefore, we have
\begin{align*}
    \|\mathcal{T}\mu^k-\mathcal{T}\mu\|_{\Ltwo} &= \left|\left| \sum_{j=1}^{m}\frac{\overline G(a^k_j+\cdot)}{\overline G(a^k_j)}-\frac{\overline G(a_j+\cdot)}{\overline G(a_j)}\right|\right|_{\Ltwo} \\
    &\leq \sum_{j=1}^{m}\left|\left| \frac{\overline G(a^k_j+\cdot)}{\overline G(a^k_j)}-\frac{\overline G(a_j+\cdot)}{\overline G(a^k_j)}+\frac{\overline G(a_j+\cdot)}{\overline G(a^k_j)}-\frac{\overline G(a_j+\cdot)}{\overline G(a_j)}\right|\right|_{\Ltwo} \\
    &\leq\sum_{j=1}^{m}\frac{\|\overline G(a^k_j+\cdot)-\overline G(a_j+\cdot)\|_{\Ltwo}}{\overline G(a^k_j)}+\|\overline G\|_{\Ltwo}\sum_{j=1}^{m}\left|\frac{1}{\overline G(a_j^k)}-\frac{1}{\overline G(a_j)}\right|,
\end{align*}
which converges to zero as $k\to\infty$ by the continuity of $\overline G$,  the continuity of the translation mapping in the $\Ltwo$ norm (see e.g. \cite[Theorem 20.1]{Dib02}) and the fact that $\overline{G} \in \mathbb{L}^2$. Similarly, we see that
\begin{align*}
    \|\left(\mathcal{T}\mu^k\right)^\prime-&\left(\mathcal{T}\mu\right)^\prime\|_{\Ltwo}
    \leq\sum_{j=1}^{m}\frac{\|g(a^k_j+\cdot)-g(a_j+\cdot)\|_{\Ltwo}}{\overline G(a^k_j)}+\|g\|_{\Ltwo}\sum_{j=1}^{m}\left|\frac{1}{\overline G(a_j^k)}-\frac{1}{\overline G(a_j)}\right|,
\end{align*}
which, since  $g\in\Ltwo\ho$, again  converges to zero as $k\to\infty$ by the same argument.

Finally, note that given the moment assumption, the function $\overline Z(r)=\int_r^\infty\overline G(x)dx$ lies in $\Ltwo\ho$, and has weak derivative $-\overline G(\cdot),$ which  also  lies in $\Ltwo\ho$ (see Remark \ref{remark_GL2}), thus showing that $\overline{Z} \in \Hone \ho$.  When combined with the first two assertions of this lemma and the fact that $\R$ and $\Hone \ho$ are vector spaces, this shows that $\Psi$ is a continuous map from $\cup_{N} \V^N$ to $\R \times \Hone \ho$.
\end{proof}

\section{A ramification of our assumptions.}\label{sec-gamgol}

In this section, we verify that the assumptions of Proposition \ref{prop-tempQ} (and Corollary \ref{xnphat_bound}) imply the assumptions H-W and $T_0$ of \cite{GamGol13}. Recalling the latter assumptions in our notation, H-W requires that $\Probil{v_j=0}=\Probil{\tilde u_n=0}=0$, and that $\Ept{v_j}=\Ept{\tilde u_n}=1$, where recall that $v_j,j\geq1$ are the service times, and $\tilde u_n,n\geq1$ are the inter-arrival times of the process $\tilde E$, and inter-arrival times of the arrival process $\en$ must satisfy $u^{(N)}_n=\tilde u_n/\lambdan$, with $\lambdan$ defined in \eqref{def_lambdan}.  This follows from our Assumptions \ref{asm_arrival} and \ref{asm_service}.\ref{asm_g}. Next, condition $T_0$ of \cite{GamGol13}, again in our notation, requires the following.
\renewcommand{\theenumi}{\roman{enumi}}
\begin{enumerate}
      \item There exists $\epsilon>0$ such that $\Eptil{\tilde u_1^{(2+\epsilon)}},\Eptil{v_1^{(2+\epsilon)}}<\infty.$
      \item $v_j$ has a non-zero variance.
      \item $\lim_{t\to 0}{t}^{-1}G(t)<\infty$.
      \item For every $N\in\N$, the number of jobs in system $\xn_t$ (which is denote by $Q^N(t)$ in \cite{GamGol13}), satisfies $\xn_t\Rightarrow \xn_\infty$ as $t\to\infty$.
\end{enumerate}
First note that (i) holds by Assumption \ref{asm_arrival} and the explicit moment assumption of Corollary \ref{xnphat_bound}. Conditions (ii) and (iii) hold because the service time distribution has a p.d.f. $g$ by Assumption \ref{asm_service}.\ref{asm_g} (see also the discussion below assumption $T_0$ on page 4 of \cite{GamGol13}), and (iv) follows from  Proposition \ref{prop_Vergodic}.

\section{Various tightness criteria.}\label{APXtight}

\subsection{Tightness and convergence of random measures.}
A criteria for tightness of a family of $\emF$-valued random variables is given in \cite[Exercise 4.11 on p.\ 39]{Kalbook}.

\begin{proposition}\label{apx_Mtight}
    A family $\{\mu_\alpha;\alpha\in\mathcal{A}\}$ of $\mathbb{M}_F[0,\infty)$-valued random variables is tight if and only if the following two conditions hold:
    \begin{enumerate}
    \renewcommand{\labelenumi}{\alph{enumi}. }
        \item $\sup_{\alpha\in\mathcal{A}} \Ept{|\mu_\alpha(\f1)|} < \infty$;
        \item $\lim_{c\to\infty} \sup_{\alpha\in\mathcal{A}} \Ept{|\mu_\alpha(\indicone{[c,\infty)})|} = 0$.
    \end{enumerate}
\end{proposition}

The following lemma summarizes the  standard approach to  identify the limit of a sequence of random measures in $\emF$.

\begin{lemma}\label{random_measure_limit}
    Suppose the sequence $\{\mu_n\}_{n\in\N}$ of random measures is tight in $\emF$. Assume there exists a deterministic measure $\mu\in\emF$ such that for every $f\in\mathbb{C}_b^1\hc$, $\mu_n(f)\Rightarrow\mu(f)$ in $\R$ as $n\to\infty$. Then, $\mu_n\Rightarrow\mu$ in $\emF$.
\end{lemma}

\begin{proof}
    Since $\{\mu_n\}_{n\in\N}$ is tight, every subsequence of $\{\mu_n\}$ has a further converging subsequence $\{\mu_{n_k}\}$. Let $\tilde\mu$ be any such sub-sequential limit. For every  $f\in\mathbb{C}_b^1\hc$, the mapping $\mu\mapsto\mu(f)$ from $\emF$ to $\R$ is continuous, and hence, by the continuous mapping theorem, $\mu_{n_k}(f)\Rightarrow\tilde\mu(f)$. On the other hand, $\mu_{n_k}(f)\Rightarrow\mu(f)$ by the assumption of the lemma. Hence, $\tilde\mu(f)\deq\mu(f)$, and since $\mu$ is deterministic, $\tilde\mu(f)=\mu(f)$, almost surely. As a consequence, for a countable dense subspace  $\mathcal{C}$ of $\mathbb{C}_b^1\hc$, almost surely, the equality
    \[\tilde\mu(f)=\mu(f),\quad\quad \forall f\in\mathcal{C},\]
    holds simultaneously.  This implies $\tilde \mu=\mu$, almost surely, which uniquely characterizes the sub-sequential limits of $\{\mu_n\}$, and completes the proof.
\end{proof}

\subsection{Tightness in $\Hone\ho$.}\label{apx_Wtight}
To obtain a criterion for the tightness of a set of probability measures or random elements in $\Hone\ho$, we first obtain a characterization of compact sets in $\Hone\ho$. We start with a version of the Sobolev Embedding Theorem.

\begin{proposition}\label{thm_sobemb} \cite[Theorem 6.2 on p. 144, Part II, equation (6)]{Adamsbook}
    Let $\Omega_0$ be a bounded subinterval of $\ho$. Then the embedding
    \[\Hone(\Omega_0)\mapsto\Ltwo(\Omega_0)\]
    is compact.
\end{proposition}

The following result helps to extend the last result to $\Ltwo(\Omega)$ when $\Omega=\ho$ is unbounded.

\begin{proposition}\label{thm_exttoinfty}\cite[Theorem 2.22 on p. 35]{Adamsbook}
    A subset $K\subset \Ltwo(\Omega)$ is pre-compact if there exist sub-domains $\Omega_j,j\in\N,$ of $\Omega$ such that
    \begin{enumerate}
      \item $\Omega_j \subset \Omega_{j+1}$ for $j\in\N$;
      \item $K_j=\{f|_{\Omega_j}; f \in K\}$ is pre-compact in $\Ltwo(\Omega_j)$ for $j\in\N$;
      \item for every $\epsilon>0$, there exists an index $j\in\N$ such that \[\int_{\Omega \backslash \Omega_j}|f(x)|^2dx < \epsilon,\quad \forall f\in K. \]
    \end{enumerate}
\end{proposition}

Propositions \ref{thm_sobemb} and \ref{thm_exttoinfty} provide criteria for the tightness of a sequence of random elements in $\Ltwo\ho$.

\begin{lemma}\label{apx_Wcompact}
    A subset $K\subset \Ltwo\ho$ is compact if
    \begin{enumerate}
    \renewcommand{\labelenumi}{\alph{enumi}. }
        \item $\sup_{f\in K} \;\|f\|_{\Hone(0,L)}<\infty,\quad \forall L\in \N,$
        \item $\lim_{L\to\infty} \;\sup_{f \in K}\; \|f\|_{\Ltwo(L,\infty)}=0.$
    \end{enumerate}
    Moreover, a family $\{\zeta_\alpha,\alpha\in\mathcal{A}\}$ of $\Hone\ho$-valued random elements defined on probability spaces $\left( \Omega_\alpha, \mathbb{P}_\alpha \right)$ is tight in $\Ltwo\ho$ if
    \begin{enumerate}
    \renewcommand{\labelenumi}{\alph{enumi}. }
        \item $\lim_{\lambda\uparrow\infty} \;\sup_{\alpha\in\mathcal{A}} \;\mathbb{P}_\alpha\left\{ \|\zeta_\alpha\|_{\Hone(0,L)} > \lambda \right\} = 0,\quad \forall L\in \mathbb{N},$
        \item $\lim_{L\to\infty} \;\sup_{\alpha\in\mathcal{A}} \;\mathbb{P}_\alpha\left\{ \|\zeta_\alpha\|_{\Ltwo(L,\infty)} >\epsilon \right\} = 0,\quad \forall \epsilon>0.$
    \end{enumerate}
\end{lemma}

\begin{proof}
The first claim is a direct sequence of Propositions \ref{thm_sobemb} and \ref{thm_exttoinfty}, with $\Omega_N=(0,N)$. For every $N\in \mathbb{N}$, $\Omega_N$ is bounded, and Proposition \ref{thm_sobemb} asserts that $K$ is compact in $\Ltwo(\Omega_N)$ because it is uniformly bounded in the $\Hone(\Omega_N)$ norm. The second claim  follows from the first claim by exactly the same argument as in \cite[Theorem 4.10]{KarShr91}.
\end{proof}

\begin{proof}[Proof of Proposition \ref{prop_Htightness}]
    The result is a consequence of the second claim in Lemma \ref{apx_Wcompact}, the fact that $\zeta_\alpha$ is tight in $\Hone\ho$ if $\zeta_\alpha$ and $\zeta^\prime_\alpha$ are tight in $\Ltwo\ho$, the definition of the $\Hone\ho$ norm and the elementary inequality $\mathbb{P}(A+B>2\lambda)\leq\mathbb{P}(A>\lambda)+\mathbb{P}(B+\lambda)$.
\end{proof}

\section{Verification of Assumptions for Certain Families of Distributions} \label{apver}

In this section, we show  that a large class of distributions of interest satisfy our assumptions.

\begin{lemma}\label{lem_verify}
  Assumptions \ref{asm_service} and \ref{asm_serviceP} are satisfied when $G$ belongs to one of the following families of distributions:
\renewcommand{\theenumi}{\arabic{enumi}}
\begin{enumerate}
    \item Generalized Pareto distributions with location parameter $\mu=0$ (a.k.a.\ Lomax distribution), and
 shape parameter $\alpha>3$.
\label{verify_pareto}
    \item  \label{verify_lognormal}
The log-normal distribution with location parameter $\mu \in (-\infty, \infty)$ and scale parameter
$\sigma > 0$.
    \item The Gamma distribution with shape parameter $\alpha\geq3$.
\label{verify_gamma}
    \item Phase-type distributions.\label{verify_phasetype}
  \end{enumerate}
\end{lemma}

\begin{proof}
{\em Family \ref{verify_pareto}.} The Lomax distribution (equivalently, the generalized Pareto distribution with location parameter $\mu=0$) with  scale parameter $\lambda>0$ and shape parameter $\alpha>0$ has the following complementary c.d.f.:
\[
    \overline G(x) = \left(1+\frac{x}{\lambda}\right)^{-\alpha},\quad\quad x\geq0.
\]
Elementary calculations show that for $\alpha > 1$, the distribution has a finite mean, which is equal to $\lambda/(\alpha-1)$. In particular, the distribution has mean $1$ when $\lambda = \alpha - 1$. The probability density function (p.d.f.) $g$ of $G$ clearly exists, is continuously differentiable and satisfies
\[
    g(x)=\frac{\alpha}{\lambda}\left(1+\frac{x}{\lambda}\right)^{-(\alpha+1)},\quad \mbox{ and } \quad
    g^\prime(x) = -\frac{\alpha(\alpha+1)}{\lambda^2}\left(1+\frac{x}{\lambda}\right)^{-(\alpha+2)},
\]
for $x \in (0,\infty)$. Thus, the hazard rate function $h$ is equal to
\[
    h(x)=\frac{g(x)}{\overline G(x)}=\frac{\alpha}{\lambda}\left(1+\frac{x}{\lambda}\right)^{-1},\quad\quad x\geq0,
\]
which is uniformly bounded by $\alpha/\lambda,$ and
\[
    h_2(x)=\frac{g'(x)}{\overline G(x)}=-\frac{\alpha(\alpha+1)}{\lambda^2}\left(1+\frac{x}{\lambda}\right)^{-2},\quad\quad x\geq0,
\]
which shows that $|h_2|$ is uniformly bounded by $\alpha^2/\lambda^2$. Therefore, Assumption \ref{asm_service} is satisfied. Moreover, as $x\to\infty,$
\[
    \overline G(x) = x^{-\alpha}\left(\frac{1}{x}+\frac{1}{\lambda}\right)^{-\alpha}=\mathcal{O}(x^{-\alpha}),
\]
and hence, Assumption \ref{asm_serviceP}.\ref{asm_moment3} holds when $\alpha>3$. Finally, $g'$ is differentiable with derivative
\[
    g''(x)= \frac{\alpha(\alpha+1)(\alpha+2)}{\lambda^3}\left(1+\frac{x}{\lambda}\right)^{-(\alpha+3)},\quad\quad x\geq0,
\]
which is bounded and satisfies $g''(x)=\mathcal{O}(x^{-(3+\alpha)})$ as $x\to\infty$. Therefore, Assumption \ref{asm_serviceP}.\ref{asm_gpp} holds for every $\alpha>0$.

{\em Family \ref{verify_lognormal}.} The complementary c.d.f.\ of the log-normal distribution with location parameter $\mu\in R$ and shape parameter $\sigma>0$ has the form
\[
    \overline G_{\mu, \sigma}(x)=\frac{1}{2}-\frac{1}{2} \text{ erf} \left(\frac{\log{x}-\mu}{\sqrt{2}\sigma}\right),\quad\quad x\geq 0,
\]
where $\text{erf }(y)=\frac{2}{\sqrt{\pi}}\int_0^ye^{-t^2}dt$ is the error function.  Simple calculations show that the mean is given by $e^{\mu+\sigma^2/2}$, which is equal to $1$ when $\mu=-\sigma^2/2$.   Fix  $\sigma > 0$.  For any $\mu \in \mathbb{R}$, $\overline G_{\mu,\sigma} (x)= \overline G_{0,\sigma}(cx)$, with  $c \doteq e^{-\mu}$,  for all $x\geq0$.
Therefore,  it suffices to verify  Assumptions \ref{asm_service}.\ref{asm_h} and \ref{asm_serviceP} for $\overline G\doteq \overline G_{0,\sigma}$. The p.d.f.\ $g$ of $G = G_{0,\sigma}$ exists, is continuous, and is given explicitly by
\[
    g(x)=\frac{1}{x\sqrt{2\pi}\sigma}e^{-\frac{(\log x)^2}{2\sigma^2}}= \frac{1}{x\sigma}\phi\left(\frac{\log x}{\sigma}\right),\quad\quad x>0,
\]
where $\phi$ is the p.d.f.\  of the standard Gaussian distribution. The p.d.f.\ $g$ itself is continuously differentiable with derivative
\[
    g^\prime(x)=-\frac{\log x+\sigma^2}{x^2\sigma^3}\phi\left(\frac{\log x}{\sigma}\right), \quad\quad x>0.
\]

The complementary c.d.f.\ $\overline G$ can be written as
\[
    \overline G(x) = Q\left(\frac{\log x}{\sigma}\right),\quad\quad x\geq0,
\]
where $Q$ is the function $Q(z)=1/2+1/2\text{ erf}(z/\sqrt{2})$, which satisfies the bounds \cite[equation (8)]{BorSun79}
\begin{equation}\label{temp_Qfunction}
    Q(z) \geq \frac{z}{1+z^2}\phi(z),\quad\quad z\geq 0.
\end{equation}
For $x \geq 0$,  set $z_x\doteq\log(x)/{\sigma}$. Then, using the bound \eqref{temp_Qfunction}, for $x \geq  e^{\sigma}$ (and hence $z_x\geq1$) we have
\[
    h(x) =\frac{\phi(z_x)}{x\sigma Q(z_x)} \leq  \frac{(1+z_x^2)}{\sigma z_xe^{\sigma z_x}}  \leq \frac{(1+z_x^2)}{\sigma} e^{-\sigma z_x}.
\]
Moreover, for $x\in[e^\sigma,\infty)$, since $\log x<x$,
\[
    \frac{g'(x)}{g(x)}=\frac{\log x+\sigma^2}{x\sigma^2}\leq \frac{1}{\sigma^2}+e^{-\sigma},
\]
and hence, $h_2=(g'/g)h$ is also bounded on $[e^\sigma,\infty)$. Moreover,
\[
    \lim_{x\to 0}g(x)=\lim_{z\to -\infty}\frac{1}{\sqrt{2\pi}\sigma}e^{-\sigma z-\frac{z^2}{2}} =0,
\]
and
\[
    \lim_{x\to 0}g^\prime(x)=\lim_{z\to-\infty}-\frac{z + \sigma}{\sqrt{2\pi}\sigma^2}   e^{-2 \sigma z-\frac{z^2}{2}}=0.
\]
Since $g$ and $g'$ are continuous and $\overline G$ is decreasing, it follows from the last two displays that $h$ and $h_2$ are also bounded on $(0,e^\sigma)$. Therefore, Assumption \ref{asm_service}.\ref{asm_h} holds.

Moreover, it is straightforward to see that for a random variable $X$ with log-normal distribution,
\[
    \Ept{X^n}=e^{n\mu+\frac{n^2\sigma^2}{2}}<\infty.
\]
In other words,  all moments of the log-normal distribution exist, and in particular, Assumption \ref{asm_serviceP}.\ref{asm_moment3} holds.

Finally, $g'$ is continuously differentiable with derivative
\[
g''(x)=\frac{(\log x)^2+3\sigma^2\log x+2\sigma^4-\sigma^2}{x^3\sigma^5}\phi(\frac{\log x}{\sigma}).
\]
Since $g''$ is continuous and satisfies
\[
    \lim_{x\to 0,\infty}g''(x)=\lim_{z\to\pm\infty}-\frac{z^2+3\sigma z+2\sigma^2-1}{\sqrt{2\pi}\sigma^3}  e^{-3 \sigma z-\frac{z^2}{2}}=0,
\]
it is also bounded. Moreover, for every $\epsilon=1$
\[
  \lim_{x\to\infty}\frac{g''(x)}{x^{-2-\epsilon}}=\lim_{z\to\infty}-\frac{z^2+3\sigma z+2\sigma^2-1}{\sqrt{2\pi}\sigma^3}  e^{-\frac{z^2}{2}}=0,
\]
and therefore, $g''(x)=\mathcal{O}(x^{-(2+\epsilon)})$ as $x\to\infty$. This shows that Assumption \ref{asm_serviceP}.\ref{asm_gpp} holds as well.

{\em Family \ref{verify_gamma}.}  The complementary c.d.f.\ of a Gamma distribution with  shape parameter $\alpha>0$ and rate parameter $\beta>0$ is given by
    \[\overline G(x)=\frac{1}{\Gamma(\alpha)}\Gamma(\alpha,\beta x),\quad\quad\quad x>0,\]
where $\Gamma(\cdot,\cdot)$ is the upper incomplete Gamma function. The mean is $\alpha/\beta$, which is equal to one when $\alpha=\beta$. The p.d.f.\  $g$ is equal to
\[
    g(x)=\frac{\beta^\alpha}{\Gamma(\alpha)}x^{\alpha-1}e^{-\beta x},\quad\quad x>0,
\]
which is itself continuously differentiable on $\ho$, with derivative
\[
    g^\prime(x)=\frac{\beta^\alpha}{\Gamma(\alpha)}(\alpha-1-\beta x)x^{\alpha-2}e^{-\beta x},\quad\quad x>0.
\]
When $\alpha \geq 2$,
\[
    \lim_{x\to 0}h(x)=\lim_{x\to 0}g(x)=\frac{\beta^\alpha}{\Gamma(\alpha)}\lim_{x\to 0} x^{\alpha-1} = 0.
\]
Also, by L'H\^{o}pital's rule,
    \begin{align*}
        \lim_{x\to\infty} h(x) =   \lim_{x\to\infty}\frac{\beta^\alpha x^{\alpha-1}e^{-\beta x}}{\int_{\beta x}^\infty t^{\alpha-1}e^{-t}dt}=   \lim_{x\to\infty}  \beta^\alpha\frac{(\alpha-1)x^{\alpha-2}e^{-\beta x} -\beta x^{\alpha-1}e^{-\beta x} }{-\beta^\alpha x^{\alpha-1}e^{-\beta x}} = \beta.
    \end{align*}
Moreover, again when $\alpha\geq2$,
\[
    \lim_{x\to 0}h_2(x)=\lim_{x\to 0}g^\prime(x)=\frac{\beta^\alpha}{\Gamma(\alpha)}(\alpha-1)\lim_{x\to0}x^{\alpha-2}<\infty
\]
and by L'H\^{o}pital's rule,
\begin{align*}
    \lim_{x\to\infty} h_2(x) &=   \lim_{x\to\infty} \frac{\beta^\alpha (\alpha-1-\beta x)x^{\alpha-2}e^{-\beta x}}{\int_{\beta x}^\infty t^{\alpha-1}e^{-t}dt} \\
    &=   \lim_{x\to\infty} \frac{\beta^\alpha \left( \beta^2x^{\alpha-1}-2\beta(\alpha-1)x^{\alpha-2} + (\alpha-1)(\alpha-2)x^{\alpha-3}\right) e^{-\beta x} }{-\beta^\alpha x^{\alpha-1}e^{-\beta x}}\\
     &=  - \beta^2.
    \end{align*}
Since $h$ and $h_2$ are continuous on $\ho$, the last four displays show that Assumption \ref{asm_service} holds for $\alpha\geq 2$.

Moreover, the Gamma distribution has finite exponential moments in a neighborhood of the origin, and in particular, Assumption \ref{asm_serviceP}.\ref{asm_moment3} holds.

Finally, $g'$ is continuously differentiable on $\ho$ with derivative
\[
    g''(x)=\frac{\beta^\alpha}{\Gamma(\alpha)}\big((\alpha-1)(\alpha-\beta-2)+x\beta(\beta-\alpha+3)\big)x^{\alpha-3}e^{-\beta x},\quad\quad x>0.
\]
When $\alpha\geq3$, $g''$ satisfies
\[
    \lim_{x\to 0} g''(x)= \frac{\beta^\alpha}{\Gamma(\alpha)}\big((\alpha-1)(\alpha-\beta-2)\big) \lim_{x\to 0}x^{\alpha-3}<\infty.
\]
Moreover, for every $\epsilon>0$,
\[
    \lim_{x\to\infty}\frac{g''(x)}{x^{-2-\epsilon}}= 0.
\]
Since $g''$ is continuous on $\ho$, the last two displays show that when $\alpha\geq 3$, $g''$ is bounded and satisfies $g''(x)=\mathcal{O}(x^{-(2+\epsilon)})$. Therefore Assumption \ref{asm_serviceP}.\ref{asm_gpp} holds.

{\em Family  \ref{verify_phasetype}.}   A phase-type distribution with size $m$, an $m\times m$-subgenerator matrix $\bs$ (which has eigenvalues with negative real part) and probability row vector  $\ba$, the complementary c.d.f.\ function has the representation
\[
    \overline G(x)= \sum_{j=1}^mp_j(x)=\ba e^{x\bs}\f1,\quad\quad x\geq 0,
\]
where $\f1$ is an $m\times 1$ column vector of ones and
\[
    \bp(x)= [p_1(x),...,p_m(x)] \doteq \ba e^{x\bs},\quad\quad\quad x\geq 0
.\]
Note that the vector $\ba$ can be chosen such that the mean $-\ba \bs^{-1}\f1$ is set to one. Defining $\bmu=[\mu_1,...,\mu_m]\doteq-\bs\f1,$ the probability density function $g$ can be written as
\[
    g(x) = -\ba e^{x\bs}\bs\f1 = \ba e^{x\bs}\bmu =\sum_{j=1}^mp_j(x)\mu_j,
\]
and is continuous on $\ho$. Therefore, Assumption \ref{asm_service}.\ref{asm_g} holds. The p.d.f.\ $g$ is continuously differentiable with derivative
\[
    g^\prime(x) = -\ba e^{x\bs}\bs^2\f1 = \ba e^{x\bs} \bnu,\quad\quad x>0,
\]
where $\bnu=[\nu_1,...,\nu_m]^T \doteq -\bs^2\f1.$ The hazard rate function $h$ satisfies
\[
    h(x)= \frac{\sum_{j=1}^mp_j(x)\mu_j}{\sum_{j=1}^mp_j(x)}\leq  \max_{j=1,...,m}\mu_j<\infty,\quad\quad x\geq0.
\]
Hence, $h$ is uniformly bounded on $\hc$. Moreover,
\[
    |h_2(x)|= \frac{\left|\sum_{j=1}^mp_j(x)\nu_j\right|}{\sum_{j=1}^mp_j(x)}\leq  \max_{j=1,...,m}|\nu_j|<\infty,\quad\quad x>0.
\]
Therefore,  Assumption \ref{asm_service}.\ref{asm_h} holds.

Moreover, for a random variable $X$ with a phase-type distribution,
\[
    \Ept{X^{n}}=(-1)^{n}n!\ba{S}^{-n}\f1<\infty, \quad n \in \mathbb{N}.
\]
Therefore, all moments are finite and in particular, Assumption \ref{asm_serviceP}.\ref{asm_moment3} holds.

Finally, $g'$ is continuously differentiable with derivative
\[
    g''(x) = -\ba e^{x\bs}\bs^3\f1 ,\quad\quad x>0,
\]
The function $g''$ satisfies
\[
    \lim_{x\to 0} g''(x)= -\ba\bs^3\f1,
\]
and since all eigenvalues of $\bs$ have negative real parts, $g''(x)$ decays exponentially when $x\to\infty$. Therefore, Assumption \ref{asm_serviceP}.\ref{asm_gpp} holds.
\end{proof}

\section*{Acknowledgments.}
RA was partially supported by NSF grants CMMI-1234100 and DMS-1407504 and the Charles Lee Powell Foundation; and KR was partially supported by the grant AFOSR FA9550-12-1-0399.

\bibliographystyle{plain} 
\bibliography{reference}

\end{document}